\documentclass{amsart}
\usepackage{amsmath,amsthm}
\usepackage {latexsym}
\usepackage{amssymb}

\newcommand{\less}{\lesssim}

\newcommand{\beal}{\begin{align}}
\newcommand{\enal}{\end{align}}
\newcommand{\bealn}{\begin{align*}}
\newcommand{\enaln}{\end{align*}}
\newcommand{\bear}{\begin{eqnarray}}
\newcommand{\eear}{\end{eqnarray}}
\newcommand{\beeq}{\begin{equation}}
\newcommand{\eneq}{\end{equation}}

\newcommand{\spec}{{\rm spec}}

\newcommand{\eps}{{\varepsilon}}
\newcommand{\R}{{\mathbb R}}
\newcommand{\Compl}{{\mathbb C}}

\newcommand{\Z}{{\mathbb Z}}
\newcommand{\Nat}{{\mathbb N}}
\newcommand{\calC}{{\mathcal C}}

\newcommand{\calB}{{\mathcal B}}

\newcommand{\calF}{{\mathcal F}}
\newcommand{\calJ}{{\mathcal J}}
\newcommand{\calL}{{\mathcal L}}
\newcommand{\calQ}{{\mathcal Q}}
\newcommand{\calS}{{\mathcal S}}
\newcommand{\calD}{{\mathcal D}}
\newcommand{\calK}{{\mathcal K}}

\newcommand{\calN}{{\mathcal N}}

\newcommand{\tileps}{{\tilde{\eps}}}

\newcommand{\tilf}{{\tilde{f}}}

\newcommand{\la}{\langle}
\newcommand{\ra}{\rangle}

\renewcommand{\ln}{\log}

\newcommand{\IS}{IS}

\def\pr{\partial}

\def\nn{\nonumber}
\def\bm{\left[ \begin{array}{cc}}
\def\endm{\end{array}\right]}

\def\calR{{\mathcal R}}

\def\calC{{\mathcal C}}
\def\calB{{\mathcal B}}

\def\eps{\varepsilon}
\def\pr{\partial}
\def\calL{{\mathcal L}}
\def\calD{{\mathcal B}}
\def\calF{{\mathcal F}}
\def\calK{{\mathcal K}}
\def\calS{{\mathcal S}}
\def\Id{{\rm Id}}
\def\bm{\left[\begin{matrix} }
\def\endm{\end{matrix}\right]}
\def\la{\langle}
\def\ra{\rangle}

\def\nn{\nonumber}
\def\les{\lesssim}

\def\R{{\mathbb R}}

\newtheorem{theorem}{Theorem}
\newtheorem{lemma}[theorem]{Lemma}
\newtheorem{defi}[theorem]{Definition}

\newtheorem{proposition}[theorem]{Proposition}

\theoremstyle{remark}
\newtheorem{remark}[theorem]{Remark}

\def\calR{{\mathcal R}}

\def\half{\frac12}

\def\cQ{{\mathcal Q}}

\def\dom{{\rm Dom}}

\def\cL{\mathcal L}
\def\tilphi{\phi}
\def\tilm{m}

\def\tiltheta{\theta}

\renewcommand{\Im}{\,{\rm Im}\,}
\renewcommand{\Re}{\,{\rm Re}\,}

\def\Bla{\Big\langle}
\def\Bra{\Big\rangle}

\newcommand{\go}{\phi_0}
\def\cN{{\mathcal N}}
\renewcommand{\hat}{\widehat}
\renewcommand{\epsilon}{\eps}
\renewcommand{\tilde}{\widetilde}
\numberwithin{equation}{section}
\numberwithin{theorem}{section}

\begin{document}

\title{Renormalization and blow up for the critical Yang-Mills problem.}

\author{J.\ Krieger}
\address{The University of Pennsylvania, Department of Mathematics,
David Rittenhouse Lab, 209 South 33rd Street, Philadelphia, PA 19118, U.S.A.}
\email{kriegerj@math.upenn.edu}

\author{W.\ Schlag}
\address{Department of Mathematics, The University of Chicago, 5734
South University Avenue, Chicago, IL 60637, U.S.A.}
  \email{schlag@math.uchicago.edu}

\author{D.\ Tataru}
\address{Department of Mathematics, The University of California at Berkeley, Evans Hall, Berkeley, CA 94720, U.S.A.}
\email{tataru@math.berkeley.edu}

\thanks{The authors were partially supported by the National Science Foundation, J.\ K.\ by DMS-757278 and a Sloan Fellowship, W.\ S.\ by
DMS-0617854, D.~T.\ by DMS-0354539, and DMS-0301122. J.\ K.  gratefully acknowledges the hospitality of the University of California, Berkeley, as well as the University of Chicago.}



\maketitle

\section{Introduction}
\label{sec:intro}

We describe singularity formation for the semi-linear wave equation
\begin{equation}\label{eq:wave}
\Box u -\frac{2}{r^2}u(1-u^2) =0,\qquad \Box=\partial_{tt} -\Delta
\end{equation}
in $\R^{2+1}$. This equation arises as follows: consider Yang-Mills
fields in $(d+1)$-dimensional Minkowski spacetime. The gauge
potential $A_\alpha$ is a one-form with values in the Lie algebra
$\frak g$ of a compact Lie group~$G$. In terms of the curvature
$F_{\alpha\beta}=\partial_\alpha A_\beta-\partial_\beta A_\alpha +
[A_\alpha, A_\beta]$ the Yang Mills equations take the form
\[
\partial_\alpha F^{\alpha\beta}  + [A_\alpha, F^{\alpha\beta}] =0,
\]
where $[\cdot,\cdot]$ is the Lie bracket on $G$. We take $G=SO(d)$
with $\frak g$ being the skew-symmetric $d\times d$ matrices. In
particular $A_\alpha = \{A^{ij}_{\alpha}\}_{i,j=1}^d$. Assuming the
spherically symmetric ansatz (see~\cite{SST} and~\cite{GS} for
analogous considerations in the context of the Yang-Mills heat flow)
\[
A^{ij}_\mu (x) = (\delta^i_\mu x^j - \delta^j_\mu x^i)
\frac{1-u(t,r)}{r^2}
\]
the Yang Mills  equations reduce to the semi-linear wave equation
\[
\Box_{d-2} u = \partial_{tt} u - \Delta_{d-2} u = \frac{d-2}{r^2} u
(1-u^2)
\]
This equation is invariant under the scaling $u(r,t)\mapsto
u(r/\lambda,t/\lambda)$. With respect to this scaling the energy
\[
E = \int_0^\infty \Big[ u_t^2 + u_r^2 + \frac{d-2}{2r^2}
(1-u^2)^2\Big]r^{d-3}\, dr
\]
is invariant iff $d=4$ which is the case we consider in this paper.
Equation~\eqref{eq:wave} admits the stationary solution
\[
Q(r)=\frac{1-r^2}{1+r^2},
\]
called {\em instanton}. In $3+1$ dimensions, the Yang-Mills
equations are subcritical relative to the energy. Eardley and
Moncrief~\cite{EM1}, \cite{EM2} showed that in that case there are
global smooth solutions. See also Klainerman, Machedon~\cite{KM} who
lowered the regularity assumptions on the data. In the energy
critical case of $4+1$ dimensions, local well-posedness in $H^s$
with $s>1$ was shown by  Klainerman, Tataru~\cite{KT}. However, it
was conjectured that global wellposedness fails and that
singularities should form, see Bizon, Tabor~\cite{BT} and Bizon,
Ovchinnikov, Sigal~\cite{BOS} for numerical and heuristic arguments
to that effect. However, such a phenomenon had not been observed
rigorously. In this paper we show how to construct a
 solution to the wave equation~\eqref{eq:wave} as a perturbation of a
 time-dependent profile
\begin{equation}\nn
u_0 = Q(R), \qquad R =r
\lambda(t),\quad \Phi(R)=\frac{R^2}{(1+R^2)^2}
\end{equation}
with $\lambda(t)$ a logarithmic correction to the self-similar ansatz
\[
\lambda(t)=t^{-1}(-\log t)^{\beta},\,\beta>0
\]
In other words, we prove that in general the energy critical
Yang-Mills equations develop singularities in finite time.

 As in our earlier work \cite{KST1} for energy
critical wave maps, and~\cite{KST2} for the energy critical
semi-linear wave equation in ~$\R^3$, the blow up rate is
prescribed. Since a continuum of rates is admissible, the blow up
solutions which we construct can of course not be stable. In
contrast to the rates $\lambda(t)=t^{-1-\nu}$ which appeared in
\cite{KST1} and~\cite{KST2}, in the case of Yang-Mills we only make
logarithmic corrections to the self-similar rate. This has to do
with the fact that the linearized Yang Mills operator  has a zero
energy eigenvalue in $4+1$ dimensions, whereas for wavemaps as well
as the three-dimensional semi-linear focusing wave equation, it
exhibits  a zero energy resonance. This difference is very important
as in the case of an eigenvalue an orthogonality condition appears
which is not present in the zero energy resonant case. It is this
fact which required major changes to our scheme, especially to the
``renormalization part'' in which we construct approximate
solutions. In addition, in contrast to our earlier work on wave
maps~\cite{KST1}, the approximate solutions here are much rougher,
and indeed asymptotically only lie in $H^{1}$, the threshold for
local well-posedness of the critical Yang-Mills equation. The reason
for this is the much more singular nature of the ODE's arising in
the renormalization step, due to the different blow up rate.

\begin{theorem}  Let
\[
\lambda(t):=t^{-1}(-\log t)^{\beta}
\]
For each $\beta>\frac{3}{2}$ there exists a spherically symmetric
solution $u$ to \eqref{eq:wave} inside the cone $\{r < t, t < t_0\}
$ which has the form
\[
u(x,t) = Q(r \lambda(t)) + v(x,t)
\]
where the function $v$ has the size and regularity, with
$S:=t\pr_t+r\pr_r$,
\[
 \|\nabla v\|_{L^2}+ \|\nabla S v\|_{L^2} + \| \nabla S^2 v\|_{L^2}
\lesssim |\log t|^{-1}
\]
as well as the pointwise decay
\[
|v(t,x)| \lesssim  |\log t|^{-1}
\]
\label{tmain}\end{theorem}

We emphasize that our solutions are just barely better than $H^1$,
in contrast to our earlier work on wavemaps. While $H^1$ local
wellposedness is not known for the general Yang-Mills problem, it is
known in the  equivariant case, see~\cite{ST}. This is important for
our purposes, as it shows that our solutions belong to a class for
which a local wellposedness theory is available. In addition, the
vector field $S$ is required to control the strong singularity in
the nonlinearity  $r^{-2}u^3$ at $r=0$; this is in the spirit of the
method of invariant vector fields in nonlinear wave equations which
allows for improved decay away from the characteristic light-cone
$\{|x|=t\}$. More precisely, one can use elliptic estimates close to
$r=0$ to control the aforementioned singularity.

 \section{The proof of the main theorem}
 \label{sec:mainproof}

 This section contains the proof of Theorem~\ref{tmain}.  The first
 step is construct an arbitrarily good approximate solution to the
 wave equation~\eqref{eq:wave} as a perturbation of a time-dependent
 profile $u_0 = Q(R)$. The result is as follows:

\begin{theorem}
  For each integer $N$ there exists a spherically symmetric
  approximate solution $u_N$ to \eqref{eq:wave} inside the cone $\{r <
  t, t < t_0\}$ which has the form 
\[
u_N(r,t) = Q(R) + v_{10}(r,t) + v_N(r,t),\quad  v_{10}(r,t) := (t\lambda(t))^{-2}
\frac{R^4}{4(1+R^2)^2}
\]
with $\lambda(t)=t^{-1}|\log t|^\beta$, $R=r\lambda(t)$,  and $v_N$ satisfying the pointwise bounds
\[
|v_N(r,t)|+|Sv_N(r,t)| + |S^2 v_N(r,t)| \lesssim \frac{r^2}{t^2 |\log t|}=
\frac{R^2}{(t\lambda(t))^2 |\log t|}
\]
and so that the corresponding error
\[
e_N = \Box u_N -\frac{2}{r^2}u_N(1-u_N^2)
\]
satisfies
\[
|e_N(r,t)|+|Se_N(r,t)| + |S^2 e_N(r,t)| \lesssim  t^{-2} |\log t|^{-N}
\]
\end{theorem}

The proof of the above theorem is carried out in
Section~\ref{sec:mod}. The description of the approximate solutions
$u_N$ obtained there is much more precise than what is stated above.
In particular, the functions $u_N$ are analytic up to the cone $t =
r$, and the nature of the singularity at the cone is clearly
explained.

Given the approximate solutions $u_N$ constructed above,
we look for a solution $u$ to \eqref{eq:wave} of the form
 \[
 u(t,r)=u_{N}(t,r)+\epsilon(t,r),
 \]
 where $\epsilon$ is to be determined via Banach iteration. The
 equation for $ \epsilon$ is
\[
\Big(-\partial_{t}^{2}+\partial_{r}^{2}+\frac{1}{r}\partial_{r}+\frac{2}{r^{2}}
(1-3 u_N(t,r)^2 - 3 \epsilon(t,r) u_N(t,r)
-\epsilon^2(t,r))\Big)\epsilon(t,r) = e_N
\]
We divide this equation into a linear part and a nonlinear
perturbative term. Based on past experience one would expect that in
the main linear part $u_N$ is simply replaced by $Q(\lambda(t)r)$.
However, in this case that is not enough. Instead, as it turns out,
some of the effects of the first correction term $v_{10}$ also need
to be taken into account. Hence the above equation is rewritten in
the form
\begin{equation}
\Big(-\partial_{t}^{2}+\partial_{r}^{2}+\frac{1}{r}\partial_{r}+\frac{2}{r^{2}}
(1-3Q(\lambda(t)r)^2 -6 Q(\lambda(t)r) v_{10}) \Big)\epsilon = e_N +
\cN(\epsilon) \label{eqnlin}\end{equation} where
\[
\cN(\epsilon) = \frac{2}{r^{2}}\left( 3 \epsilon(u_N^2
-Q(\lambda(t)r)^2
  -2  Q(\lambda(t)r) v_{10}) +     3 \epsilon^2 u_N +\epsilon^3\right)
\]

We first consider the linear problem
\begin{equation}
  \Big(-\partial_{t}^{2}+\partial_{r}^{2}+\frac{1}{r}\partial_{r}+\frac{2}{r^{2}}
    (1-3Q(\lambda(t)r)^2 -6 Q(\lambda(t)r) v_{10}) \Big)\epsilon = f
\label{eqlin}\end{equation}
where the principal spatial part is given by the selfadjoint operator
\[
L_t = -\partial_{r}^{2}-\frac{1}{r}\partial_{r}-\frac{2}{r^{2}}
    (1-3Q(\lambda(t)r)^2)
\]
This is time dependent, but is obtained by rescaling from
the operator
\[
L =  -\partial_{r}^{2}-\frac{1}{r}\partial_{r}-\frac{2}{r^{2}}(1-3Q(r)^2)
\]
We remark that, as proved in the next section, $L$ is a nonnegative
operator.

A difficulty that we face in solving \eqref{eqnlin} iteratively is
in handling the singularity at $0$ in the $\epsilon^3$ term in
$\cN(\epsilon)$. Energy estimates on $\epsilon$ do not suffice, so
we introduce the scaling vector field
\[
S := t \partial_t + r \partial_r
\]
and we seek to simultaneously bound $\epsilon$, $S\epsilon$ and
$S^2 \epsilon$ in a norm that is a scale adapted version of the $H^1$
norm,
\[
\|\eps\|_{H^1_N} := \sup_{0<t<t_0} |\log t|^{N-\beta-1} \big(
\|L_t^{\frac12} \eps\|_{L^2(rdr)}+\|\pr_t
\eps\|_{L^2(rdr)} + \lambda(t)   |\log t|^{-\beta} \|\eps\|_{L^2(rdr)} \big)
\]
For $f$, on the other hand, we just use uniform $L^2$ bounds,
\[
\|f\|_{L^2_N} := \sup_{0<t<t_0} \lambda^{-1}(t)|\log
t|^N\|f(t)\|_{L^2(rdr)}
\]
The main result of the linear theory is the following theorem. It is
proved in Section~\ref{sec:lin}.

\begin{theorem}
  There exists a linear operator $\Phi$, so that for each $f$ the
  function $\epsilon = \Phi f$ solves \eqref{eqlin}, and for all large
  enough $N_0 \gg N_1 \gg N_2$ it satisfies the bounds
\begin{align}
 \| \Phi f\|_{H^1_{N_0}}  &\lesssim  \frac1{N_0} \| f\|_{L^2_{N_0}}
\label{kbd0}\\
  \| S\Phi f\|_{H^1_{N_1}}  &\lesssim  \frac1{N_1}(\|Sf\|_{L^2_{N_1}} +
\| f\|_{L^2_{N_0}}) \label{kbd1}\\
   \| S^2\Phi f\|_{H^1_{N_2}}  &\lesssim  \frac1{N_2}(\|S^2f\|_{L^2_{N_2}}
   + \|Sf\|_{L^2_{N_1}} +
\| f\|_{L^2_{N_0}}) \label{kbd2}\end{align} 
 The implicit constants here depend only on~$\beta$.
\label{tlin}\end{theorem}

We note that $\Phi$ is in effect the forward solution operator for the
equation \eqref{eqlin}. In this theorem $f$ is not required to be
supported inside the cone $\{r \leq t\}$. However, if that is the case
the $\Phi f$ is also supported inside the cone due to the finite speed
of propagation.

  In order to prove the above theorem it is
convenient to pass to different coordinates in which the Schr\"odinger
operator is no longer time dependent. Specifically, introduce new
coordinates $(\tau,R)$ given by
 \[
\tau=\int_{t}^{t_{0}}\lambda(s)ds, \qquad R = \lambda(t) r
\]
Then, denoting
\[
\tileps(\tau, R):=R^{\frac{1}{2}}\epsilon(t,r), \qquad \tilf(\tau,
R) := \lambda^{-2}R^{\frac{1}{2}} f(t,r)
\]
 where $\lambda$ is now understood as a function of $\tau$, the
 equation \eqref{eqlin} becomes
\[
\left[-(\partial_{\tau}+\frac{\lambda_{\tau}}{\lambda}R\partial_{R})^{2}+
\frac{1}{4}(\frac{\lambda_{\tau}}{\lambda})^{2}+
\frac{1}{2}\partial_{\tau}(\frac{\lambda_{\tau}}{\lambda})\right]\tileps-\calL\,\tileps-\frac{12}{R^{2}}Q(R)v_{10}\,\tileps
=\tilf,
\]
 where
\[
\calL:=-\frac{\partial^{2}}{\partial R^{2}}+\frac{5}{14
  R^{2}}-\frac{24}{(1+R^{2})^{2}}
\]
The spectral properties of the operator $\calL$ are studied in
Section~\ref{sec:spec}. These are essential in the proof of
Theorem~\ref{tlin} in Section~\ref{sec:lin}.

We next continue the proof of our main result with the perturbative
argument for the equation \eqref{eqnlin}. By the construction in
Section~\ref{sec:mod} we know that for arbitrarily large $N_0 \gg
N_1 \gg N_2$ we can find an approximate solution $u_N$ so that the
corresponding error $e_N$ satisfies
\[
\|e_N\|_{Y} := \| e_N\|_{L^2_{N_0}} + \| S e_N\|_{L^2_{N_1}} + \|
S^2 e_N\|_{L^2_{N_2}} \ll1
\]
where the smallness is gained by taking $t_0$ small enough. It is
important to note that, even though $e_N$ has limited regularity,
the roughness is relative to the self-similar variable
$a:=\frac{r}{t}$ which satisfies $Sa=0$. For this reason $S^j e_N$
does not lose any regularity. We iteratively construct the sequence
$\{(\epsilon_j, f_j)\}_{j \geq 0}$ by
\[
f_0 := e_N, \qquad \epsilon_j := \Phi f_j, \qquad f_{j+1} := e_N +
\cN(\epsilon_j)
\]
and show that it converges to a solution $\epsilon$
of~\eqref{eqnlin} in the norm
\[
\| \epsilon\|_{X} :=  \| \epsilon \|_{H^1_{N_0}} + \| S \epsilon
\|_{H^1_{N_1}} + \| S^2 \epsilon \|_{H^1_{N_2}}
\]
 By Theorem~\ref{tlin} we know that
$\Phi$ is a bounded operator with small norm,
\begin{equation}\label{eq:Phinorm}
\| \Phi f\|_{X} \lesssim N_0^{-1} \|f\|_{Y}
\end{equation}
The proof is concluded if we show that the nonlinear term satisfies a
similar bound:

\begin{lemma}
The map $f \mapsto \cN(\Phi f)$ is locally Lipschitz from $Y$ to
$Y$, with Lipschitz constant of size $O(N_2^{-1})$.
\end{lemma}

\begin{proof}
We denote $\epsilon = \Phi f$, and successively consider the linear
and the nonlinear terms in $\cN(\epsilon)$.

{\bf A. The linear term} has the form
\[
\cN_1(\epsilon) = g \epsilon, \qquad g = \frac{2}{r^2}(u_N^2 -
Q(\lambda(t)r)^2 - 2Q(\lambda(t)r) v_{10})
\]
By construction we have
\[
|g| +|Sg| +|S^2 g| \lesssim \frac{1}{t^2|\log t|} =\lambda(t)^2
|\log t|^{-2\beta-1}
\]
which directly leads to
\[
\| \cN_1(\epsilon)\|_{Y} \lesssim \|\epsilon\|_{X}.
\]
where only the $L^2$ components of the $H^1_{N_j}$ norms are being
used (as part of the space~$X$). The desired conclusion now follows
from~\eqref{eq:Phinorm}.

{\bf B. The nonlinear term} has the form
\[
\cN_2(\epsilon) = \frac{2}{r^2} (3u_N \epsilon^2+\epsilon^3)
\]
The coefficient $u_N$ satisfies
\[
|u_N| + |Su_N| + |S^2 u_N| \lesssim 1
\]
so we can neglect it. The main difficulty here arises from the
singular factor $\frac{1}{r^{2}}$ on the right-hand side. To address
that we  will establish several bounds. The first is a pointwise bound,
\begin{equation}
 |w| \lesssim  |\log t|^{1+2\beta-N} \|w\|_{H^1_N}
\label{linfty}\end{equation}
which applied to $\epsilon$, $S\epsilon$ and $S^2 \epsilon$ yields
\begin{equation}
 \|S^k \epsilon(t)\|_{L^\infty} \lesssim \frac{1}{N_k}
 |\log t|^{1+2\beta-N_k}\|f\|_{Y}, 
\qquad k = 0,1,2
\label{linftye}\end{equation}
The second is a weighted $L^2$ bound, namely
\begin{equation}
\| r^{-2} \epsilon(t)\|_{L^2} \lesssim  \lambda(t) |\log t|^{-N_2+1}   \| f\|_{Y}
\label{l2s2}\end{equation}
Interpolating between the $k=2$ case of \eqref{linftye}
and \eqref{l2s2} we also obtain
\begin{equation}
\| r^{-1} S\epsilon(t)\|_{L^4} \lesssim \frac{1}{N_2^\frac12} \lambda(t)^\frac12 
|\log t|^{-N_2+1+\beta}   \| f\|_{Y}
\label{l4s}\end{equation}
By \eqref{linftye} with $k=0$ and \eqref{l2s2}
we obtain
\begin{equation}\label{bootstr1}\begin{split}
\big\|\cN_1(\epsilon)(t) \big\|_{L^{2}_{N_0}}
&\lesssim
\big[\|\epsilon(t)\|_{L^{\infty}}+\|\epsilon(t)\|_{L^{\infty}}^{2}\big]\Big\|\frac{\epsilon(t)}{r^{2}}
\Big\|_{L^{2}}\\& 
\lesssim \frac{1}{N_0} \lambda(t) |\log t|^{2+2\beta -N_0-N_2} (\|f\|_Y^2+\|f\|_{Y^3})
\end{split}\end{equation}
Using also \eqref{linftye} with $k=1$ we similarly obtain
\begin{equation}\label{bootstr2}
\big\|S\cN_1(\epsilon) (t) \big\|_{L^{2}_{N_0}}
\lesssim \frac{1}{N_1} \lambda(t) |\log t|^{2+2\beta -N_1-N_2} (\|f\|_Y^2+\|f\|_{Y^3})
\end{equation}
Finally, due to \eqref{linftye} with $k=2$ and \eqref{l4s} we also
have
\begin{equation}\label{bootstr3}
\big\|S^2\cN_1(\epsilon) (t) \big\|_{L^{2}_{N_0}}
\lesssim \frac{1}{N_2} \lambda(t) |\log t|^{2+2\beta -2N_2} 
(\|f\|_Y^2+\|f\|_{Y^3})
\end{equation}
Together, the bounds \eqref{bootstr1}-\eqref{bootstr3} suffice to 
obtain the conclusion of the lemma provided that
$N_2$ is large enough.

It remains to prove the bounds  \eqref{linfty} and \eqref{l2s2}. For the
operator $L$ we have the straightforward elliptic bound
\[
 \|\nabla w\|_{L^2} + \| r^{-1} w\|_{L^2} \lesssim
\| L^\frac12 w\|_{L^2} + \|w\|_{L^2}
\]
By rescaling this gives
\begin{equation}
\begin{split}
\|\nabla w\|_{L^2} + \|r^{-1} w\|_{L^2} \lesssim &\ \| L_t^\frac12
w\|_{L^2} + \lambda(t) \|w\|_{L^2}
\\ \lesssim &\ |\log t|^{1+2\beta -N}  \| w \|_{H^1_N}
\end{split}
\label{ent}\end{equation}
Then~\eqref{linfty} follows from the point-wise bound for
spherically symmetric functions in~$\R^2$
\[
\| u\|_{L^\infty} \lesssim \|\nabla u\|_{L^2} + \|r^{-1} u\|_{L^2}
\]
Next we turn our attention to the bound \eqref{l2s2}. Due to
\eqref{linftye} ($k=0$) it suffices to consider the region $r \leq t/2$. In this
region we use the scaling vector field $S=t\pr_t+r\pr_r$ to derive a
stronger equation for $\epsilon$. From
\[
t\partial_t = S - r \partial_r
\]
one infers that
\[
t^2 \partial_t^2 \epsilon  + t \partial_t \epsilon = -S^2 \epsilon +
r^2 \partial_r^2 \epsilon + 2 t \partial_t S \epsilon
\]
and further
\[
\Big(\frac{t^{2}-r^{2}}{t^2}\partial_{rr}+\frac{1}{r}\partial_{r}-
\frac{4}{r^{2}}\Big) \epsilon = t^{-2}(-S^2 \epsilon + 2 t
\partial_t S \epsilon -t\partial_t \epsilon) - \big(\Box +
\frac{4}{r^2}\big)\epsilon
\]
Due to \eqref{eqlin} we can estimate the last term by
\[
\left |\left(\Box + \frac{4}{r^2}\right)\epsilon\right| \lesssim |f|
+ \lambda^{2}|\epsilon|
\]
This leads to the bound
\[
\begin{split}
\left\|\left(\frac{t^{2}-r^{2}}{t^2}\partial_{rr}+\frac{1}{r}\partial_{r}-
\frac{4}{r^{2}}\right) \epsilon\right\|_{L^2} \lesssim&\
t^{-2}[\|S^2 \epsilon \|_{L^{2}} + \|(t\pr_t S) \epsilon\|_{L^{2}} +
\|(t\pr_t)\epsilon\|_{L^{2}}] + \\
& + \lambda^{2}\|\epsilon\|_{L^{2}}+\|f\|_{L^2}
\\
 \lesssim &\ \lambda(t)|\log
 t|^{1-N_2}\|\epsilon\|_{X}+\|f\|_{L^{2}}
\end{split}
\]
Taking also into account \eqref{ent} applied to $\epsilon$ with
$N=N_0$, the estimate \eqref{l2s2} would follow from the fixed time
bound
\begin{equation}
\label{eq:ell_est} \|r^{-2} \epsilon\|_{L^2} \lesssim \|t^{-1}
\nabla \epsilon\|_{L^2} + \|t^{-1} r^{-1} \epsilon\|_{L^2} +
\left\|\left(\frac{t^{2}-r^{2}}{t^2}\partial_{rr}+\frac{1}{r}\partial_{r}-
\frac{4}{r^{2}}\right) \epsilon\right\|_{L^2}
\end{equation}
This rescales to $t=1$, in which case it is a standard local
elliptic estimate near $r=0$.
\end{proof}

\section{The renormalization step}
\label{sec:mod}

In this section, roughly following~[KST1], we show how to construct an
arbitrarily good approximate solution to the wave
equation~\eqref{eq:wave} as a perturbation of a time-dependent
profile
\begin{equation}\label{eq:u0def}
u_0 = Q(R), \qquad R =r
\lambda(t),\quad \Phi(R)=\frac{R^2}{(1+R^2)^2}
\end{equation}
with $\lambda(t)$ a logarithmic correction to the self-similar ansatz
\[
\lambda(t)=t^{-1}(-\log t)^{\beta},\,\beta\ge 1
\]
In fact, for ease of notation we will take $\beta\in\Z$; the general
case is only a minor modification of the integral one and we leave
it to the reader.  This ansatz is quite natural in light of a
necessary orthogonality condition which makes its appearance in the
ensuing considerations. We note, however, that by contrast to
~[KST1], the approximate solutions here are much rougher, and indeed
asymptotically only lie in $H^{1}$, the threshold for local
well-posedness of the critical Yang-Mills equation. The reason for
this is the much more singular nature of the ODE's arising in the
renormalization step, due to the different blow up rate.

 The following is the main theorem of the first half of the paper.
 Throughout this section, we will work on the light-cone $\{r<
 t\}$ (in particular, all functions in this section will be defined
 only on $r\le t$).

\begin{theorem}
  \label{thm:sec2} Let $k \in \Nat$. There exists an approximate
  solution $u_{2k-1}\in H^{1}$ for \eqref{eq:wave} of the asymptotic
  form (as $R\rightarrow \infty$)
  \[
  u_{2k-1}(r,t) = Q(\lambda(t) r) 
   +\frac{1}{(t\lambda)^2}\frac{R^4}{4(1+R^2)^2}
   +O\left(\frac{R^2}{(t \lambda )^2 |\log t|}\right)  
  \]
 so that the corresponding error has size
  \[
  e_{2k-1} = O\left(\frac{1}{t^2 (t \lambda)^{2k}}\right)
  \]
  Here the $O(\cdot)$ terms are uniform in $0\le r\le t$ and $0<t<t_0$
  where $t_0$ is a fixed small constant; they are also stable with
  respect to the application of powers of the scaling operator $S$.
  We also have $u_{2k-1}(.,t)\in C^{\infty}([0,t))$, and further
  $u_{2k-1}\in H^{1}$. The only singularity arises on the light cone
  $r=t$.
\end{theorem}

\begin{proof}
  We iteratively construct a sequence $u_k$ of better approximate
  solutions by adding corrections $v_k$,
  \[
  u_k = v_{k} + u_{k-1}
  \]
  The error at step $k$ is
  \[
  e_k = (\partial_t^2 - \partial_r^2 - \frac{1}r \partial_r) u_k
  -\frac{2}{r^{2}}f(u_k),\qquad f(u)=u(1-u^2)
  \]
  To construct the increments $v_k$ we first make a heuristic
  analysis. If $u$ were an exact solution, then the difference
  \[
  \epsilon = u-u_{k-1}
  \]
  would solve the equation
  \begin{equation}\label{eq:gleich}\begin{split}
  &(-\partial_t^2 + \partial_r^2 +\frac{1}r \partial_r) \epsilon +
  \frac{2}{r^2}f'(u_{k-1})\eps \\
  & = e_{k-1} -\frac{2}{r^2}
  (f(u_{k-1}+\eps) - f'(u_{k-1})\eps - f(u_{k-1})) \\
  & = e_{k-1} +\frac{2}{r^2}(3\eps^2
  u_{k-1} + \eps^3)
  \end{split}
  \end{equation}
  In a first approximation we linearize this equation around $\epsilon
  =0$ and substitute $u_{k-1}$ by $Q(R)$. Then we obtain the linear
  approximate equation
  \begin{equation}
    \left(-\partial_t^2 + \partial_r^2 +\frac{1}r \partial_r +
      \frac{2}{r^{2}}(1-3Q^2)\right) \epsilon  \approx e_{k-1}
    \label{eps}  \end{equation} For $r \ll t$ we expect the time
  derivative to play a lesser role  so we neglect it and we are left
  with an elliptic equation with respect to the variable $r$,
  \begin{equation}
    \left(\partial_r^2 +\frac{1}r \partial_r +
      \frac{2}{r^{2}}(1-3Q^2)\right) \epsilon  \approx e_{k-1}, \qquad r
    \ll t \label{epsodd}\end{equation} For $r \approx t$ we  rewrite \eqref{eps} in the form
  \[
  \left(-\partial_t^2 + \partial_r^2 +\frac{1}r \partial_r -
    \frac{4}{r^{2}}\right) \epsilon \approx e_{k-1}
  \]
  Here the time and spatial derivatives have the same strength.
  However, we can identify another principal variable, namely $a =
  r/t$ and think of $\epsilon$ as a function of $(t,a)$.  As it turns
  out, neglecting a "higher order" part of $e_{k-1}$ which can be directly included
  in $e_k$, we are able to use scaling and the exact structure of the
  principal part of $e_{k-1}$ to reduce the above equation to a
  Sturm-Liouville problem in~$a$
  which becomes singular at~$a=1$.

  The above heuristics lead us to a two step iterative construction of
  the $v_k$'s. The two steps successively improve the error in the two
  regions $r \ll t$, respectively $r \approx t$. To be precise, we
  define $v_k$ by
  \begin{equation}
    \left(\partial_r^2 +\frac{1}r \partial_r + \frac{2}{r^{2}}(1-3Q^2)\right) v_{2k+1}  = e_{2k}^0
    \label{vkodd}\end{equation} respectively
  \begin{equation}
    \left(-\partial_t^2+\partial_r^2 +\frac{1}r \partial_r - \frac4{r^{2}}\right) v_{2k}  = e_{2k-1}^0
    \label{vkeven}\end{equation} both equations having zero Cauchy
  data\footnote{The coefficients are singular at $r=0$, therefore this
    has to be given a suitable interpretation} at $r=0$.  Here at each
  stage the error term $e_k$ is split into a principal part and a
  higher order term (to be made precise below),
  \[
  e_k = e_k^0 + e_k^1
  \]
  The successive errors are then computed as
  \[
  e_{2k} = e_{2k-1}^1 + N_{2k} (v_{2k}), \qquad e_{2k+1} = e_{2k}^1
  +
  \partial_t^2 v_{2k+1} + N_{2k+1} (v_{2k+1})
  \]
  where
  \begin{equation}
    N_{2k+1}(v) =  \frac{6}{r^2}(u_{2k}^2-Q^2) v +  \frac{2}{r^{2}}
  (3v^2 u_{2k} + v^3)
    \label{eodd}\end{equation} respectively
  \begin{equation}
    N_{2k}(v) = \frac{6}{r^2}(u_{2k-1}^2-1) v +  \frac{2}{r^{2}}
  (3v^2 u_{2k-1} + v^3)
    \label{eeven}\end{equation}

  To formalize this scheme we need to introduce suitable function
  spaces in the cone \[\calC_0=\{(r,t)\::\:0\le r<t, 0<t<t_0\}\] for
  the successive corrections and errors. We first consider the $a$
  dependence. For the corrections $v_k$ we use the following general concept

\begin{defi}\label{def:Q}
Let $k \geq 0$. Then  $\mathcal Q_k$ is the algebra of continuous
functions $q(a,\alpha,\alpha_1)$
\[
q:(0,1] \times \R \times \R \to  \R
\]
with the following properties:

(i) $q$ is smooth in $ a \in (0,1)$, and meromorphic and even around
$a=0$. Further, the restriction to the diagonal
\[ \tilde{q}(a,b):=q(a,b,b)\]
extends analytically to $a=0$ and has an even expansion there.

(ii) $q$ has the form
\[
q(a,\alpha,\alpha_1) = \sum_{i+j \leq - k}^{j \leq 0, i \leq |j|/2 }
q_{ij}(a,\ln \alpha, \ln \alpha_1) \alpha^i \alpha_1^j
\]
with $q_{ij}$ polynomial in $\log\alpha, \log\alpha_{1}$. The sum
only has finitely many terms.

 (iii) Near $a=1$ we have a
representation of the
  form
  \[
  q = q_0(a,\alpha) + (1-a^2)^{\frac12} q_1(a,\alpha,\alpha_1)
  \]
  with coefficients $q_0$, $q_1$ analytic in $a$ around~$a=1$.
 \end{defi}

The order of the pole at $a=0$ as it appears in
Definition~\ref{def:Q}, part~(i), is controlled by some absolute
constant depending only on~$k$. The same comment applies to every
pole at $a=0$ appearing in this section and will be assumed tacitly
throughout.   For the errors $e_k$ we introduce another functions
class:

\begin{defi}
   $\mathcal Q_{k}^{'}$ is the space of continuous functions $q(a,\alpha,\alpha_1)$
\[
q:(0,1) \times \R \times \R \to  \R
\]
with the following properties:

(i) $q$ is smooth in $ a \in (0,1)$,   meromorphic and even around
$a=0$. The restriction to the diagonal
\[
\tilde{q}(a,b,b)
\]
extends analytically to $a=0$

(ii) $q$ has a representation as in (ii) of the preceding definition

 (iii) Near $a=1$ we have a representation of the
  form
\[
  q = q_0(a,\alpha) + (1-a^2)^{\frac12}
  q_1(a,\alpha,\alpha_1)+
(1-a^2)^{-\frac12} q_2(a,\alpha,\alpha_1)
\]
  with coefficients $q_0$, $q_1$, $q_2$ analytic with respect to $a$ around $a=1$.
  Moreover, $q_2$ has the same representation as $q$ in (ii), but with $k$ replaced by
$k+1$  and $j\leq -1$.
 \end{defi}

Next we define the class of functions of $R$:

\begin{defi}
  $S^m(R^k (\ln R)^\ell)$ is the class of analytic functions
  $v:[0,\infty) \to \R$ with the following properties:

  (i) $v$ vanishes of order $m$ at $R=0$, and $R^{-m}v$ has an even expansion around $R=0$.

  (ii) $v$ has a convergent expansion near $R=\infty$,
  \[
  v(R) = \sum_{0 \leq j \leq \ell+i} c_{ij}\, R^{k-2i} (\ln R)^{j}
  \]
\end{defi}
Finally, we introduce the auxiliary variables
\[
b := |\ln t|, \qquad b_1:= |\ln t| + |\ln p(a)|
\]
where $p$ is a real even polynomial with the following properties:
\[
p(1)= 0,\quad p'(1) = -1, \quad p(a)=1+O(a^M) \text{\ \ as\ \ }a\to0
\]
where $M$ is a very large number (depending on the number $k$ of
steps in our iteration), and $p$ has no zeroes in $(0,1)$. We can
now define the main function class for our construction.

\begin{defi}
  a) $S^m(R^k (\ln R)^\ell,\mathcal Q_n)$ is the class of analytic
  functions
\[
v:[0,\infty) \times [0,1]\times \R^{2} \to \R
\]

  (i) $v$ is analytic as a function of $R$,
  \[
  v: [0,\infty)  \to \mathcal Q_{n}
  \]

  (ii) $v$ vanishes of order $m$ at $R=0$

  (iii) $v$ has a convergent expansion at $R=\infty$,
  \[
  v(R,\cdot,b,b_1) = \sum_{0 \leq j \leq \ell+i} c_{ij}(\cdot,b,b_1) R^{k-2i}
  (\ln R)^{j}
  \]
  with coefficients $c_{ij} \in \cQ_{n}$.

  b) $\IS^m(R^k (\ln R)^\ell,\mathcal Q_n)$ is the class of analytic
  functions $w$ on the cone $\calC_0$ which can be represented as
  \[
  w(r,t) = v(R,a,b,b_1), \qquad v \in S^m(R^k (\ln R)^\ell,\mathcal Q_n)
  \]
\end{defi}

We note that the representation of functions on the cone as in part
(b) is in general not unique since $R,a,b$ are dependent variables.
Later we shall exploit this fact and switch from one representation
to another as needed. We start our construction with some explicit 
computations which allow us to establish the regularity of the first
few terms in the iteration, namely
\begin{align}
\label{v1}
v_1 & \in 
(t\lambda)^{-2}\frac{R^4}{4(1+R^2)^2} + (\lambda t)^{-2} 
\left(  \frac{1}{|\log t|} IS^4(R^2)+\frac{1}{|\log t|^2} IS^4(R^2)\right)
\\ \label{e1}
t^2 e_1 & \in  (\lambda t)^{-2} \left( IS^{4}(1)+ \frac{1}{|\log t|}
IS^4(R^2)+\frac{1}{|\log t|^2} IS^4(R^2)\right)
\\ \label{v2}
v_2 & \in a^4 IS(1,\cQ_{1})
\end{align}
After these few steps we reach the general pattern, and prove by
induction that the successive corrections $v_k$ and the corresponding
error terms $e_k$ can be chosen with the following properties:
\begin{align}
  v_{2k-1} &\in \IS^4(R^2 (\log R)^{k-1}, \cQ_{2\beta k})
  \label{v2k-1}\\
  t^2 e_{2k-1} &  \in  \IS^2(R^2 (\log R)^{k-1}, \cQ'_{2\beta k})
 \label{e2k-1}\\
  v_{2k} &\in a^4 \IS((\log R)^{k-1}, \cQ_{2\beta(k-1)})
  \label{v2k}\\
  t^2 e_{2k} &\in a^{2}\IS((\log R)^{k-1}, \cQ_{2\beta k}')+\IS^4((\log R)^{k-1}, \cQ_{2\beta k})
\label{e2k}
  \end{align}
The properties \eqref{v1}-\eqref{e2k} suffice in order to reach the 
conclusion of the theorem. We note that  is easy to verify that all the
above classes of functions are left invariant by the scaling operator $S$.

The proof of \eqref{v1}-\eqref{e2k}  roughly follows that in [WM], [SL].
There is, however, an important difference near the light cone: for
the critical Wave Maps problem as well as the critical focussing
semilinear equation, the singularity at the boundary of the light
cone is well modeled by the expression $(1-a)^{\frac{1}{2}+\nu}$,
which comes from the choice of blow up rate $t^{-1-\nu}$. For
Yang-Mills, due to the much faster blow up speed, we need to
essentially use the much more singular expression
\begin{equation}\label{eq:crux}
  \frac{(1-a)^{\frac{1}{2}}}{|\log t|+|\log p(a)|} 
\end{equation}
where $p(a)$ is a polynomial so that $p(1)=0$.  This renders the
algebra significantly more delicate. We remark that~\eqref{eq:crux}
appears canonically in this section. On one hand, $(1-a)^{\frac12}$ is
part of a fundamental system of that ODE which~\eqref{vkeven} reduces
to in the self-similar variable $a=\frac{r}{t}$. This is exactly what
one would obtain by neglecting all but the selfsimilar components of
the wave operator. However, unlike in [WM], [SL], here we encounter a
nontrivial non-selfsimilar effect which forces the exact denominator
in \eqref{eq:crux}.  In particular this saves the day by insuring
that~\eqref{eq:crux} belongs to $H^1(0,1)$ which of course is a
minimal requirement here.

To commence the construction of the $v_{k}$, we recall that
\[
Q(R) = \frac{1-R^2}{1+R^2}, \qquad \Phi(R) = \frac{R^2}{(1+R^2)^2}
\]
where $\Phi$ is the zero eigenfunction, $L \Phi = 0$
with
\[
L = \partial_R^2 + \frac{1}R \partial_R + \frac{2}{R^2}(1-3Q^2)
\]
By \eqref{eq:u0def} we have
\begin{equation}\nn
\begin{split}
t^2 e_0  =&\  -t^2 \partial_t^2 Q(R)\\ =&\ \left(1+\frac{\beta}{|\log
  t|}\right)^{2}R\Phi'(R)+
\left(1+\frac{\beta}{|\log t|}-\frac{\beta}{|\log t|^{2}}\right)\Phi(R)
\end{split}
\end{equation}

\medskip

{\bf Step 1:} {\em Begin with $e_{0}$ as above and choose $v_{1}$ so that \eqref{v2k-1} for $k=1$ holds. Further the error $e_{1}$ thereby generated is of the form \eqref{e2k-1} for $k=1$.}
   \\
 Here, we simply put $e_{0}^{0}:=e_{0}$. Reformulate the equation for $v_{1}$ as follows:
  \[
  (t\lambda)^{2}\tilde{L} \sqrt{R}v_{1}=\sqrt{R}t^{2}e_{0},\quad
\tilde{L}=\partial_{R}^{2}-\frac{15}{4R^{2}}+\frac{24}{(1+R^{2})^{2}}
  \]
  Using the above calculation of $e_{0}$, it is then straightforward
  to write down an absolutely convergent Taylor expansion of $v_{1}$ around
  $R=0$. Since $t^2 e_0$ vanishes of second order at $0$, it follows
  that $v_1$ vanishes of order four at $0$.

  Now we turn to the expansion around $R=\infty$. The leading term in
  $t^2 e_0$ is $R\Phi'(R)+\Phi(R)$. A key fact is that this satisfies
  the orthogonality condition
\[
\la R\Phi'(R)+\Phi(R),\Phi \ra_{\R^2} = 0
\]
 It is partly this orthogonality condition which motivates our choice of $\lambda(t)$.
As a consequence, the solution to $Lv_{10} = R\Phi'(R)+\Phi(R)$ does not grow at $\infty$, precisely it equals
\[
v_{10} = \frac14 (t\lambda)^{-2}\frac{R^4}{(1+R^2)^2}
\]
For the remaining terms we do not have such a precise representation
since we lack the orthogonality condition.  We use this fundamental
system of solutions for $\tilde L$:
  \[
  \phi_0(R) = \frac{R^{\frac52}}{(1+R^2)^2}, \qquad \theta_0(R) =
  \frac{-1-8R^2 + 24R^4\log R + 8R^6+ R^8}{4R^{\frac32}\,(1+R^2)^2}
  \]
  Their Wronskian is $W(\phi_0,\theta_0)=1$. Then
  $\Phi(R)=R^{-\frac12} \phi_0(R)$ and define
  $\Theta(R):=R^{-\frac12}\theta_0(R)$ so that $L\Phi=0$, $L\Theta=0$,
  respectively\footnote{Note the appearance of $\Phi(R)\log R$ as part
    of $\Theta$.}.  One thus obtains an integral representation for
  $v_{1}$ using the variation of parameters formula, which gives
  \begin{align}
    (t\lambda)^{2}v_{1}(R) &=  \frac{\theta_0(R)}{\sqrt{R}}\! \int_0^R
    \!\! \phi_0(R')\sqrt{R'} t^{2}e_{0}(R')\,
  \,dR' -  \frac{\phi_0(R)}{\sqrt{R}} \! \int_1^R\!\!  \theta_0(R') \sqrt{R'}
  t^{2}e_{0}(R')\, \,dR'\nn\\&
  =\Theta(R) \int_0^R \Phi(R')t^{2}e_{0}(R') R'\, dR' - \Phi(R)\int_1^R
  \Theta(R') t^{2}e_{0}(R') R'\, dR'  \nn
   \end{align}
In the end we obtain the representation
\begin{equation} \label{dv1}
v_1=v_{10}+v_{11}, \qquad v_{11} \in (\lambda t)^{-2} \left(  \frac{1}{|\log t|} IS^4(R^2)+\frac{1}{|\log t|^2} IS^4(R^2)\right)
\end{equation}
which implies \eqref{v1}.

Next, we determine the error, which is given by
\[
t^2 e_{1}=t^2 \partial_{t}^{2}v_{1}- \frac{3 (\lambda t)^2}{R^2}(3 v_1^2 Q + v_1^3)
\]
After some computations we obtain the relation \eqref{e1}, namely
\begin{equation} \label{de1}\begin{split}
t^2 e_{1} &\in  (\lambda t)^{-2} \left( IS^{4}(1)+ \frac{1}{|\log t|}
IS^4(R^2)+\frac{1}{|\log t|^2} IS^4(R^2)\right) 
\end{split}
\end{equation}

\medskip

{\bf Step 2:} Recall that $v_2$ is determined by \eqref{vkeven},
which requires specifying $e_1^0$. This will be done iteratively,
which means that
\begin{equation}\label{eq:e10}
e_1^0 = \sum_{j=0}^J e_1^{0j}
\end{equation}
where $J=J(\beta)$ grows with $\beta$ and $e_1^{0j}$ is specified
recursively. To being with, we extract the leading order (in terms
of growth in~$R$) from $e_1$ and set
\[
t^2 e_{1}^{00} :=  c_1(\lambda t)^{-2}  \frac{1}{|\log t|} R^2 + c_2
(\lambda t)^{-2}  \frac{1}{|\log t|^2} R^2 = c_1 \frac{a^2}{b}+
c_2\frac{a^2}{b^2}
\]
with suitable constants $c_1, c_2$.  Note that then
\[ 
e_{1}^{10}:=e_{1}-e_{1}^{00}\in IS^{2}(1, \cQ_{2\beta})
\] 
which is
admissible for $e_2$, see~\eqref{e2k}. Replacing $Q$ by $1$ we now
seek to  solve the linear differential equation
 \begin{equation}\label{eq:v2e100}
  t^2\left(-\partial_t^2 + \partial_r^2 +\frac{1}r \partial_r -
    \frac{4}{r^{2}}\right) v_2 = t^{2}\,e_{1}^{00}
 \end{equation}
In the $a,b$ coordinates the above equation is
rewritten as
\[
L_{ab}\,  v_2(a,b) =   c_1\frac{a^2}{b}  +    c_2  \frac{a^2}{b^2}
\]
where
\[
 L_{ab}\,  = -(\partial_b + a \partial_a)^2 - (\partial_b + a \partial_a) + \partial_a^2 +\frac{1}a \partial_a -
    \frac{4}{a^{2}}
 \]
Set also the $b$ independent part
\[
L_a =  (1-a^2)\partial_a^2 +\left(\frac{1}a-2a\right) \partial_a -
    \frac{4}{a^{2}}
\]
For technical reasons, we will only obtain an approximate
solution~$v_2$ of~\eqref{eq:v2e100}. We then face a dichotomy:
either the error $L_{ab}\, v_2-e_1^{00}$ is acceptable for~$e_2$ or
not; in the latter case, we repeat the procedure by including the
unacceptable error in~$e_1^0$ and solving for a correction to~$v_2$.
This process (which also needs to take the nonlinear component
of~$e_2$ into account, see~\eqref{eeven}) then leads to the
aforementioned iterative construction of~$e_1^0$ and~$v_2$.

We begin by constructing an approximate solution to $L_{ab}\,
w_2=\frac{a^2}{b}$.
 The approximate
solution in the following lemma is called $w_2$ rather than $v_2$
since the latter will be the sum of various expressions, the first
being~$w_2$.

\begin{lemma}\label{lem:3.6}
Let $e(a)$ be even analytic and quadratic at $a=0$.
There is an approximate solution $w_2$ for
\[
L_{ab}\,  w_2 = b^{-1} e(a)
\]
which is of the form
\begin{equation}\label{eq:w2def}
w_2(a,b) = b^{-1} W_2^0 (a) +  b_1^{-1} (1-a)^{\frac12} W_2^1 (a)
\end{equation}
where $W_2^0$, $W_2^1$ are even analytic in $a \in (0,1]$ with an
$a^{-2}$ leading term at $0$ so that $w_2$ vanishes to fourth order
at $a=0$. The error has the form
\begin{align}
f_2^0 &:= L_{ab}\,  w_2 - b^{-1} e(a) \nonumber \\
&=E_2^0 (a,b^{-1}) +  (1-a)^{\frac12} E_2^1 (a, b_1^{-1})
\label{eq:Lab_err}
\end{align}
where $E_2^0$, $E_2^1$ are analytic in $a \in (0,1]$, linear in
 $b^{-2}$, $b_1^{-2}$,  $b^{-3}$, $b_1^{-3}$, and vanish
quadratically at $a=0$.
\end{lemma}

\begin{proof}
We begin with the ansatz
\[
w_2 = \frac{W_2^0 (a)}b + \frac{ (1-a)^{\frac12} W_2^1 (a)}{b_1}
\]
where \[L_a W_2^0 (a)=e(a), \quad L_a((1-a)^{\frac12} W_2^1(a)) =0\]
The solvability of these equations will be discussed later in the
proof.  Then
\[
L_{ab}\,  w_2 =  \frac{L_a W_2^0 (a)}b + \frac{ L_a((1-a)^{\frac12}
W_2^1
  (a))}{b_1} + f_2^0 = \frac{e(a)}{b} + f_2^0
\]
where
\[
\begin{split}
 f_2^0 = & ( -\partial_b^2 -2\partial_b a \partial_a -\partial_b)\frac{W_2^0 (a)}b
- (1-a)^{\frac12}W_2^1 (a)( \partial_b^2  +\partial_b) [\frac{1}{b_1}]
\\ &\ + (1-a)^{\frac12}W_2^1 (a) ((a^{-1}-2a+1)\partial_a -(1-a)^2
\partial_a^2) [\frac{1}{b_1} ]\\ &\ + 2 (1-a^2)
  \partial_a  ((1-a)^{\frac12}W_2^1 (a)) \partial_a [\frac{1}b_1]
-2a(1-a)^{\frac12} \partial_{a}W_2^1(a) \partial_b [\frac1{b_1}]
\\ &\ -2 aW_2^1(a) \partial_a \partial_b [ \frac{(1-a)^{\frac12}}{b_1}
] + (1-a)^{\frac12} W_2^1 (a) (-\partial_a +
((1-a)^2+(1-a^2))\partial_a^2)[ \frac{1}{b_1} ]
\end{split}
\]
The final term here is the same as
\begin{align*}
 & (1-a)^{\frac12} W_2^1 (a) (-\partial_a +
2(1-a)\partial_a^2)[ \frac{1}{b_1} ]\\
&=  W_2^1 (a) (-\partial_a + 2(1-a)\partial_a^2)[
\frac{(1-a)^{\frac12}}{b_1} ] +2(1-a)^{\frac12}W_2^1(a)\partial_a
[\frac{1}b_1]
\end{align*}
which implies that the error equals
\begin{align}
 f_2^0 &= ( -\partial_b^2 -2\partial_b a \partial_a -\partial_b)\frac{W_2^0 (a)}b
- (1-a)^{\frac12}W_2^1 (a)( \partial_b^2  +\partial_b)
[\frac{1}{b_1}] \nn
\\ &\ + (1-a)^{\frac12}W_2^1 (a) ((a^{-1}-2a+1)\partial_a -(1-a)^2
\partial_a^2) [\frac{1}{b_1} ]\nn \\ &\ + (1-a)^{\frac32}
(W_2^1(a)+2(1+a)
  \partial_a  W_2^1 (a)) \partial_a [\frac{1}b_1] \label{eq:f2block} \\
& -2a(1-a)^{\frac12} \partial_{a}W_2^1(a) \partial_b [\frac1{b_1}]
\nn
\\ &\ +
W_2^1 (a) (-2a \partial_b \partial_a + 2(1-a)\partial_a^2
-\partial_a)[ \frac{(1-a)^{\frac12}}{b_1} ] \label{eq:final_term}
\end{align}
In the first term we gain at least one power of $b^{-1}$. In the
second and fifth terms we gain at least one power of $b_1^{-1}$.
Since
\[
(1-a) \partial_a b_1 = - \frac{(1-a) p'(a)}{p(a)}
\]
which is analytic in $[0,1]$ it follows that in the third and fourth
terms we gain at least one power of $b_1^{-1}$ without losing any
power of~$(1-a)$.

So far we have considered the negligible terms. The key expression
is the one in the final term, which determines the choice of our
ansatz. Here there is a nontrivial cancellation which yields an
additional $1-a$ factor. To begin with, recall that
\[
(2(1-a)\partial_a^2 -\partial_a)(1-a)^{\frac12}=0
\]
This implies that in \eqref{eq:final_term} at least one derivative
has to fall on~$b_1$ leading to a gain of at least one power
of~$b_1$. However, we need to check that there is no loss in terms
of powers of~$(1-a)$. This can be seen via the factorization (we
first consider $\partial_a\partial_b$ since the difference from
$a\partial_a\partial_b$ gains a factor of $1-a$)
\begin{equation}
  \label{eq:main_fact}
  \begin{split}
  &(-2 \partial_b \partial_a + 2(1-a)\partial_a^2
-\partial_a) (1-a)^{\frac12} g(a,b) \\
&=
\big(2(1-a)^{\frac12}\partial_a-(1-a)^{-\frac12}\big)\big(-\partial_b+(1-a)\partial_a\big)g(a,b)
\end{split}
\end{equation}
provided $g(a,b)$ is smooth. In particular, setting
$g(a,b)=\frac{1}{b_1}$,
\begin{align}
 &(-2 \partial_b \partial_a + 2(1-a)\partial_a^2
-\partial_a) \frac{(1-a)^{\frac12}}{b_1} \nn\\
&=
\big(2(1-a)^{\frac12}\partial_a-(1-a)^{-\frac12}\big)\big(-\partial_b+(1-a)\partial_a\big)[\frac{1}{b_1}]\nn \\
&=
\big(2(1-a)^{\frac12}\partial_a-(1-a)^{-\frac12}\big)b_1^{-2}\big(\partial_b-(1-a)\partial_a\big)b_1\nn
\\
&= \Big(2\frac{(1-a)^{\frac12}}{b_1^{2}}\partial_a -
4\frac{(1-a)^{\frac12}}{b_1^{3}}\partial_a b_1
-\frac{(1-a)^{-\frac12}}{b_1^{2}}\Big)\big(\partial_b-(1-a)\partial_a\big)b_1\label{eq:4b1}
\end{align}
Given our choice of $b_1$,
\[
(-\partial_b +(1-a) \partial_a ) b_1 = -1- \frac{p'(a)}{p(a)}(1-a)=O(1-a)
\]
Thus, the $(1-a)$-gain in the second factor in~\eqref{eq:4b1}
cancels the $(1-a)$-loss that we incur in the first factor. At the
same time we get at least a~$b_1^{-2}$ factor. In conclusion,
\begin{equation}\label{b1cancel}
 (-2 a\partial_b \partial_a + 2(1-a)\partial_a^2
-\partial_a) \frac{(1-a)^{\frac12}}{b_1} = O(\frac{(1-a)^\frac12}{b_1^2})
\end{equation}
where the $O(\cdot)$-term here depends linearly on $b_1^{-2}$
and~$b_1^{-3}$.
 This establishes the desired estimate on the error~$f_2^0$.

We now consider the principal part, for which we need to solve
\begin{equation}\label{eq:Lainh}
L_a W_2^0 = e(a), \qquad L_a ((1-a)^\frac12 W_2^1(a)) = 0
\end{equation}
In order to analyze these equations, we first discuss fundamental
systems of $L_a$ and their respective behaviors at the regular
singular points $a=0$ and $a=1$ of $L_a$ (we can ignore the regular
singular point $a=-1$ of $L_a$). From
\[
L_a(a^k) = (k^2-4)a^{k-2} - k(k+1) a^k
\]
we conclude that $L_a [a^{\pm 2}(1+a^2\phi_{\pm}(a))]=0$ where
$\phi_{\pm}$ are even analytic functions around $a=0$. Moreover, a
particular solution to $L_a(f)=a^2$ is given by
$f(a)=-\frac{a^2}{6}$. Similarly, for any $e(a)$ as in the statement
of the lemma there is a particular solution $f(a)$ to $L_a f= e$
with $f$ even analytic around $a=0$ and vanishing quadratically at
$a=0$. Note that $f$ is not unique. However, adding a suitable
multiple of the $a^2$-homogeneous solution we can achieve that
$f(a)$ vanishes to fourth order at $a=0$ (i.e.~$f(a)=O(a^4)$) and is
unique.

To analyze a fundamental system around $a=1$ we write
\begin{align*}
L_a &= 2(1-a)^{\frac12} \partial_a ((1-a)^{\frac12} \partial_a) -
(1-a)^2 \partial_a^2 + a^{-1}(1+2a)(1-a)\partial_a - \frac{4}{a^2}\\
&=: L_{a,0} + L_{a,1}
\end{align*}
where $L_{a,0}:=2(1-a)^{\frac12} \partial_a ((1-a)^{\frac12}
\partial_a)$. Now
\begin{align*}
 L_{a,0} (1-a)^k &= k(2k-1) (1-a)^{k-1}\\ L_a (1-a)^k &=
k(2k-1) (1-a)^{k-1} + O((1-a)^k)
\end{align*}
with an analytic $O(\cdot)$-term. This implies that $L_a\psi_0=L_a
\psi_1=0$ with \begin{equation}\label{eq:psis}
\psi_0(a)=1+(1-a)\tilde\psi_0(a), \quad \psi_1(a)=
(1-a)^{\frac12}(1+(1-a)\tilde\psi_1(a)) \end{equation} where
$\tilde\psi_0,\tilde\psi_1$ are analytic around $a=1$. In
particular, we can solve for $W_2^1$ in~\eqref{eq:Lainh} and $W_2^1$
is unique up to a constant factor. For future reference we remark
that
\begin{align*}
L_a &= \rho_1 \partial_a (\rho_2\partial_a) -\frac{4}{a^2} \\
\rho_1(a) &= \frac{1}{a}\sqrt{1-a^2},\quad \rho_2(a) = a\sqrt{1-a^2}
\end{align*}
To solve \eqref{eq:Lainh}, we first solve for $W :=
W_2^0+(1-a)^\frac12 W_2^1$ and then extract $W_2^0$ and $W_2^1$ from
it. The logic here is as follows:
 At $a=0$ we want $w_2$ to vanish to
fourth order. This implies that $W$ must also vanish to the same
order since \[b_1-b=|\log p(a)|= -\log(1-O(a^M))=O(a^M)\] with $M$
large. Therefore, as discussed above, $W$ is uniquely determined  as
a solution to
\[
L_a W(a) = a^2, \quad -\eps<a<1
\]
where $\eps>0$ is some small constant. By variation of parameters
there exist unique constants $c_0, c_1, c_2$  with the property that
\[
W(a) = c_0\psi_0(a) + c_1 \psi_1(a) + c_2 \int_a^1
[\psi_0(a)\psi_1(u)-\psi_1(a)\psi_0(u)](\rho_1(u))^{-1}u^2\,du
\]
By inspection, the integral on the right-hand side is smooth around
$a=1$. This shows that we need to set
\begin{align*}
W_2^1(a) &:= c_1 \psi_1(a) \\ W_2^0(a) &:=c_0\psi_0(a) + c_2
\int_a^1
[\psi_0(a)\psi_1(u)-\psi_1(a)\psi_0(u)](\rho_1(u))^{-1}u^2\,du
\end{align*}
Observe that at $a=0$ we have no guarantee that $W_2^0, W_2^1$ are
smooth; in fact, they may exhibit $a^{-2}$-type behavior.
\end{proof}

We remind the reader that $b_1$ in \eqref{eq:w2def} cannot be
replaced with~$b$ since we require that $w_2\in H^1(0,1)$ relative
to the~$a$ variable. The proof also shows that one cannot dispense
with the $(1-a)^{\frac12}$ part of~$w_2$ since it is part of the
fundamental system of~$L_a$.
 Another important feature of the previous proof is
the cancellation in~\eqref{eq:main_fact}. For our purposes,
$g(a,b)=h(b_1)$ whence~\eqref{eq:main_fact} becomes
\begin{equation}
  \label{eq:main_fact2}
  \begin{split}
  &(-2 \partial_b \partial_a + 2(1-a)\partial_a^2
-\partial_a) (1-a)^{\frac12} h(b_1) \\
&=
\big(2(1-a)^{\frac12}\partial_a-(1-a)^{-\frac12}\big)h'(b_1)\big(-\partial_b+(1-a)\partial_a\big)
b_1\\
&= \big(2(1-a)^{\frac12}h''(b_1)\partial_a b_1 +
2(1-a)^{\frac12}h'(b_1)\partial_a
-h'(b_1)(1-a)^{-\frac12}\big)O(1-a)\\
&= O((1-a)^{\frac12}h''(b_1))+ O((1-a)^{\frac12}h'(b_1))
\end{split}
\end{equation}
In view of \eqref{eq:main_fact2}, the proof of  Lemma~\ref{lem:3.6}
generalizes  to right-hand sides such as~$\frac{e(a)}{b^k}$ for any
$k\ge1$.

If we were to now set $v_2:=w_2$ (from the previous lemma), then the
error $f_2^0$ from~\eqref{eq:Lab_err}, as well as the remaining
$c_2\frac{a^2}{b^2}$ piece from~$e_1^{00}$, would have to be
included in $e_2$. However, if $\beta>1$ this is inadmissible since
the error $e_2$ needs to decay at least
like~$(t\lambda(t))^{-2}=b^{-2\beta}$. The importance of
$(t\lambda)^{-2}$ lies with scaling; indeed, the elliptic
equation~\eqref{vkodd} scales like~$R^2$ which equals $(t\lambda)^2$
at its largest.

These are not the only obstacles we face here: the nonlinear part
of~$e_2$ (again if $v_2=w_2$) is
\begin{equation}
\label{eq:e2nonlin}
\frac{6}{r^{2}}(u_{1}^{2}-1)w_{2}+\frac{2}{r^{2}}(3w_{2}^{2}u_{1}+w_{2}^{3}),
\end{equation}
where $u_{1}=Q+v_{1}$, see~\eqref{eeven}. One easily checks that the
preceding expression times $t^{2}$ lies in \[(\lambda
t)^{-2}IS^{4}(1, \cQ_{1})+(\lambda t)^{-2}IS^{4}(R^{2}, \cQ_{2}).\]
The term $(\lambda t)^{-2}IS^{4}(1, \cQ_{1})$ can be incorporated
into $t^{2}e_{2}$; however, the term
\[(\lambda t)^{-2}IS^{4}(R^{2},
\cQ_{2})\]
is not acceptable for $e_2$ due to the $R^2$ growth.

We deal with these obstacles by including all unacceptable
errors~$e$ (with regard to~$e_2$) in $e_1^0$ and solving $L_{ab}\,
w=e$. For example, using the notation of Lemma~\ref{lem:3.6} the
second term in~\eqref{eq:e2nonlin} contributes
\[
e= a^{-2} \frac{(1-a)^{\frac12}W_2^0(a) W_2^1(a)}{bb_1}
\]
where we replaced $u_1$ with $1$. The corresponding ansatz for~$w$
would then necessarily contain the term
\[
w= (bb_1)^{-1}(1-a)^{\frac12} W(a)
\]
If $\partial_a\partial_b$ (which is part of $L_{ab}$) hits this
term, then we obtain (amongst others) the error term
\[
(1-a)^{-\frac12} b^{-1}b_1^{-2}
\]
Iterating once more with this error on the right-hand side produces
the expression
\[
(1-a)^{\frac12}b_1^{-2}\log b
\]
In order to remove possible singularities at $a=0$ (as in the
previous proof) one needs as many powers of $\ln b_1$ as of $\ln b$.
These observations should serve to motivate the following result
which will finally allow us to carry out the full iteration
for~$v_2$ (as well as for $v_{2k}$ in Step~4 below). We begin with a
definition.

\begin{defi}
  \label{def:Fk}
Let $2k\ge m\ge k\ge1$.
  By $\calF_{k,m}$ we mean the function class
  \begin{align*}
    \calF_{k,m} &:= \Big\{f_k\;\Big|\; f_k=\ b^{-k} e_0( a,\ln b)     +  (1-a)^{\frac12} \sum_{j=1}^m
 e_{j}^0( a, \ln b,\ln b_1)  b^{j-k} b_1^{-j}
\\ &\ +  (1-a)^{-\frac12} \sum_{j=1}^m
  e_{j}^1( a, \ln b,\ln b_1)   b^{j-k-1} b_1^{-j} \Big\}
  \end{align*}
where for each $j$ the functions $e_0$, $e_{j}^0(a)$, $e_j^1(a)$ are
smooth in $a\in(0,1)$, analytic around $a=1$, meromorphic and even
around $a=0$. Moreover, these functions are polynomials in the
variables~$\log b$, and $\log b_1$, respectively. Further,
$f_k=O(a^2)$  as $a\to0$.
\end{defi}

Recall that the order of the pole at $a=0$ is controlled by a
constant depending only on~$k$. In what follows, we will tacitly
assume that $M$ in the definition of~$b_1$ is sufficiently large
depending on~$k$ (in fact, the order of the pole at $a=0$ in the
previous definition). Since we are only going to consider finitely
many~$k$, this is not  an issue.  Since $\log b_1-\log b= O(a^M)$ we
see that $f_k(a)=O(a^2)$ is therefore the same as
\[
e_0( a,\ln b) + \sum_{j=1}^m
 (1-a)^{\frac12} e_{j}^0( a, \ln b,\ln b) + (1-a)^{-\frac12} \sum_{j=1}^m
  b^{-1}e_{j}^1( a, \ln b,\ln b) = O(a^2)
\]
The left-hand side is a polynomial in $\ln b$, $b^{-1}$, so this
amounts to the corresponding condition for each of its coefficients.
 Now for the main iterative lemma.

\begin{lemma}\label{advancedode}
  The equation
\begin{equation}\label{labv}
 L_{ab}\,  v = f_k \in\calF_{k,m}
\end{equation}
admits an approximate solution
\[
v(a,b) =  b^{-k} V_0( a,\ln b)     +  (1-a)^{\frac12} \sum_{j=1}^m
V_{j}( a,\ln b,\ln b_1)  b^{j-k} b_1^{-j}
\]
where $V_0$, $V_j$ are smooth in $a\in(0,1)$, analytic around $a=1$,
meromorphic around $a=0$,  and polynomial in the variables  $\log
b$, $\log b_{1}$. Moreover,
 $v$ vanishes to fourth order at $a=0$ and
\[
 L_{ab}\,  v - f_k  \in\calF_{k+1,m}+\calF_{k+2,m}
\]
\end{lemma}

\begin{proof} Let $\ell(0)$ be the order of the polynomials appearing in the definition of~$f_k$ relative
to~$\log b$, and $\ell(j)$ the order relative to~$\log b_1$ with
$1\le j\le m$. We first re-write the source term: choose a smooth
partition of unity $\phi_{1,2}(a)$, subordinate to the cover
$(0,1)=(0,2\eps)\cup (\eps, 1)$ for some small $\eps>0$. Then write
\begin{align}
 & \phi_{1}(a)[b^{-k} e_0( a,\ln b)     +  (1-a)^{\frac12} \sum_{j=1}^m
 e_{j}^0( a, \ln b,\ln b_1)  b^{j-k} b_1^{-j}\nn
\\ &+  (1-a)^{-\frac12} \sum_{j=1}^m
  e_{j}^1( a, \ln b,\ln b_1)   b^{j-k-1} b_1^{-j}]\nn
  \\&=\phi_{1}(a)[b^{-k} e_0( a,\ln b)     +  (1-a)^{\frac12} \sum_{j=1}^m
 e_{j}^0( a, \ln b,\ln b)  b^{-k}\label{eq:brackets}
\\ & +  (1-a)^{-\frac12} \sum_{j=1}^m
  e_{j}^1( a, \ln b,\ln b)   b^{-k-1}]
  \label{eq:brackets'}  \\&
   +(\log b-\log
   b_{1})\phi_{1}(a)\tilde{f}_{k}+(b-b_{1})\phi_{1}(a)\tilde{g}_{k+1}\label{eq:bb1diff}
\end{align}
where $\tilde{f}_{k}, \tilde{g}_{k}$ have the same properties as
$f_{k}$. Note that in the expression in brackets
in~\eqref{eq:brackets} and~\eqref{eq:brackets'}, all singular powers
cancel. For \eqref{eq:bb1diff}, expand
\[
\phi_{1}(a)[\log b_1-\log b]=\phi_{1}(a)\log
(1-\frac{\log|p(a)|}{b})=-
\phi_{1}(a)\sum_{j=1}^{N}[\frac{(\frac{\log|p(a)|}{b})^{j}}{j}]+\text{error}
\]
Here we may achieve arbitrarily fast decay in time for the error term upon choosing $N$
large enough, and hence we can discard its contribution. However, now all the terms in
\[
\phi_{1}(a)(\log|p(a)|)^{j}\tilde{f}_{k}, \quad
\phi_{1}(a)(\log|p(a)|)^{j}\tilde{g}_{k+1},\quad j\geq 1,
\]
are smooth up to $a=0$, and so are all terms in
\[
\begin{split}
\phi_{2}(a)f_{k}=\phi_{2}(a)[&\ b^{-k} e_0( a,\ln b)     +  (1-a)^{\frac12} \sum_{j=1}^m
 e_{j}^0( a, \ln b,\ln b_1)  b^{j-k} b_1^{-j}
\\ &\ +  (1-a)^{-\frac12} \sum_{j=1}^m
  e_{j}^1( a, \ln b,\ln b_1)   b^{j-k-1} b_1^{-j}]
\end{split}
\]
These considerations show that we may as well assume that  $e_{0}$,
$e_{j}^{0}$, $e_{j}^{1}$ are each analytic at $a=0$ as well as of
the form $O(a^2)$. With $v$ as in the statement of the lemma, we
compute
\[
\begin{split}
L_{ab}\,  v =& \ b^{-k} L_a V_0 (a,\ln b) +  \sum_{j=1}^m
 b^{j-k} b_1^{-j} L_a
((1-a)^{\frac12} V_{j}( a,\ln b,\ln b_1)) \\
& +     \sum_{j=1}^m  a\partial_b(  b^{j-k}  b_1^{-j}
(1-a)^{-\frac12} V_{j}( a,\ln b,\ln b_1))
  + \text{error}
\end{split}
\]
where $b_1$ is treated as a parameter, i.e., no derivatives fall on
it. Here the last term comes from $\partial_a
\partial_b$ in $L_{ab}\, $ with the $\partial_a$ applied to
$(1-a)^{\frac12}$ and $\partial_b$ applied to $b$ or~$\log b$.
Assuming that $V_j$ are smooth and that $v$ vanishes of order four
at $a=0$ one sees that the error has the desired form
\[
\text{error} \in  a^2 \cQ_{k+1}
\]
This is done  using the same type of calculations leading
to~\eqref{eq:f2block} and the following properties,
cf.~\eqref{eq:main_fact2},
\begin{align*}
 (-2 \partial_b \partial_a + 2(1-a)\partial_a^2
-\partial_a) \frac{(1-a)^{\frac12}}{b_1^k} &=
O\Big(\frac{(1-a)^\frac12}{b_1^{k+1}}  \Big)\\
(-2\partial_{a}\partial_{b}+2(1-a)\partial_{a}^{2}-\partial_{a})[(1-a)^{\frac{1}{2}}(\log
b_{1})^{k}] &=O(\frac{(1-a)^{\frac{1}{2}}(\log
b_{1})^{k-1}}{b_{1}^{2}})
\end{align*}
We also observe that in the second sum in $L_{ab}\,  v$ only
$V_{j}(1,\log b,\log b_1)$ is important. The rest can be also
assigned to the error. Thus matching the $(1-a)^{\frac12}$ like
terms we are left with the equations
\[
L_a V_0 (a,\ln b) = e_0(a,\ln b)
\]
\[
 L_a  (V_{j}(a,\ln b,\ln b_1)
(1-a)^{\frac12}) = e_{j}^0(a,\ln b,\ln b_1)  (1-a)^{\frac12}
\]
Matching the $(1-a)^{-\frac12}$  at $a=1$ we get the boundary conditions
(recall that $b_{1}$ here is treated as a parameter)
\begin{equation}\label{bdry}
 \partial_b(  b^{j-k}  V_{j}(1,\ln b,\ln b_1)) =
 b^{j-k-1} e_{j}^1( 1,\ln b,\ln b_1),  \qquad j = 1,\ldots, m
\end{equation}
More explicitly, \eqref{bdry} means the following. Separating into
monomials in $\log b_1$ we seek $s'$ and $\{c_\ell\}_{\ell=0}^{s'}$
so that
\[
\partial_b \Big( b^{j-k} \sum_{\ell=0}^{s'} c_\ell \log^\ell b\Big) =
b^ {j-k-1} \sum_{\ell=0}^s c_\ell^0 \log^\ell b
\]
for given $s$ and $\{c_\ell^0\}_{\ell=0}^s$.  If $j>k$ then we set
$s':=s$ and
\begin{align*}
  (j-k)c_\ell + (\ell+1)c_{\ell+1} &= c_\ell^0 \qquad 0\le\ell<s\\
  c_s &= \frac{c_s^0}{j-k}
\end{align*}
whereas in case $j=k$ we set $s':=s+1$ and
$c_\ell=\frac{c_{\ell-1}^0}{\ell}$ for all $1\le\ell\le s'$ (in
particular, we generate extra powers of $\log b$ in this case and
$c_0$ is not determined). Write
\[ e_{0}(a, \log b)=\sum_{j=0}^{\ell(0)}P_{j}(a)\log^{j} b\]
with $P_{j}(a)$ is smooth on $[0,1]$, analytic close to~$a=0$, and
$P_j(a)=O(a^2)$. Then we solve the problems
\[ L_{a} V_{0,j}=P_{j},\;\;j=0,\ldots,\ell(0)\]
where we select a solution which is smooth at $a=1$. Using the
notations of~\eqref{eq:psis} and variation of parameters,
\[
V_{0,j}(a)= c\psi_0(a) + c_0\int_a^1
[\psi_0(a)\psi_1(u)-\psi_0(u)\psi_1(a)](\rho_1(u))^{-1}P_j(u)\,du
\]
where $c_0$ is an absolute, and $c$ an arbitrary, constant. Note
that around $a=0$,
\[
V_{0,j}(a)=O(a^2) + c_{0,j}\,\varphi_0(a)
\]
where $L_a\, \varphi_0=0$ and $\varphi_0(a)=a^{-2}(1+O(a^2))$ with
analytic $O(a^2)$ (as can be seen from a power series ansatz).
 Then
define
\[V_{0}(a,\log b):= \sum_{j=0}^{\ell(0)}V_{0,j}(a)\log^{j} b.\]
Even though this expression will in general be singular at $a=0$,
the singular part is of the form
\[
\varphi_0(a)\sum_{j=0}^{\ell(0)}c_{0,j}\log^{j} b
\]
Similarly, we write
\[ e_{j}^{0}(a, \log b, \log b_{1})=
\sum_{k=0}^{\ell(j)}\sum_{\ell+n=k}q_{j,\ell, n}(a) \log^\ell b\;
\log^n b_{1}\] where $q_{j,\ell, n}$ are smooth, analytic around
$a=0$ and vanishing to second order at $a=0$,  and solve the
problems
\[L_{a} [(1-a)^{\frac{1}{2}}V_{j,\ell,n}(a)]= (1-a)^{\frac{1}{2}}q_{j,\ell, n}(a)\]
by variation of parameters, i.e.,
\begin{align*}
&(1-a)^{\frac12}V_{j,\ell,n}(a) = c_{j,\ell,n}\,\psi_1(a)\\
&\quad + c_0\int_a^1
[\psi_0(a)\psi_1(u)-\psi_0(u)\psi_1(a)](\rho_1(u))^{-1}(1-u)^{\frac{1}{2}}q_{j,\ell,
n}(u)\,du \end{align*} where $c_{j,\ell,n}$ is arbitrary.  Note that
$V_{j,\ell,n}(a)$ is smooth around $a=1$. As for the behavior around
$a=0$, one has \[
(1-a)^{\frac12}V_{j,\ell,n}(a)=O(a^2)+c\varphi_0(a)\] as before.
Moreover, since
\[
(1-a)^{-\frac12}\psi_1(a) = 1+O(1-a)
\]
we conclude that $V_{j,\ell,n}(1)$ can be assigned arbitrary values.
This is crucial with regard to the boundary condition~\eqref{bdry}.
More precisely, setting
\[V_{j}(a,\log b, \log b_{1}) := \sum_{k=0}^{\ell(j)}\sum_{\ell+n=k}V_{j,\ell,n}(a)\log^\ell b\; \log^n b_{1}\]
we can satisfy the boundary condition~\eqref{bdry} at $a=1$.
Generally speaking, the approximate solution
\[V_{sing}(a,b):=b^{-k}V_{0}(a,\log
b)+(1-a)^{\frac{1}{2}}\sum_{j=1}^{m}V_{j}(a, \log b, \log
b_{1})b^{j-k}b_{1}^{-j}\] will not be smooth at the origin $a=0$,
let alone vanish to fourth order. To remedy this problem, we
subtract the correction function
\[
\tilde{V}(a,b_{1}):=b_{1}^{-k}\tilde{V}_{0}(a,\log b_{1})+(1-a)^{\frac{1}{2}}\sum_{j=1}^{m}
\tilde{V}_{j}(a, \log b_{1}, \log b_{1})b_{1}^{-k}
\]
which solves the homogeneous equation
$L_{a}\tilde{V}\in\calF_{k+1}+\calF_{k+2}$ and has the same singular
behavior at $a=0$ as $V_{sing}$. More precisely, we first set
$b_1=b$ in $V_{sing}(a,b)$ and write the resulting expression in the
form
\[
b^{-k}\sum_{\nu} V_\nu(a) \log^\nu b_1
\]
In view of our discussion regarding the singularity at $a=0$, we see
that
\[
V_\nu(a)=c_\nu\varphi_0(a)+c_\nu'\varphi_1(a)+O(a^4)
\]
where $O(a^4)$ is analytic and $\varphi_1$ is the regular
homogeneous solution, i.e., $L_a\varphi_1=0$,
$\varphi_1(a)=a^2(1+O(a^2))$.  Hence, we see that
\[
\tilde{V}(a,b_{1}):= b_1^{-k} \sum_{\nu} (c_\nu\,\varphi_0(a)
+c_\nu'\,\varphi_1(a)) \log^\nu b_1
\]
has the desired properties, i.e., \[v:=V_{\rm sing}-\tilde V\]
vanishes to fourth order at $a=0$. Finally, as above one checks that
\[ L_{ab}\, \tilde{V}\in \cQ_{k+1},\]
which therefore is an error. Finally, the error $f_{k+1}+f_{k+2}$
generated by this entire procedure vanishes at least to second order
at the origin as claimed.
\end{proof}

By design, Lemma~\ref{advancedode} allows for arbitrary many
iterations. Therefore, we can now carry out the process leading
to~$v_2$ as explained above, see~\eqref{eq:e10}. At each step we
gain a power of $b^{-1}$ or $b_{1}^{-1}$, while paying at most one
power of $\log b$ and~$\log b_{1}$.  We iterate sufficiently often,
and let
\[
v_{2}=w_{2}+w_{3}+\ldots
\]
By construction $v_2$ vanishes of order four at $a=0$, therefore
we can factor out an $a^4$ to obtain
\[
v_2 \in a^4 IS(1,\cQ_1)
\]
Recalling also that we have neglected terms of the form $(\lambda(t)
t)^{-2}IS^{4}(1)$, we find that the remaining error satisfies
\[
t^{2}e_{2}\in a^{2}IS^{4}(1, \cQ_{2\beta}')+IS^{4}(1,\cQ_{2\beta})
\]
as desired.

\medskip

{\bf Step 3:}  We  now consider the general setup. {\it{Commence
with $e_{2k}$, $k\geq 1$, satisfying \eqref{e2k} and choose
$v_{2k+1}$ so that \eqref{v2k-1}, \eqref{e2k-1} hold with $k$
replaced by $k+1$.}} Note that we can move that  part of $e_{2k}$
which belongs to
\[
a^{2}IS^{4}((\log R)^{k-1}, \cQ'_{2\beta k})
\]
into the next error, $e_{2k+1}$. Hence we only need to deal with the
part of $e_{2k}$ in
\[
IS^{4}((\log R)^{k-1}, \cQ_{2\beta k}),
\]
which we denote as $e_{2k}^{0}$. Proceeding as in Step~1, we then
set
\begin{equation}\begin{split}
(t\lambda)^{2}v_{2k+1}(R,a,b,b_{1})=& \
\Theta(R) \int_0^R \Phi(R')t^{2}e_{2k}^{0}(R',a,b,b_{1}) R'\, dR'\\
&\  - \Phi(R)\int_1^R
  \Theta(R') t^{2}e_{2k}^{0}(R',a,b,b_{1}) R'\, dR'  \nn
  \end{split}\end{equation}
Here we treat $a, b, b_{1}$ as constant parameters. Then it is clear that
\[v_{2k+1}\in IS^{4}(R^{2}(\log R)^{k}, \cQ_{2\beta(k+1)})\]
We need to check that the error satisfies \eqref{e2k-1} for $k+1$
instead of $k$. The error is comprised of the terms arising from
$\partial_{t}^{2}$, when one of the variables $a, b, b_{1}$ is
differentiated, as well as the nonlinear terms. More precisely, we
write
\[
e_{2k+1}=N_{2k+1}(v_{2k+1})+E^{t}v_{2k+1}+E^{a}v_{2k+1}
\]
where the first represents nonlinear errors, the second represents
$\partial_{t}^{2}v_{2k+1}$, and the third represents those
constituents in
\[ (-\partial_{t}^{2}+\partial_{r}^{2}+\frac{1}{r}\partial_{r})v_{2k+1}(R, a,b,b_{1})\]
in which at least one derivative falls on  $a$, or $b$, or $b_{1}$.
It is straightforward to check that
\[t^{2}E^{t}v_{2k+1}\in IS^{4}(R^{2}(\log R)^{k}, \cQ_{2\beta(k+1)})\subset IS^{4}(R^{2}(\log R)^{k}, \cQ_{2\beta(k+1)}')\]
Next, the terms in $t^{2} E^{a}v_{2k+1}$ are of the form
\[[(1-a^{2})\partial_{a}^{2}+(a^{-1}-2a)\partial_{a}]v_{2k+1}(R, a,b,b_{1})\]
\[[(1-a^{2})\partial_{a}+(a^{-1}-2a)](\partial_{a}b_{1}\partial_{b_{1}}v_{2k+1}(R, a,b,b_{1}))\]
\[(1-a^{2})t\partial_{t}aR \partial_{R}\partial_a v_{2k+1}(R,a,b,b_{1})-(1-a^{2})a^{-1}R\partial_a\partial_R v_{2k+1}(R, a,b,b_{1})\]
Each of these is easily seen to be in $ IS^{4}(R^{2}(\log R)^{k}, \cQ_{2\beta(k+1)}')$.
The nonlinear errors are of the form
\[
\frac{6}{r^{2}}(u_{2k}^{2}-Q^{2})v_{2k+1}+\frac{2}{r^{2}}(3v_{2k+1}^{2}u_{2k}+v_{2k+1}^{3})
\]
For the term on the left, expand $u_{2k}=Q+\sum_{1\leq i\leq
2k}v_{i}$. Using that
\[
\sum_{1\leq i\leq 2k}v_{i}\in IS^{4}(R^{2}, \cQ_{2\beta}),
\]
we check that
\[ t^{2}\frac{6}{r^{2}}(u_{2k}^{2}-Q^{2})v_{2k+1}\in  IS^{4}(R^{2}(\log R)^{k}, \cQ_{2\beta(k+1)})\]
Similarly, we get
\[ \frac{2t^{2}}{r^{2}}(3v_{2k+1}^{2}u_{2k}+v_{2k+1}^{3})\in  IS^{4}(R^{2}(\log R)^{k}, \cQ_{2\beta(k+1)})\]

\medskip

{\bf Step 4}
 {\it{Commence with $e_{2k-1}$, $k\geq 1$, satisfying \eqref{e2k-1} and choose $v_{2k}$ so that \eqref{v2k}, \eqref{e2k} hold.}}
 Pick the leading order term in $e_{2k-1}$, which can be written as
 \[
 t^2 e_{2k-1}^{0}:=R^{2}\sum_{j=0}^{k-1}g_{j}(a,b,b_{1})(\log R)^{j},
 \]
 with $g_{j}(a)\in \cQ_{2\beta k}'$. We then claim that the error
 $e_{2k-1}^{1}:=e_{2k-1}- e_{2k-1}^{0}$ can be absorbed into $e_{2k}$. Indeed, we can write
 \[
 e_{2k-1}^{1}=a^{2}e_{2k-1}^{1}+(1-a^{2})e_{2k-1}^{1},
 \]
 and we have $(1-a^{2})\cQ_{2\beta k}'\subset \cQ_{2\beta k}$. Next, rewrite
 \[
  t^{2}e_{2k-1}^{0}=\sum_{j=0}^{k-1}h_{j}(a,b,b_{1})(\log R)^{j},
  \qquad
h_{j}(a,b,b_{1}) = a^2 g_j(a,b,b_1) \in
  a^2\cQ_{2(k-1)\beta}'
  \]
  We first seek an approximate solution $w_{2k}$ for \eqref{vkeven} of
  the form
 \[
w_{2k}= \sum_{j=0}^{k-1} z_{j} (a,b,b_{1}) (\log R)^{j}, \qquad 
z_{j} \in  a^4\cQ_{2(k-1)\beta}
\]
This we then refine, iterating application of
Lemma~\ref{advancedode} sufficiently often to obtain~$v_{2k}$.
 To find the functions $z_j$ we proceed inductively,
 starting with the largest power of $\log R$.
 Indeed, matching corresponding powers of $\log R$, we get a recursive
 system. Denoting 
\[
L^\infty
:=t^{2}\Big(-\partial_{t}^{2}+\partial_{r}^{2}+\frac{1}{r}\partial_{r}-\frac{4}{r^{2}}\Big)
\]
we calculate
\[
\begin{split}
L^\infty w_{2k}(a,b) = & \ \sum_{j=0}^{k-1} \Big\{ \, (\log R)^{j} L^\infty z_{j} 
- 2 (t \partial_t) z_{j}  \, (t \partial_t) (\log R)^{j}
+2 (t \partial_r) z_{j}   (t \partial_r) (\log R)^{j}
\\ & 
\qquad +t^{2}z_j(-\partial_{t}^{2}+\partial_{r}^{2}+\frac{1}{r}\partial_{r})
(\log R)^{j} \Big\}
\\ 
=  &\ \sum_{j=0}^{k-1} \Big\{(\log
R)^{j} L_{ab} z_{j}   + j  (\log
R)^{j-1}  L_{ab}^1  z_{j} + j(j-1) (\log R)^{j-2}  L_{ab}^2  z_{j}
\Big\}
\end{split}
\]
where
\begin{align*}
L_{ab}^1 &=  -2\Big(1+\frac{\beta}{b}\Big)
\Big( a\partial_a + \partial_b + \big(1-\frac{ap'(a)}{p(a)}\big)\partial_{b_1}\Big) 
+2a^{-1}\Big(\partial_a - \frac{p'(a)}{p(a)}\partial_{b_1}\Big)\\
&\quad -1-\frac{\beta}{b}+\frac{\beta}{b^2}  \\ 
L_{ab}^2 &=  \Big(1+\frac{\beta}{b}\Big)^2 + a^{-2}
\end{align*}
This leads to the recursive system for $0 \leq j \leq k-1$,
\begin{equation}
 L_{ab} z_{j} = h_j - (j+1) L_{ab}^1 z_{j+1} - (j+1)(j+2)
 L_{ab}^2 z_{j+2},
\qquad z_k = z_{k+1} = 0
\label{zsyst}\end{equation}
Since $h_j \in a^2\cQ_{2(k-1)\beta}'$ and we seek
approximate solutions $z_j \in  a^4\cQ_{2(k-1)\beta}$, it suffices
to take  the principal part of the system \eqref{zsyst}, namely
\begin{align*}
 L_{ab} z_{j} &= h_j + (j+1)(1+2(a-a^{-1})\partial_{a})z_{j+1} - (j+1)(j+2)(1+a^{-2})z_{j+2}\\
z_k &= z_{k+1} = 0
\label{zsysta}\end{align*}
For this we apply Lemma~\ref{advancedode} to obtain 
approximate solutions $z_j  \in  a^4\cQ_{2(k-1)\beta}$ with lower
order errors 
\[
 L_{ab} z_{j} -( h_j + (j+1)(1+2(a-a^{-1})\partial_{a} z_{j+1}) \in   a^2\cQ_{2(k-1)\beta+1}'
\]
The other terms on the right hand side of \eqref{zsyst} have a similar
form,
\[
 (j+1) (L_{ab}^1+1+2(a-a^{-1})\partial_{a} ) z_{j+1} - (j+1)(j+2)
 L_{ab}^2 z_{j+2} \in  a^2\cQ_{2(k-1)\beta+1}'
\]
In addition to the above error terms, by adding $w_{2k}$ to the
approximate solution we have also generated errors from the nonlinear
terms, which we recall are (upon multiplication by $t^{2}$)
\[
\frac{6t^{2}}{r^{2}}(u_{2k-1}^{2}-1)w_{2k}+\frac{2t^{2}}{r^{2}}(3w_{2k}^{2}u_{2k-1}+w_{2k}^{3})
\]
where $u_{2k-1}=Q+v_{1}+\ldots+v_{2k-1}$. We expand the first term
here in the form
\[
\frac{t^{2}}{r^{2}}(Q-1+v_{1} +\ldots+v_{2k-1})(Q+1+
v_{1} +\ldots+v_{2k-1})w_{2k}
\]
with $v_1 = v_{10} + v_{11}$.
First we write
\[
\frac{t^{2}}{r^{2}}(Q+1)(Q-1)w_{2k}=a^{-4}\frac{R^{2}}{1+R^{2}}\frac{1}{(t\lambda)^{2}}w_{2k}\in IS^{4}((\log R)^{k-1}, \cQ_{2k\beta}),
\]
which we can absorb into $e_{2k}$.  The terms
\[
\frac{t^{2}}{r^{2}}(Q\pm 1) v_{10} w_{2k}
\]
are similar but simpler. On the other hand we recall from Step~1 that
$v_{11}$ satisfies
$v_{11}\in IS^{4}(R^{2},\cQ_{2\beta+1})$. Hence we obtain
\[
\frac{t^{2}}{r^{2}}(Q-1)v_{1}w_{2k}\in a^{2}IS^{4}((\log R)^{k-1}, \cQ_{2(k-1)\beta+1})\subset IS^{4}((\log R)^{k-1}, \cQ_{2(k-1)\beta+1}'),
\]
which we cannot absorb into $e_{2k}$ yet, whence we iteratively
apply the preceding procedure to it. The remaining interactions
satisfy at least
\[
\frac{t^{2}}{r^{2}}(Q\pm 1) v_{j}w_{2k},\,\frac{t^{2}}{r^{2}}v_{i}v_{j}w_{2k}\in  IS^{4}((\log R)^{k-1}, \cQ_{2(k-1)\beta+2}'),\]
and we re-iterate the preceding procedure for those which cannot yet be absorbed into $e_{2k}$.
We similarly deduce
\[
\frac{2t^{2}}{r^{2}}(3w_{2k}^{2}u_{2k-1}+w_{2k}^{3})\in IS^{4}((\log R)^{k-1}, \cQ_{2k\beta}),
\]
which can therefore be absorbed into $e_{2k}$. We now re-iterate
(sufficiently often) the procedure from the beginning of the present
step for those errors which cannot yet be absorbed into $e_{2k}$,
resulting in $w_{2k}=w_{2k}^{0}$, $w_{2k}^{1},\ldots,
w_{2k}^{2\beta}$. Finally,  $v_{2k}:=\sum_{j=0}^{2\beta}w_{2k}^{j}$
has all the desired properties.
\end{proof}

 \section{The analysis of the underlying strongly singular Sturm-Liouville
  operator}
\label{sec:spec}

In this section we develop the scattering and spectral theory of the
linearized operator $\cL$. The main tool developed in this section,
which is crucial to this paper, is the distorted Fourier transform.
The main difference between the linearized operator in~\cite{KST1}
and the one of this paper is that in~\cite{KST1} the linearized
operator had a zero energy resonance and here zero is an eigenvalue.
In both instances, though, there is no negative spectrum (unlike the
semi-linear case~\cite{KST2}, where we had to deal with a negative
eigenvalue and the resulting exponential instabilities).

\begin{defi}
  The  half-line operator
  \[ \cL := -\frac{d^2}{dR^2} + \frac{15}{4R^2} - \frac{24}{(1+R^2)^2} \] on
  $L^2(0,\infty)$ is self-adjoint with domain
  \[ \dom(\cL)=\{ f\in L^2((0,\infty))\::\: f,f'\in
  AC_{{\rm loc}}((0,\infty)),\; \cL f \in L^2((0,\infty)) \}
  \]
\end{defi}

Because of the strong singularity of the potential at $R=0$ no
boundary condition is needed there to insure self-adjointness.
Technically speaking, this means that $\cL_0$ and $\cL$ are in the
{\em limit point case} at $R=0$, see Gesztesy, Zinchenko~\cite{GZ}.
We remark that $\cL_0$ and $\cL$ are in the limit point case at
$R=\infty$ by a standard criterion (sub-quadratic growth of the
potential).

\begin{lemma}
  \label{lem:spec} The spectrum of $\cL$ is purely
  absolutely continuous and equals $\spec(\cL)=[0,\infty)$.
\end{lemma}
\begin{proof}
  That $\cL$ has no
  negative spectrum follows from
\begin{equation}
  \label{eq:reson}
  \cL \phi_0 =0, \qquad \phi_0(R)= \frac{R^{5/2}}{(1+R^2)^2}
\end{equation}
with $\phi_0$ positive (by the Sturm oscillation theorem). The
purely absolute continuity of the spectrum of $\cL$ is an immediate
consequence of the fact that the potential of $\cL$ is integrable at
infinity.
\end{proof}

We now briefly summarize the results from~\cite{GZ} relevant for our
purposes, see Section~3 in their paper, in particular Example~3.10.

\begin{theorem}\label{thm:GZ}
a) For each $z \in \Compl$ there exists a fundamental system
$\tilphi(R,z)$, $\tiltheta(R,z)$ for $\cL-z$ which is analytic in
$z$ for each $R > 0$ and has the asymptotic behavior
\begin{equation}\label{eq:phitheta}
\tilphi(R,z) \sim R^\frac52, \qquad \tiltheta(R,z) \sim \frac14
R^{-\frac32} \text{\ \ as\ \ }R\to0
\end{equation}
In particular, their Wronskian is
$W(\theta(\cdot,z),\phi(\cdot,z))=1$ for all $z\in\Compl$. We remark
that $\phi(\cdot,z)$ is the Weyl-Titchmarsh solution\footnote{Our
$\phi(\cdot,z)$ is the $\tilde\phi(z,\cdot)$ function from~\cite{GZ}
where the analyticity is only required in a strip around~$\R$ -- but
here there is no need for this restriction.} of $\cL-z$ at $R=0$. By
convention, $\phi(\cdot,z), \theta(\cdot,z)$ are real-valued for
$z\in\R$.

b) For each  $ z\in\Compl$, $\Im z>0$, let $\psi^+(R,z)$ denote the
Weyl-Titchmarsh solution of $\cL-z$ at $R=\infty$ normalized so that
\[
\psi^+(R,z) \sim z^{-\frac14}\,e^{i z^{\frac12} R} \text{\ \ as\ \
}R\to \infty,\;\Im z^{\frac12}>0
\]
If $\xi>0$, then the limit $\psi^+(R,\xi+i0)$ exists point-wise for
all $R>0$ and we denote it by $\psi^+(R,\xi)$. Moreover, define
$\psi^-(\cdot,\xi):=\overline {\psi^+(\cdot,\xi)}$. Then
$\psi^+(R,\xi)$, $\psi^-(R,\xi)$ form a fundamental system of
$\cL-\xi$ with asymptotic behavior $\psi^\pm(R,\xi)\sim
\xi^{-\frac14}\,e^{\pm i\xi^{\frac12} R}$ as $R\to\infty$.

c) The spectral measure of $\cL$ is  given by
\begin{equation} \label{sm}
\mu(d\xi )  =  \|\phi_0\|_2^{-2}\delta_0+ \rho(\xi)\, d\xi,\quad
\rho(\xi):=\frac{1}{\pi} \Im\; \tilm(\xi +i0)\,\chi_{[\xi>0]}
\end{equation}
with the ``generalized Weyl-Titchmarsh" function
\begin{equation}\label{mw}
  \tilm(\xi) =\frac{W({\theta}(.,\xi),\,\psi^{+}(.,\xi))}{W(\psi^{+}(.,\xi),
    {\phi}(.,\xi))}
\end{equation}

d) The distorted Fourier transform defined as
  \begin{equation}\nonumber
    \calF: f\longrightarrow \hat{f}(\xi)=\lim_{b\rightarrow
      \infty}\int_{0}^{b}\tilphi(R,\xi)f(R)\,dR
  \end{equation} for all $\xi\ge 0$
  is a unitary operator from $L^2 (\R^+)$ to
  $L^{2}({{\R^+}},{\mu})= \R\oplus L^{2}({{\R^+}},\rho) $ and its inverse is given by
 \begin{equation}\label{eq:inverse}
    \calF^{-1}: \hat{f}\longrightarrow f(R)= \hat{f}(0)\|\phi_0\|_2^{-2}\phi_0(R)+ \lim_{s \rightarrow
      \infty}\int_{0}^{s} \tilphi(R,
    \xi)\hat{f}(\xi)\,{\rho}(\xi)\,d\xi
  \end{equation}
Here $\lim$ refers to the $L^{2}({{\R^+}},\mu)$, respectively the
$L^2(\R^+)$, limit.
\end{theorem}

\begin{remark}\label{rem:vector}
It is best to view the distorted Fourier transform of any $f\in L^2
(\R^+)$ as a vector, namely $f\mapsto \binom{a}{g(\cdot)}$ where
$a\in\R$ and $g\in L^{2}({{\R^+}},\rho)$. The inversion formula
being
\[
f = a \|\phi_0\|_2^{-2}\phi_0+ \int_{0}^{\infty} \tilphi(\cdot,
    \xi)g(\xi)\,{\rho}(\xi)\,d\xi
\]
The first term is the projection of $f$ onto $\phi_0$, whereas the
second  one is the projection onto the orthogonal complement
of~$\phi_0$. We remark that
\[
\|\phi_0\|_2^2 = \int_0^\infty \frac{R^5}{(1+R^2)^4}\, dR = \frac16
\]
\end{remark}

\subsection{Asymptotic behavior of $\tilphi$ and $\tiltheta$}

Beginning with two explicit solutions for $\cL f=0$, namely
\[
\phi_0(R) = \frac{R^\frac52}{(1+R^2)^2}, \qquad \theta_0(R) =
\frac{-1-8R^2 + 24 R^4 \ln R +8 R^6 + R^8}{4R^\frac32(1+R^2)^2}
\]
we construct power series expansions for $\tilphi$
from~\eqref{eq:phitheta} in $z \in\Compl$ when $R>0$ is fixed. A
similar expansion is possible for $\tiltheta(R,z)$. Since is it not
only more complicated but also not needed here, we skip it.

\begin{proposition}  \label{pphitheta} For any $z\in\Compl$
the solution $\tilphi(R,z)$ from Theorem~\ref{thm:GZ} admits an
absolutely convergent asymptotic expansion
\begin{align*}
\tilphi(R,z) &=  \phi_0(R) + R^{-\frac32} \sum_{j=1}^\infty   (R^2
z)^{j} \tilde \phi_j(R^2)
\end{align*}
The functions $\tilde \phi_j$ are holomorphic in $\Omega = \{ \Re u
> -\frac12\} $ and satisfy the bounds
\[
|\tilde \phi_j(u)| \leq \frac{C^j}{j!}|u|^2\la u\ra^{-1}, \quad
j\ge1
\]
for all $u\in\Omega$. In particular\footnote{If $a,b>0$, then $a\ll
b$ means that $a<\eps b$ for some small constant $\eps>0$, whereas
$a\asymp b$ means that for some constant $C>0$ one has $C^{-1}
a<b<Ca$}, in the region $\xi^{-\frac14}\ll R \ll \xi^{-\frac12}$,
\begin{equation}
  \label{eq:phiR2}
  \begin{split}
  |\phi(R,\xi)| &\asymp R^4\xi \phi_0(R) \asymp
  R^{\frac52}\xi\\
  |\pr_R \phi(R,\xi)| &\asymp
  R^{\frac32}\xi
  \end{split}
\end{equation}
\end{proposition}

\begin{proof}
Write $\phi(R,z)=\sum_{j=0}^\infty z^j \phi_j(R)$. The functions
$\phi_j$ then need to solve $ \cL \phi_j = \phi_{j-1}$. Since
$\phi_0$ is not analytic, it is technically convenient to set
$\phi_j(R) = R^{-\frac32} f_j(R)$ (note that $R^{-\frac32}$ is the
decay of $\phi_0$). Our system of ODEs is then, with $j\ge1$,
\[
\cL(R^{-\frac32} f_j) = R^{-\frac32} f_{j-1},\quad f_0(R) =
\frac{R^4}{1+R^4}
\]
The forward fundamental solution for $\cL$ is
\[
H(R,R') = (\phi_0(R) \theta_0(R') -\phi_0(R') \theta_0(R))1_{[R
> R']}
\]
Hence we have the iterative relation
\begin{align*}
f_j(R) &= \int_0^R  R^{\frac32} (R')^{-\frac32}  (\phi_0(R)
\theta_0(R') -\phi_0(R') \theta_0(R)) f_{j-1}(R')\, dR', \\
f_0(R) &= \frac{R^{4}}{(1+R^2)^2}
\end{align*}
Using the expressions for $\phi_0$, $\theta_0$ we rewrite this as
\[\begin{split}
f_j(R) &= 
\int_0^R  \Big[ R^{4}(-1-8R'^2 + 24 R'^4 \ln R' +8 R'^6 + R'^8) - \\
&- R'^{4}(-1-8R^2 + 24 R^4 \ln R +8 R^6 +
R^8) \Big] \frac{f_{j-1}(R')R'}{R'^4(1+R^2)^2(1+R'^2)^2} \, dR' 
\end{split}
\]
We claim that all functions $f_j$ extend to even holomorphic
functions in any even simply connected domain not containing $\pm
i$, vanishing at $0$. Indeed, we now suppose that $f_{j-1}$ has
these properties and we shall prove them for $f_j$. Clearly, $f_j$
extends to a holomorphic function in any even simply connected
domain not containing $\pm i$ and $0$. We first show that at $0$
there is at most an isolated singularity.  For this we consider a
branch of the logarithm which is holomorphic in $\Compl \setminus
\R^-$ and show that $f_j(R+i0) = f_j(R-i0)$ for $R < 0$.
Disregarding the terms not involving logarithms, we need to show
that for any holomorphic function $g$ we have
\[
\int_{0}^{R+i0} (\log R' - \log (R+i0)) g(R')\, dR' =
\int_{0}^{R-i0} (\log R' - \log (R-i0)) g(R') \, dR'
\]
This is obvious since for $R' < 0$ we have
\[
\log (R'+i0) - \log (R+i0) = \log (R'-i0) - \log (R-i0)
\]
The singularity at $0$ is a removable singularity. Indeed, for $R'$
close to $0$ we have $|f_{j-1}(R')| \lesssim |R'|$ which by a crude
bound on the denominator in the above integral leads to $ |f_{j}(R)|
\lesssim |R|$ (again with $R$ close to $0$). This also shows that
$f_j$ vanishes at $0$ (better bounds will be obtained below).
The fact that $f_j$ is even is obvious if we substitute $2 \log R'$
and $2 \log R$ by $\log R'^2$ respectively $\log R^2$ in the
integral. This is allowed since due to the above discussion we can
use any branch of the logarithm. Indeed, denoting $\tilde
f_{j-1}(R'^2) = f_{j-1}(R')$ the  change of variable $R'^2 = v$
yields the iterative relation, with $\tilde f_0 (u) =
\frac{u^2}{(1+u)^2}$,
\begin{equation}\label{eq:tilfj}\begin{split}
\tilde f_j(u) = 
& \int_0^{u}  \Big[ u^2(-1-8 v + 12 v^2\log v +8v^3
+v^4) \\
& -v^2(-1-8u+ 12u^2 \log u
  + 8u^3+u^4)\Big] \frac{\tilde f_{j-1}(v)}{2v^2(1+u)^2(1+v)^2} \, dv
\end{split}\end{equation}
Next, we obtain bounds on the functions $\tilde f_j$. To avoid the
singularity at $-1$ we restrict ourselves to the region $U = \{ \Re
u > -\frac12\}$.  We claim that the $\tilde f_j$ satisfy the bound
\[
|\tilde f_j(u)| \leq \frac{C^{j}}{j!}  |u|^{j+2}\la u\ra^{-1}
\]
The kernel above can be estimated by
\[
\left|\frac{u^2(-1-8 v + 12 v^2\log v +8v^3 +v^4) -v^2(-1-8u+ 12u^2
\log u
  + 8u^3+u^4)}{2v^2(1+u)^2(1+v)^2}\right| \leq C \frac{|u|^2}{|v|^2}
\]
We have
\[
|\tilde f_0(u)| \leq  \frac{|u|^2}{1+|u|^2}
\]
which yields
\[
|\tilde f_1(u)| \leq C\, |u|^2 \int_{0}^{|u|} \frac{1}{1+x^2}\, dx
\le C\, |u|^3\la u\ra^{-1}
\]
From here on we use induction, noting that for $j\ge1$
\[
|\tilde f_{j+1}(u)|\le \frac{C^{j}}{j!} \int_{0}^{|u|} x^{j}\la
x\ra^{-1}|u|^2 \, dx \leq \frac{C^{j+1}}{(j+1)!} |u|^{j+3}\la
u\ra^{-1}
\]
Finally, note that the functions $\tilde \phi_j$ are given by
$\tilde\phi_j (u) = u^{-j} \tilde f_j(u)$ and satisfy the desired
pointwise bound.

The statement \eqref{eq:phiR2} follows from the fact that $|\tilde
\phi_1(u)|\gtrsim u$ for $u\gg1$.
\end{proof}

 We note that although the above series for $\tilphi$
converges for all $R,z$, we can only use it to obtain various
estimates for $\tilphi$ in the region $|z| R^2 \lesssim 1$. On the
other hand, in the region $\xi R^2 \gtrsim 1$ where $z=\xi>0$,  we
will represent $\tilphi$ in terms of $\psi^+$ and use the $\psi^+$
asymptotic expansion, described in what follows.

\subsection{The asymptotic behavior of $\psi^+$}

The following result provides good asymptotics for $\psi^+$ in the
region $R^2 \xi \gtrsim 1$.

\begin{proposition} For any $\xi>0$, the solution  $\psi^+(\cdot,\xi)$ from Theorem~\ref{thm:GZ}  is of
the form
\[
\psi^+(R,\xi) =  \xi^{-\frac14}e^{iR \xi^\frac12}
\sigma(R\xi^\frac12,R),\qquad R^2\xi\gtrsim 1
\]
where $\sigma$ admits the asymptotic series approximation
\[
\sigma(q,R) \approx \sum_{j=0}^\infty q^{-j} \psi^+_j(R), \qquad
\psi^+_0 = 1, \qquad \psi_1^+ = \frac{15i}{8} + O(\frac{1}{1+R^2})
\]
with zero order symbols  $\psi^+_j(R)$ that are analytic at
infinity,
\[\sup_{R>0}
|(R \partial_R)^k \psi^+_j(R)| <\infty
\]
in the sense that for all large integers $j_0$, and all indices
$\alpha$, $\beta$, we have
\[
\sup_{R>0}\Bigl|(R \partial_R)^\alpha (q \partial_q)^\beta
\Big[\sigma(q,R)
  - \sum_{j=0}^{j_0} q^{-j} \psi_j^+(R)\Big]\Bigr| \leq
c_{\alpha,\beta,j_0}  q^{-j_0-1}
\]
for all $q>1$.
 \label{ppsipsi}\end{proposition}

\begin{proof}
With the notation
\[
\sigma(q,R) =  \xi^{\frac14}\psi^+(R,\xi)  e^{-iR \xi^\frac12}
\]
we need to solve the conjugated equation
\begin{equation}
\left(-\partial_R^2 - 2 i \xi^\frac12 \partial_R + \frac{15}{4R^2} -
  \frac{24}{(1+R^2)^2}\right) \sigma(R\xi^{\frac12}, R) = 0
\label{conjug}\end{equation} We look for a formal power series
solving this equation, i.e.,
\begin{equation}
\sigma(q,R)=\sum_{j=0}^\infty \xi^{-\frac{j}2} f_j(R) \label{formal}
\end{equation}
 This yields a recurrence relation for the $f_j$'s,
\[
2i f_j'(R) = \left(-\frac{d^2}{dR^2} + \frac{15}{4R^2} -
  \frac{24}{(1+R^2)^2}\right) f_{j-1}(R), \qquad f_0 = 1
\]
which is solved by
\[
f_j(R) = \frac{i}{2}  f_{j-1}'(R) + \frac{i}{2} \int_{R}^\infty
\left( \frac{15}{4R'^2} -
  \frac{24}{(1+R'^2)^2}\right) f_{j-1}(R')\, dR'
\]
Extending this into the complex domain, it is easy to see that the
functions $f_j$ are holomorphic in $\Compl \setminus [-i,i]$. They
are also holomorphic at $\infty$, and the leading term in the Taylor
series at $\infty$ is $R^{-j}$. At~$0$ one has the  estimate
\[
|(R \partial_R)^k f_j(R)| \leq c_{jk}\, R^{-j} \qquad \forall R>0
\]
which is easy to establish inductively. The functions
\[
\psi_j^+(R) := R^j f_j(R)
\]
now satisfy the desired bounds due to the bounds above on $f_j$. The
remainder of the proof is the same as in our wave-map
paper~\cite{KST1} and we skip it.
\end{proof}

\subsection{The spectral measure}

We now describe the spectral measure by means of \eqref{mw}. This
requires relating the functions $\tilphi$, $\tiltheta$ and
$\psi^\pm$. By examining the asymptotics at $R = 0$ we see that
\begin{equation}
W(\tiltheta,\tilphi) = 1\label{phitheta}\end{equation} Also by
examining the asymptotics as $R \to \infty$ we obtain
\begin{equation}
W(\psi^+,\psi^-) = -2 i \label{psipsi}\end{equation}

\begin{lemma} \label{lem:spec_meas}
a) We have
\begin{equation}\label{eq:phiapsi}
\tilphi(R,\xi) =  a(\xi) \psi^+(R,\xi) + \overline{ a(\xi)
\psi^+(R,\xi)}
\end{equation}
where $a$ is smooth, always nonzero, and has size
\[
|a(\xi)| \asymp \left\{ \begin{array}{cc} 1  & \text{\ if\ \ }\xi
\ll 1 \\ \xi^{-1} & \text{\ if\ \ } \xi \gtrsim 1
\end{array}\right.
\]
Moreover, it satisfies the symbol type bounds
\[
| (\xi \partial_\xi)^k a(\xi) | \leq c_k |a(\xi)|\quad
\forall\;\xi>0
\]

b) The absolutely continuous part of the spectral measure
$\mu(d\xi)$ has density $\rho(\xi)$ which satisfies
\[
\rho(\xi)\asymp \left\{ \begin{array}{cc}
   1 & \text{\ if\ \ }\xi \ll 1 \\ \xi^2 & \text{\ if\ \ } \xi \gtrsim 1
\end{array}\right.
\]
with symbol type estimates on the derivatives.
\end{lemma}

\begin{proof}
a) Since $\tilphi$ is real-valued, due to \eqref{psipsi}, the
relation~\eqref{eq:phiapsi} above holds with
\[
a(\xi) = -\frac{i}{2} W(\tilphi(\cdot,\xi),{\psi^-}(\cdot,\xi))
\]
We evaluate the Wronskian in the region where both the
$\psi^+(R,\xi)$ and $\tilphi(R,\xi)$ asymptotics are useful, i.e.,
where $R^2 \xi \approx 1$. The bounds from above on $a$ and its
derivatives thus follow from Propositions~\ref{pphitheta}
and~\ref{ppsipsi}.

For the bound from below on $a$ we use that
\[
|a(\xi)| \geq \frac{| \partial_R\tilphi(R,\xi)|}{2|\partial_R
\psi^+(R,\xi)|}
\]
which was obtained in~\cite{KST1}.  We use this relation for $R =
\delta \xi^{-\frac12}$ with a small constant $\delta$. Then by
Proposition~\ref{pphitheta} we have
\[
|\partial_R \tilphi(R,\xi)| \gtrsim \left\{ \begin{array}{cc}
   R^{-\frac12} & \xi \ll 1 \cr  \cr R^{\frac32} &  \xi \gtrsim 1
\end{array}\right.
\]
while by  Proposition~\ref{ppsipsi}
\[
{|\partial_R \psi^+(R,\xi)|} \lesssim \xi^{\frac14}
\]
This give the desired bound from below on $a$.

b) In \cite{KST1} it was shown that
\[
\rho(\xi)
 = \frac{1}\pi |a(\xi)|^{-2}
\]
The bounds on $\rho(\xi)$ now follow from part a).
\end{proof}

 \section{The transference identity}
 \label{sec:transference}

 We now write the radiation part $\tilde{\epsilon}$
  in terms of the
 generalized Fourier basis $\phi(R, \xi)$ from Theorem~\ref{thm:GZ}, i.e.,
 \begin{equation}\nonumber
    \tileps(\tau, R)=x_0(\tau) \phi_0(R) + \int_{0}^{\infty}  x(\tau,\xi)
   \phi(R,\xi)\rho(\xi)\,d\xi
 \end{equation}
 As in \cite{KST1}, \cite{KST2}  we define the
 error operator $\calK$ by
 \begin{equation}\label{eq:transfer}
 \widehat{R \partial_R u} = - 2 \xi \partial_\xi \hat u + \calK \hat u
 \end{equation}
 where the hat denotes the ``distorted Fourier transform"
  and the operator $-2\xi \partial_\xi$
acts only on the continuous part of the spectrum.  In view of
Remark~\ref{rem:vector} we obtain a matrix representation for
$\calK$, namely
\[
\calK = \left(\begin{array}{cc} \calK_{ee} & \calK_{ec} \cr
\calK_{ce} & \calK_{cc} \end{array} \right)
\]
Here    `c' and `e' stand for ``continuous" and ``eigenvalue",
respectively.  Using the expressions for the direct and inverse
Fourier transform in Theorem~\ref{thm:GZ} we obtain
\begin{align}
\calK_{ee}& =  \Bla R\partial_{R} \go(R), \go(R)\Bra_{L^2_R}\nn
\\
 \calK_{ec} f & =  \Bla \int_{0}^{\infty}f(\xi) R\partial_{R} \tilphi(R,
 \xi)\rho(\xi)\,d\xi\,,\, \go(R) \Bra_{L^2_R}\nn
\\
 \calK_{ce} (\eta)& = \Bla R\partial_{R} \go(R),  \tilphi(R,
 \eta) \Bra_{L^2_R}\nn
\\
 \calK_{cc} f(\eta) &=  \Bla
 \int_{0}^{\infty}f(\xi) R\partial_{R} \tilphi(R,
 \xi)\rho(\xi)\,d\xi\,,\, \tilphi(R, \eta)\Bra_{L^2_R}\nn\\
 & \quad + \Bla
 \int_{0}^{\infty}2\xi\partial_{\xi} f(\xi) \tilphi(R,
 \xi)\rho(\xi)\,d\xi\,,\, \tilphi(R, \eta)\Bra_{L^2_R}\label{eq:Kcc}
\end{align}
Integrating by parts with respect to $R$ in the first two relations
we obtain
\[
\calK_{ee} = -\frac12\|\phi_0\|_2^2=-\frac{1}{12}, \qquad \calK_{ec}
f = - \int_0^\infty f(\xi) K_e(\xi) \rho(\xi) d\xi, \qquad
\calK_{ce} (\eta) = K_e(\eta)
\]
where
\[
K_e(\eta) =  \Bla R\partial_{R} \go(R),  \tilphi(R,
 \eta) \Bra_{L^2_R}
\]
Integrating by parts with respect to $\xi$ in \eqref{eq:Kcc} yields
 \begin{equation}\label{calk}\begin{split}
     \calK_{cc} f(\eta) &=  \Bla
     \int_{0}^{\infty}f(\xi)[R\partial_{R}-2\xi\partial_{\xi}]
\tilphi(R,\xi)\rho(\xi)\,d\xi\,,\,
     \tilphi(R, \eta)\Bra_{L^2_R} \\
&\quad - 2 \left(1+ \frac{\eta
         \rho'(\eta)}{\rho(\eta)}\right)  f(\eta)
   \end{split}
\end{equation}
 where the scalar product is to be interpreted in the principal value
 sense with $f\in C_0^\infty((0,\infty))$.

 In this section, we study the boundedness properties of the operator
 $\calK$. We begin with a description of the function $K_e$
and of the kernel $K_0(\eta,\xi)$ of $\calK_{cc}$.

 \begin{theorem}\label{tp}
a)   The operator $\calK_{cc}$ can be written as
   \begin{equation}\label{eq:kern}
     \calK_{cc} = -\Big(\frac{3}{2} + \frac{\eta \rho'(\eta)}{\rho(\eta)}\Big)\delta(\xi-\eta) + \calK_0
   \end{equation}
   where the operator $\calK_0$ has a kernel $K_0(\eta, \xi)$ of the
   form\footnote{The kernel below is interpreted in the principal
     value sense}
   \begin{equation} \label{ketaxi} K_0(\eta,
     \xi)=\frac{\rho(\xi)}{\eta-\xi} F(\xi,\eta)
   \end{equation}
   with a symmetric function $F(\xi,\eta)$ of class $C^2$ in
   $(0,\infty) \times (0,\infty)$ and continuous on $[0,\infty)^2$. Moreover, it satisfying the bounds
   \[\begin{split}
   | F(\xi,\eta)| &\lesssim \left\{ \begin{array}{cc} \xi+\eta &
       \xi+\eta \leq 1 \cr (\xi+\eta)^{-\frac52} (1+|\xi^\frac12
       -\eta^\frac12|)^{-N} & \xi+\eta \geq 1
     \end{array} \right.\\
   | \partial_{\xi} F(\xi,\eta)|+| \partial_{\eta} F(\xi,\eta)| &\lesssim \left\{
     \begin{array}{cc} 1 & \xi+\eta \leq 1 \cr (\xi+\eta)^{-3}
       (1+|\xi^\frac12 -\eta^\frac12|)^{-N} & \xi+\eta \geq 1
     \end{array} \right.\\
  \sup_{j+k=2} | \partial^j_{\xi}\partial^k_{\eta} F(\xi,\eta)| &\lesssim \left\{
     \begin{array}{cc} (\xi+\eta)^{-1} & \xi+\eta \leq 1 \cr
       (\xi+\eta)^{-\frac72} (1+|\xi^\frac12 -\eta^\frac12|)^{-N} &
       \xi+\eta \geq 1
     \end{array} \right.
     \end{split}
   \]
   where $N$ is an arbitrary large integer.

b) The function $K_e$ and $K_e'$ are bounded, continuous,  and
rapidly decaying at infinity.
\end{theorem}

 \begin{proof}
   We first establish the off-diagonal behavior of $\calK_{cc}$, and later
   return to the issue of identifying the $\delta$-measure that sits
   on the diagonal. We begin with \eqref{calk} with $f \in
   C_0^\infty((0,\infty))$. The integral
   \[
   u(R) =
   \int_{0}^{\infty}f(\xi)[R\partial_{R}-2\xi\partial_{\xi}]\tilphi(R,
   \xi)\rho(\xi)\,d\xi
   \]
   behaves like $R^{\frac52}$ at $0$ and is a Schwartz function at
   infinity. The second factor $\tilphi(R,\eta)$ in \eqref{calk} also
   decays like $R^{\frac52}$ at $0$ but at infinity it is only bounded
   with bounded derivatives. Then the following integration by parts
   is justified:
   \[
   \eta \calK_{cc} f(\eta) = \Bla u(R), \cL \tilphi(R, \eta)\Bra_{L^2_R} = \Bla
   \cL u(R), \tilphi(R, \eta)\Bra_{L^2_R}
   \]
   Moreover,
   \[
   \begin{split}
    ( \cL u)(R) =& \int_{0}^{\infty}f(\xi)[\cL,R\partial_{R}] \tilphi(R,
     \xi)\rho(\xi)\,d\xi +
     \int_{0}^{\infty}f(\xi)(R\partial_{R}-2\xi\partial_{\xi}) \xi
     \tilphi(R,\xi)\rho(\xi)\,d\xi \\ = & \int_{0}^{\infty}f(\xi) [\cL,R\partial_{R}]
     \tilphi(R, \xi)\rho(\xi)\,d\xi + \int_{0}^{\infty}\xi
     f(\xi)(R\partial_{R}-2\xi\partial_{\xi})
     \tilphi(R,\xi)\rho(\xi)\,d\xi \\
     & - 2\int_{0}^{\infty}\xi
     f(\xi)\tilphi(R,\xi)\rho(\xi)\,d\xi
   \end{split}
   \]
   with the commutator
   \[
   [\cL,R\partial_{R}] = 2\cL + \frac{48}{(1+R^2)^2} -
   \frac{96R^2}{3(1+R^2)^3} =: 2\cL + U(R)
   \]
    Thus,
    \[
(\cL u)(R) =\int_{0}^{\infty}f(\xi) U(R)
     \tilphi(R, \xi)\rho(\xi)\,d\xi + \int_{0}^{\infty}\xi
     f(\xi)(R\partial_{R}-2\xi\partial_{\xi})
     \tilphi(R,\xi)\rho(\xi)\,d\xi
    \]
   Hence we obtain
   \[
   \eta \calK_{cc} f(\eta) - \calK_{cc} (\xi f)(\eta) = \Bla
   \int_{0}^{\infty}f(\xi) U(R) \tilphi(R, \xi)\rho(\xi)\,d\xi ,
   \tilphi(R, \eta)\Bra_{L^2_R}
   \]
   The double integral on the right-hand side is absolutely
   convergent, therefore we can change the order of integration to
   obtain
   \[
   (\eta -\xi) K_0(\eta,\xi) = \rho(\xi) \Bla U(R) \tilphi(R, \xi),
   \tilphi(R, \eta)\Bra_{L^2_R}
   \]
   This leads to the representation in \eqref{ketaxi} when $\xi\ne\eta$ with
   \[
   F(\xi,\eta) = \Bla U(R) \tilphi(R, \xi), \tilphi(R,
   \eta)\Bra_{L^2_R}
   \]
   It remains to study its size and regularity. First, due
   to our pointwise bound from the previous section,
\[\begin{aligned}
|\phi(R,\xi)| &\less \min(R^{\frac52}\la R\ra^{-4}
(1+R^4\xi),\xi^{-\frac14}) &\forall\;  0\le\xi<1\\
|\phi(R,\xi)| &\less \min(R^{\frac52},\xi^{-\frac54}) &\forall\;
\xi>1
\end{aligned}
\]
Note that these bounds imply that for all $\xi\ge0$,
\[
\la R\ra^{-2} |\phi(R,\xi)|\less \phi_0(R)\less \la R\ra^{-\frac32}
\]
Hence, $|F(\xi,\eta)|\less 1$ for all $0\le\xi,\eta<1$. Moreover,
$F(\xi,\eta)$ is continuous on $[0,\infty)^2$ by dominated
convergence. Finally, using that $|\phi(R,\xi)|\less \xi^{-\frac54}$
when $\xi>1$ implies that
\begin{equation}
  \label{eq:Funif}
  |F(\xi,\eta)|\less \la \xi\ra^{-\frac54}\la \eta\ra^{-\frac54}
  \quad \forall\, \xi,\eta\ge0
\end{equation}
We shall improve on this in a number of ways, but first we consider
derivatives. By the previous section,
\begin{align*}
|\partial_\xi \tilphi(R, \xi)|
   &\less \min(R^{\frac92}, R\xi^{-\frac74}) \quad \forall\;
   \xi>1
   \\
|\partial_\xi \tilphi(R, \xi)|
   &\less \min(R^{\frac52},R\xi^{-\frac34}) \quad \forall\;
   0<\xi<1
\end{align*}
Consequently, if $0<\xi,\eta<1$, then
\begin{align*}
  |\partial_\xi F(\xi,\eta)| &\less \int_0^\infty \la R\ra^{-4}
  \min(R^{\frac52},R\xi^{-\frac34}) \min(\la
  R\ra^{-\frac32}(1+R^4\eta),\eta^{-\frac14})\, dR \\
  &\less \int_0^{\eta^{-\frac12}} \la R\ra^{-3} (1+R^4\eta)\, dR +
  \int_{\eta^{-\frac12}}^\infty \la R\ra^{-\frac32}
  \eta^{-\frac14}\, dR \less 1
\end{align*}
whereas if $0<\xi<1<\eta$, then
\[
|\partial_\xi F(\xi,\eta)|\less \int_0^\infty \la R\ra^{-\frac32}
\eta^{-\frac54}\, dR \less \eta^{-\frac54}
\]
If $0<\eta<1<\xi$, then
\begin{align*}
  |\partial_\xi F(\xi,\eta)| &\less \int_0^\infty \la R\ra^{-4}
  \min(R^{\frac92},R\xi^{-\frac74}) \min(\la
  R\ra^{-\frac32}(1+R^4\eta),\eta^{-\frac14})\, dR \\
  &\less \xi^{-\frac74} \int_0^{\eta^{-\frac12}} \la R\ra^{-\frac92} (1+R^4\eta)\, dR +
  \int_{\eta^{-\frac12}}^\infty \la R\ra^{-4} R \xi^{-\frac74}
  \eta^{-\frac14}\, dR \less \xi^{-\frac74}
\end{align*}
Finally, for $1<\xi,\eta$,
\[
|\partial_\xi F(\xi,\eta)| \less \int_0^\infty \la R\ra^{-4}
  R\xi^{-\frac74} \eta^{-\frac54}\, dR \less \xi^{-\frac74}
  \eta^{-\frac54}
\]
To summarize,
\begin{equation}
  \label{eq:Fderunif}
  |\partial_\xi F(\xi,\eta)|\less \la \xi\ra^{-\frac74}\la
  \eta\ra^{-\frac54}, \quad |\partial_\eta F(\xi,\eta)|\less \la \xi\ra^{-\frac54}\la
  \eta\ra^{-\frac74}
  \qquad \forall\, \xi,\eta\ge0
\end{equation}

For the second derivatives we use that
\begin{align*}
|\partial^2_\xi \tilphi(R, \xi)|
   &\less \min(R^{\frac{13}{2}}, R^2\xi^{-\frac94}) \quad \forall\;
   \xi>1
   \\
|\partial^2_\xi \tilphi(R, \xi)|
   &\less \min(R^{\frac92},R^2\xi^{-\frac54}) \quad \forall\;
   0<\xi<1
\end{align*}
which imply the bounds
   we always have the estimates
   \begin{equation}\label{eq:F2derunif}\begin{split}
 |\partial^2_{\xi\eta} F(\xi,\eta)| &\less
\xi^{-\frac74}\eta^{-\frac74} \qquad \forall\;
 \xi>1,\,\eta>1\\
|\partial_\xi^2 F(\xi,\eta)| &\less \xi^{-\frac94}\eta^{-\frac54} \qquad \forall\;\xi>1,\,\eta>1\\
|\partial_\eta^2 F(\xi,\eta)| &\less \xi^{-\frac54}\eta^{-\frac94}
\qquad \forall\;\xi>1,\,\eta>1
\end{split}
\end{equation}
The bounds \eqref{eq:Funif}, \eqref{eq:Fderunif}, and
\eqref{eq:F2derunif}   are only useful when $\xi$ and $\eta$ are
very close. To improve on them, we consider two
   cases:

   {\bf Case 1: $1 \lesssim \xi+\eta $.}
   To capture the cancellations when $\xi$ and $\eta$ are separated we
   resort to another integration by parts,
   \begin{equation}
   \label{eq:Fdef}
   \eta F(\xi,\eta) = \Bla U(R) \tilphi(R, \xi), \cL \tilphi(R,
   \eta)\Bra = \Bla [\cL, U(R)] \tilphi(R, \xi), \tilphi(R, \eta)\Bra +
   \xi F(\xi,\eta)
   \end{equation}
   Hence, evaluating the commutator,
   \begin{equation} \label{Kk1} (\eta -\xi) F(\xi,\eta) = -\Bla (2
   U_R(R)
     \partial_R + U_{RR}(R)) \tilphi(R, \xi), \tilphi(R, \eta)\Bra
   \end{equation}
   Since $U_R(0)=0$ it follows that $ (2 U_R(R) \partial_R + U_{RR}(R))
   \tilphi(R, \xi)$ vanishes at the same rate as $\tilphi(R, \xi)$ at $R=0$.
   Then we can repeat the argument above to obtain
   \[
   (\eta -\xi)^2 F(\xi,\eta) = -\Bla [\cL,2 U_R \partial_R + U_{RR}]
   \tilphi(R, \xi), \tilphi(R, \eta)\Bra
   \]
   The second commutator has the form, with $V(R):= -
   24(1+R^2)^{-2}$,
   \[ [\cL,2 U_R \partial_R + U_{RR}] = 4 U_{RR} \cL - 4U_{RRR}
   \partial_R -U_{RRRR}
    - 2 U_R V_R -4U_{RR} V
   \]
   Since $V(R),U(R)$ are even,  this leads to
   \[
   (\eta -\xi)^2 F(\xi,\eta) = \Bla ( U^{odd}(R) \partial_R + U^{even}(R) +
   \xi U^{even}(R) ) \tilphi(R, \xi), \tilphi(R, \eta)\Bra
   \]
   where by $ U^{odd}$, respectively $U^{even}$, we have generically
   denoted odd, respectively even, nonsingular rational functions with
   good decay at infinity.  Inductively, one now verifies the identity
        \begin{equation} \begin{split}
     &(\eta -\xi)^{2k} F(\xi,\eta) =  \Bla \Big( \sum_{j=0}^{k-1} \xi^{j}\, U_{kj}^{odd}(R)\, \partial_R +
     \sum_{\ell=0}^k \xi^\ell U_{k\ell}^{even}(R) \Big) \tilphi(R, \xi), \tilphi(R, \eta)\Bra
     \label{Kk} \\
&\la R\ra|U_{kj}^{odd}(R)| + |U_{k\ell}^{even}(R)| \less \la
R\ra^{-4-2k} \qquad \forall\; j,\ell
     \end{split}\end{equation}
   By means of the pointwise bounds on $\tilphi$  from above
as well as
\[
|\partial_R \phi(R,\xi)|\less \left\{ \begin{array}{cc} \max(\la
R\ra^{-\frac12}, \xi^{\frac14})\le 1 & \text{\ if\ \ } 0\le
\xi \le 1\\
\min(R^{\frac32}, \xi^{-\frac34})  \le \xi^{-\frac34} & \text{\ if\
\ } \xi \ge 1
\end{array}\right.
\] we infer from this that
   \[
   |F(\xi,\eta)| \lesssim \frac{\la\xi\ra^{k -\frac54}
     \la\eta\ra^{-\frac54}} {(\eta -\xi)^{2k}} \qquad\forall \;
     \xi, \;\eta>0
\]
Combining this estimate with \eqref{eq:Funif} yields, for arbitrary
$N$,
\begin{equation}\label{eq:Fabsch} |F(\xi,\eta)|\less (\xi+\eta)^{-\frac52} (1+|\xi^\frac12
       -\eta^\frac12|)^{-N} \text{\ \ provided\ \ } \xi+\eta \gtrsim 1,
   \end{equation}
   as claimed.
   For the derivatives of $F$ we follow a similar procedure. If $\xi$
   and $\eta$ are comparable,  then from~\eqref{eq:Fderunif},
   $|\partial_\eta F(\xi,\eta)|\less \la \xi\ra^{-3}$. We will use
   this bound only when $|\xi^\frac12
       -\eta^\frac12|<1$ which of course implies that
       $\xi\asymp\eta\gtrsim 1$. Thus, we now assume that $|\xi^\frac12
       -\eta^\frac12|\ge 1$ which is the same as $|\xi-\eta|>(\xi+\eta)^{\frac12}$.
   In this case,  we differentiate with respect to $\eta$ in \eqref{Kk}.
   This yields
   \[\begin{split}
   (\eta -\xi)^{2k} \partial_\eta F(\xi,\eta) &=
   \Bla \Big( \sum_{j=0}^{k-1} \xi^{j}\, U_{kj}^{odd}(R)\, \partial_R +
     \sum_{\ell=0}^k \xi^\ell U_{k\ell}^{even}(R) \Big)\tilphi(R, \xi),
   \partial_\eta \tilphi(R, \eta)\Bra\\ &
    - 2k (\eta -\xi)^{2k-1}
   F(\xi,\eta)
   \end{split}
   \]
   Using the bound on $F$ from \eqref{eq:Fabsch} as well as the
   usual estimate on~$\pr_\eta \phi(R,\eta)$, leads to
\begin{equation}\label{eq:F'absch}
|F(\xi,\eta)|\less (\xi+\eta)^{-3} (1+|\xi^\frac12
       -\eta^\frac12|)^{-N} \text{\ \ provided\ \ } \xi+\eta \gtrsim 1
   \end{equation}
   The second order derivatives with respect to
   $\xi$ and $\eta$ are treated in an analogous manner.
    We note that it is important here that the decay of
   $U_{kj}^{odd}$ and $U_{k\ell}^{even}$ improves with $k$. This is because
   the optimal second derivative bound for small $\eta$, viz.\ $|\pr_\eta\phi(R,\eta)|\less
   R^{\frac92}$,
   has a sizeable growth in~$R$.

\medskip
   {\bf Case 2: $ \xi,\eta \ll 1$}. First, we note that
   \[ F(0,0) = \la U \phi_0,\phi_0\ra = \Bla
   ([\cL,R\partial_R]-2\cL)\phi_0,\phi_0\Bra = 0
   \]
   Together with the derivative bound \eqref{eq:Fderunif}, this
   implies that
   \[
|F(\xi,\eta)|\less \xi+\eta,
   \]
as claimed.
   To bound the second order derivatives of $F$ we recall the pointwise
   bounds, for $0<\xi<1$,
   \[
   |\partial_\xi \phi(R,\xi)|\less \min(R^{\frac52},R\xi^{-\frac34})
   \]
   If $0 < \xi<\eta <1$, then these bounds imply that
   \begin{equation}\label{eq:Fxieta}\begin{split}
   | \partial_{\xi\eta} F(\xi,\eta)| &\lesssim
   \int_{0}^{\eta^{-\frac12}} \la R\ra^{-4} R^5\, dR +
   \int_{\eta^{-\frac12}}^{\xi^{-\frac12}} \la R\ra^{-4} R^{\frac72} \eta^{-\frac34}
   \,dR
   +\int_{\xi^{-\frac12}}^\infty \la R\ra^{-2} (\xi\eta)^{-\frac34}
   \, dR \\
   & \less \eta^{-1} + \xi^{-\frac14}\eta^{-\frac34}
   \end{split}
   \end{equation}
This bound is only acceptable as long as $\xi$ and $\eta$ are
comparable. Otherwise, if $0<\xi\ll\eta\le 1$, then one needs to
exploit the oscillations of~$\pr_\eta \phi(R,\eta)$
   in the regime $R^2\eta>1$ as provided by Proposition~\ref{ppsipsi}
   and Lemma~\ref{lem:spec_meas}. Thus, write
   \begin{align*}
\pr_\eta \phi(R,\eta) &= \pr_\eta[a(\eta) \psi^+(R,\eta) +
\overline{ a(\eta) \psi^+(R,\eta)}] = 2\Re \pr_\eta \big[
a(\eta)\eta^{-\frac14} e^{iR\eta^{\frac12}}
\sigma(R\eta^\half,R)\big] \\
&= 2\Re \big[ (a(\eta)\eta^{-\frac14})' e^{iR\eta^{\frac12}}
\sigma(R\eta^\half,R)\big] + R \Re \big[ i a(\eta)\eta^{-\frac34}
e^{iR\eta^{\frac12}} \sigma(R\eta^\half,R)\big] \\
&\qquad + R \Re \big[ a(\eta)\eta^{-\frac34} e^{iR\eta^{\frac12}}
\sigma_q (R\eta^\half,R)\big]
   \end{align*}
   Therefore,
\begin{align}
 &\left| \int_{\eta^{-\frac12}}^\infty U(R) \pr_\xi\phi (R,\xi)
\pr_\eta \phi(R,\eta) \,dR \right| \nn
\\&\less \left|
\int_{\eta^{-\frac12}}^\infty U(R) \pr_\xi\phi (R,\xi)
(a(\eta)\eta^{-\frac14})' e^{iR\eta^{\frac12}} \sigma(R\eta^\half,R)
\,dR \right| \label{eq:int1}\\
&  + \left| \int_{\eta^{-\frac12}}^\infty R U(R) \pr_\xi\phi
(R,\xi) a(\eta)\eta^{-\frac54} \sigma(R\eta^\half,R) \pr_R
e^{iR\eta^{\frac12}}
 \,dR \right| \label{eq:int2}\\
 &  + \left|
\int_{\eta^{-\frac12}}^\infty RU(R) \pr_\xi\phi (R,\xi)
a(\eta)\eta^{-\frac54} \sigma_q(R\eta^\half,R) \pr_R
e^{iR\eta^{\frac12}} \,dR \right| \label{eq:int3}
\end{align}
The term on the right-hand side of~\eqref{eq:int1} is bounded by
\[
\int_{\eta^{-\frac12}}^\infty R^{-\frac32} \eta^{-\frac54}  \,dR
\less \eta^{-1}
\]
whereas \eqref{eq:int2} and \eqref{eq:int3} require integrating by
parts. It will suffice to consider the former. Using that
$|\partial_{R\xi} \phi(R,\xi)|\less \min(R^{\frac32},
R\xi^{-\frac14})$ and that $|\pr_q \sigma(q,R)|\less R^{-1}$, we
obtain
\begin{align*}
& \left| \int_{\eta^{-\frac12}}^\infty R U(R) \pr_\xi\phi (R,\xi)
a(\eta)\eta^{-\frac54} \sigma(R\eta^\half,R) \pr_R
e^{iR\eta^{\frac12}}
 \,dR \right| \\
&\less \eta^{-\frac54}|R U(R) \pr_\xi\phi (R,\xi)
 |_{R=\eta^{-\frac12}} \\
& + \eta^{-\frac54}\int_{\eta^{-\frac12}}^\infty \big[\la R\ra^{-4}
|\pr_\xi\phi (R,\xi)| + \la R\ra^{-3} |\pr_{R\xi}\phi
(R,\xi)|\big]\, dR \less \eta^{-1}
\end{align*}
In conclusion, for all $0\le \xi,\eta\le1$,
\[
| \partial_{\xi\eta} F(\xi,\eta)| \less (\xi+\eta) ^{-1}
\]
as desired. Next, consider $\pr_{\xi}^2 F(\xi,\eta)$. The bound
\begin{align*}
  |\pr_{\xi}^2
F(\xi,\eta) | &\less \int_0^\infty \la R\ra^{-4} \min (R^{\frac92},
R^2\xi^{-\frac54}) \la R\ra^{\frac12}\, dR \less \xi^{-1}
\end{align*}
is acceptable as long as $0<\eta\less \xi\le1$. If, on the other
hand, $0<\xi\ll\eta\le1$, then differentiating in \eqref{Kk1} we
   obtain
   \[
   (\eta-\xi) \partial_\xi^2 F(\xi,\eta) = 2 \partial_\xi
   F(\xi,\eta) - \Bla \partial_\xi^2 \tilphi(R, \xi), (2 U_R \partial_R + U_{RR})
    \tilphi(R, \eta)\Bra
   \]
which implies that
\[
   |\partial_\xi^2 F(\xi,\eta)| \less \eta^{-1}\big[ | \partial_\xi
   F(\xi,\eta)|  + |\la \partial_\xi^2 \tilphi(R, \xi),
    U_{RR}
    \tilphi(R, \eta)\ra| + |\la \partial_\xi^2 \tilphi(R, \xi), R^{-1} U_R
   R\partial_R
    \tilphi(R, \eta)\ra| \big]
   \]
The first term in brackets is $\less 1$, the second is bounded by
\[
\int_0^{\eta^{-\half}} \la R\ra^{-6} \la R\ra^{\frac92} \la
R\ra^{-\frac32} (1+R^4\eta)\,dR + \int_{\eta^{-\half}}^\infty R^{-6}
R^{\frac92}  \eta^{-\frac14}\, dR\less 1
\]
whereas the third is the same as the second in the range $0\le R\le
\eta^{-\frac12}$, whereas in the range $R\ge \eta^{-\half}$ we need
to integrate by parts; schematically, this amounts  to
\[
\left| \int_{\eta^{-\frac12}}^\infty \la
R\ra^{-\frac12}\eta^{-\frac14} \pr_R e^{iR\eta^\half}\, dR
\right|\less 1
\]
The full details are essentially the same as in the previous
integration by parts step and we skip them.

Next, we extract the $\delta$ measure that sits on the diagonal of
the
   kernel $K$ from the representation formula~\eqref{calk}, see also~\eqref{eq:kern}.
   To do so, we can restrict
    $\xi,\eta$  to a compact subset of
   $(0,\infty)$.  This is convenient, as we then have the following
   asymptotics of $\tilphi(R, \xi)$ for $R \xi^\frac12 \gg 1$:
   \[ \begin{split}
   \tilphi(R,\xi) &=
   \Re \left[a(\xi)\xi^{-\frac{1}{4}}
     e^{iR\xi^{\frac{1}{2}}} \Big(1+
       \frac{15i}{8R\xi^\frac12}\Big)\right] + O(R^{-2})\\
   (R\partial_R -2\xi \partial_\xi) \tilphi(R,\xi) &= -2\Re \left[ \xi
     \partial_\xi(a(\xi)\xi^{-\frac{1}{4}})
     e^{i R\xi^{\frac{1}{2}}} \Big(1+
       \frac{15i}{8R\xi^\frac12}\Big)\right] + O(R^{-2})
       \end{split}
   \]
   where the $O(\cdot)$ terms depend on the choice of the compact
   subset.
   The $R^{-2}$ terms are integrable so they contribute a bounded
   kernel
    to the inner product in~\eqref{calk}. The same applies to the contribution of a
   bounded $R$ region.  Using the above expansions, we
   conclude that the $\delta$-measure contribution of the inner
   product in~\eqref{calk} can only come from one of the following
   integrals:
   \begin{align}
     &  -\int_0^\infty \int_0^\infty f(\xi) \chi(R) \Re
     \left[ \xi \partial_\xi({a}(\xi)\xi^{-\frac{1}{4}})
       a(\eta)\eta^{-\frac{1}{4}}
       e^{iR(\xi^{\frac{1}{2}}+\eta^\frac12)}\Big(1+
       \frac{15i}{8R\xi^\frac12}\Big)\Big(1+
       \frac{15i}{8R\eta^\frac12}\Big)\right]\rho(\xi)\, d\xi dR \label{eq:I1}\\
       &  -\frac12\int_0^\infty \int_0^\infty f(\xi) \chi(R)\,
      \xi \partial_\xi({a}(\xi)\xi^{-\frac{1}{4}})
       \bar{a}(\eta)\eta^{-\frac{1}{4}}
       e^{iR(\xi^{\frac{1}{2}}-\eta^\frac12)}\Big(1+
       \frac{15i}{8R\xi^\frac12}\Big)\Big(1-
       \frac{15i}{8R\eta^\frac12}\Big)\;\rho(\xi)\, d\xi dR \label{eq:I2}\\
     & -\frac12\int_0^\infty \int_0^\infty f(\xi) \chi(R)\,
     \xi \partial_\xi(\bar{a}(\xi)\xi^{-\frac{1}{4}})
       a(\eta)\eta^{-\frac{1}{4}}
       e^{-iR(\xi^{\frac{1}{2}}-\eta^\frac12)}\Big(1-
       \frac{15i}{8R\xi^\frac12}\Big)\Big(1+
       \frac{15i}{8R\eta^\frac12}\Big)\;\rho(\xi)\, d\xi
       dR\label{eq:I3}
   \end{align}
   where $\chi$ is a smooth cutoff function which equals $0$ near
   $R=0$ and $1$ near $R=\infty$.  In all of the above integrals we can
   argue as in the proof of the classical Fourier inversion formula to
   change the order of integration. Integrating by parts in the first
   integral~\eqref{eq:I1} reveals that it cannot
   contribute a $\delta$-measure. Discarding the $R^{-2}$ terms from
   \eqref{eq:I2} and~\eqref{eq:I3} reduces us further to the
   expressions
\begin{align}
     &  -\int_0^\infty \int_0^\infty f(\xi) \chi(R) \Re
     \left[ \xi \partial_\xi({a}(\xi)\xi^{-\frac{1}{4}})
       \bar{a}(\eta)\eta^{-\frac{1}{4}}
       e^{iR(\xi^{\frac{1}{2}}-\eta^\frac12)}\right]\rho(\xi)\, d\xi dR \label{eq:I4}\\
& + \frac{15}{8}\int_0^\infty \int_0^\infty f(\xi) \chi(R) \Im
     \left[ \xi \partial_\xi({a}(\xi)\xi^{-\frac{1}{4}})
       \bar{a}(\eta)\eta^{-\frac{1}{4}}
       e^{iR(\xi^{\frac{1}{2}}-\eta^\frac12)}\right]R^{-1} (\xi^{-\frac12}-\eta^{-\frac12})\rho(\xi)\, d\xi dR \label{eq:I5}
   \end{align}
   The second integral \eqref{eq:I5} has both an
   $R^{-1}$ and a $(\xi^{-\frac12} -\eta^{-\frac12})$ factor so its
   contribution to $K$ is bounded. The first integral~\eqref{eq:I4}
contributes both a Hilbert transform type kernel  as well as a
$\delta$-measure to $K$. By inspection, the $\delta$ contribution is
\[\begin{split}
&-\frac12 \int_{-\infty}^\infty \Re
     \left[ \xi \partial_\xi({a}(\xi)\xi^{-\frac{1}{4}})
       \bar{a}(\eta)\eta^{-\frac{1}{4}}
       e^{iR(\xi^{\frac{1}{2}}-\eta^\frac12)}\right]\rho(\xi)\, dR
      \\
     &  = -\pi \Re
     \left[ \xi \partial_\xi({a}(\xi)\xi^{-\frac{1}{4}})
       \bar{a}(\eta)\eta^{-\frac{1}{4}}
       \right]\rho(\xi)\delta(\xi^{\frac12}-\eta^{\frac12})\\
       &= -2\pi \xi^{\frac12} \rho(\xi) \Re
     \left[ \xi \partial_\xi({a}(\xi)\xi^{-\frac{1}{4}})
       \bar{a}(\xi)\xi^{-\frac{1}{4}}
       \right]\delta(\xi-\eta)\\
&= -2\pi \xi^{\frac12} \rho(\xi) \Re
     \left[ -\frac14 \xi^{-\frac12}|a(\xi)|^2 + \xi^{\frac12} a(\xi)\bar{a}'(\xi)
       \right]\delta(\xi-\eta)\\
       &= \Big[\frac12 +
       \frac{\xi\rho'(\xi)}{\rho(\xi)}\Big]\delta(\xi-\eta)
\end{split}\]
where we used that $\rho(\xi)^{-1} = \pi |a|^{2}$ in the final step.
Combining this with the $\delta$-measure in \eqref{calk}
yields~\eqref{eq:kern}.

b)  Arguing as in part (a) we have
\[
K_e(\eta) = \frac{F(0,\eta)}{\eta}
\]
For $F$ we use the representation in \eqref{Kk} with $\xi$ replaced
by $0$ and $\phi(\cdot,\xi)$ replaced by $\phi_0$. The conclusion
easily follows from pointwise bounds on $\phi(\cdot,\eta)$ and its
derivatives.
 \end{proof}

 Next we consider the $L^2$ mapping properties for $\calK$. We
 introduce the weighted $L^2$ spaces $L^{2,\alpha}_\rho$
of functions on $\spec(\cL)$  with norm
 \begin{equation}\label{eq:L2weight}
 \| f\|_{L^{2,\alpha}_\rho}^2 :=  |f(0)|^2 +\int_0^\infty |f(\xi)|^2
 \la \xi\ra^{2\alpha} \rho(\xi)\,d\xi
 \end{equation}
 Then we have

\begin{proposition}
 a)  The operators $\calK_0$, $\calK$  map
  \[
  \calK_0\::\: L^{2,\alpha}_\rho \to L^{2,\alpha+1/2}_\rho, \qquad
 \calK \::\: L^{2,\alpha}_\rho \to L^{2,\alpha}_\rho.
  \]
b) In addition, we have the commutator bound
\[
[ \calK,\xi \partial_\xi] \::\:  L^{2,\alpha}_\rho \to
L^{2,\alpha}_\rho
\]
with $\xi\partial_\xi$ acting only on the continuous spectrum. Both
statements hold for all $\alpha\in\R$. \label{l2k}\end{proposition}
\begin{proof} We commence with the $\calK_0$ part.
a) The first property is equivalent to showing that the kernel
\[
\rho^{\frac12}(\eta)\la \eta\ra^{\alpha+1/2} K_0(\eta,\xi) \la
\xi\ra^{-\alpha} \rho^{-\frac12}(\xi) \::\: L^2(\R^+)\to L^2(\R^+)
\]
With the notation of the previous theorem, the kernel on the
left-hand side is
\[
\tilde K_0(\eta,\xi):= \la\eta\ra^{\alpha+1/2}\la\xi\ra^{-\alpha}
\frac{\sqrt{\rho(\xi)\rho(\eta)}}{\xi-\eta} F(\xi,\eta)
\]
We first separate the diagonal and off-diagonal behavior of $\tilde
K_0$, considering several cases.

  { \bf Case 1: $(\xi,\eta) \in Q := [0,4] \times [0,4]$.}

  We cover the unit interval with dyadic subintervals $I_j =
  [2^{j-1},2^{j+1}]$. We cover the diagonal with the union of squares
  \[
  A = \bigcup_{j=-\infty}^2 I_j \times I_j
  \]
  and divide the kernel $\tilde K_0$ into
  \[
  1_Q\tilde K_0 = 1_{A\cap Q} \tilde K_0 + 1_{Q \setminus A} \tilde K_0
  \]

  { \bf Case 1(a)}: Here we show that the diagonal part $1_{A\cap Q} \tilde K_0$ of
  $\tilde K_0$ maps $L^2$ to $L^2$. By orthogonality it suffices to restrict
  ourselves to a single square $I_j \times I_j$. We recall the $T1$
  theorem for Calderon-Zygmund operators, see page~293 in~\cite{Stein}: suppose the kernel
  $K(\eta,\xi)$ on $\R^2$ defines an operator $T:\calS\to\calS'$ and has the
  following pointwise properties with some $\gamma\in(0,1]$ and a constant $C_0$:
  \begin{enumerate}
  \item[(i)] $|K(\eta,\xi)|\le C_0|\xi-\eta|^{-1}$
  \item[(ii)] $|K(\eta,\xi)-K(\eta',\xi)|\le
  C_0|\eta-\eta'|^\gamma|\xi-\eta|^{-1-\gamma}$ for all
  $|\eta-\eta'|<|\xi-\eta|/2$
\item[(iii)] $|K(\eta,\xi)-K(\eta,\xi')|\le
  C_0|\xi-\xi'|^\gamma|\xi-\eta|^{-1-\gamma}$ for all
  $|\xi-\xi'|<|\xi-\eta|/2$
\end{enumerate}
If in addition $T$ has the restricted $L^2$ boundedness property,
i.e., for all $r>0$ and $\xi_0,\eta_0\in\R$,
$\|T(\omega^{r,\xi_0})\|_2 \le C_0r^{\frac12}$ and
$\|T^*(\omega^{r,\eta_0})\|_2 \le C_0r^{\frac12}$ where
$\omega^{r,\xi_0}(\xi)= \omega((\xi-\xi_0)/r)$ with a fixed
bump-function $\omega$, then $T$ and $T^*$ are $L^2(\R)$ bounded
with an operator norm that only depends on~$C_0$.

Within the
  square $I_j\times I_j$,
  Theorem~\ref{tp} shows that the kernel of $\tilde K_0$ satisfies these properties with $\gamma=1$,
  and is thus bounded on $L^2$.

  { \bf Case 1(b)}: Consider now the off-diagonal part $1_{Q \setminus
    A} \tilde K_0$. In this region, by Theorem~\ref{tp},
    \[
|\tilde K_0(\eta,\xi)| \less (\xi\eta)^{-\frac14}
    \]
which is a Hilbert-Schmidt kernel on $Q$ and thus $L^2$ bounded.

  {\bf Case 2:} $(\xi,\eta) \in Q^c $. We cover the diagonal with the union of squares
  \[
  B = \bigcup_{j=1}^\infty I_j \times I_j
  \]
  and divide the kernel $\tilde K_0$ into
  \[
  1_{Q^c}\tilde K_0 = 1_{B\cap Q^c} \tilde K_0 + 1_{Q^c \setminus B} \tilde K_0
  \]

  {\bf Case 2a:} Here we consider the estimate on $B$. As in case 1a)
  above, we use Calderon-Zygmund theory. Evidently, $|\tilde
  K_0(\eta,\xi)|\less |\xi-\eta|^{-1}$ on $B$ by Theorem~\ref{tp}.
  To check (ii) and (iii), we differentiate $\tilde K_0$. It will suffice to
  consider the case where the $\partial_\xi$ derivative falls on $F(\xi,\eta)$.
  We distinguish two cases:
if $|\xi^{\frac12}-\eta^{\frac12}|\le 1$, then $|\xi-\eta|\less
\xi^{\frac12}$ which implies that
\[
\frac{\xi^{-\frac12}|\xi-\xi'|}{|\xi-\eta|} \less
\frac{|\xi-\xi'|^{\frac12}}{|\xi-\eta|^{\frac32}} \qquad \forall\;
|\xi-\xi'|<|\xi-\eta|/2
\]
if, on the other hand, $|\xi^{\frac12}-\eta^{\frac12}|> 1$, then
\[
\frac{\xi^{-\frac12}|\xi-\xi'|}{|\xi-\eta|\,
|\xi^{\frac12}-\eta^{\frac12}|} \less
\frac{|\xi-\xi'|}{|\xi-\eta|^2} \qquad \forall\;
|\xi-\xi'|<|\xi-\eta|/2
\]
which proves property (iii) on $B$ with $\gamma=\frac12$, and by
symmetry also (ii). The restricted $L^2$ property follows form the
cancellation in the kernel and the previous bounds on the kernel.
Hence, $\tilde K_0$ is $L^2$ bounded on~$B$.

  {\bf Case 2b:} Finally, in the exterior region $Q^c \setminus B$ we
  have the bound, with arbitrarily large $N$,
  \[
  |\tilde K_0(\eta,\xi)| \lesssim  (1+\xi)^{-N} (1+\eta)^{-N}
  \]
  which is $L^2$ bounded by Schur's lemma.

This concludes the proof of the first mapping property in part (a).
The second one follows in a straightforward manner since $K_e$ is
rapidly decaying at $\infty$.

b) A direct computation shows that the kernel $K_0^{com}$ of the
commutator $[\xi \partial_\xi, K_0]$  is given by
\[
K_0^{com}(\eta,\xi) = (\eta \partial_\eta + \xi \partial_\xi)
K_0(\eta,\xi) + K_0(\eta,\xi) =
\frac{\rho(\xi)}{\xi-\eta}F^{com}(\xi,\eta)
\]
interpreted in the principal value sense and with $F^{com}$ given by
\[
F^{com}(\xi,\eta) = \frac{\xi \rho'(\xi)}{\rho(\xi)} F(\xi,\eta) +
(\xi \partial_\xi+\eta
\partial_\eta) F(\xi,\eta)
\]
By Theorem~\ref{tp} this satisfies the same pointwise off-diagonal
bounds as $F$. Near the diagonal the bounds for $F^{com}$ and its
derivatives are worse than those for $F$ by a factor of
$(1+\xi)^\frac12$. Then the proof of the $L^2$ commutator bound for
$K_0$ is similar to the argument in part~(a).

The remaining part of the commutator $[\calK,\xi \partial_\xi]$
involves

(i) The commutator of the diagonal part of $\calK_{cc}$ with $\xi
\partial_\xi$. This is the multiplication operator by
\[
\xi \partial_\xi \frac{\xi \rho'(\xi)}{\rho(\xi)}
\]
which is bounded since $\rho$ has symbol like behavior both at $0$
and at $\infty$.

(ii) The operator $\xi \partial_\xi \calK_{ce}$ which is given by
the bounded rapidly decreasing function $\xi \partial_\xi K_e(\xi)$.

(iii)  The operator $\calK_{ec}\xi \partial_\xi $ given by
\[
\calK_{ec}\xi \partial_\xi f = \int_0^\infty K_e(\xi) \xi
\partial_\xi f(\xi) d\xi = - \int_0^\infty  f(\xi) \partial_\xi (\xi
K_e(\xi))d\xi
\]
which is also bounded due to the properties of $K_e$.
\end{proof}

\section{The second order transport equation}
\label{sec:lin}

This section is devoted to the study of the linear problem \eqref{eqlin}
which we restate here in the form
\begin{equation}
(-\pr_t^2 + \pr_r^2 +r^{-1}\pr_r + 2r^{-2}(1-3Q(R)^2))\eps = f
-12 \omega^2  \frac{R^2(1-R^2)}{(1+R^2)^3}\eps
\label{eqlina}\end{equation}
We recall that the second term on the right-hand side here arises due
to fact that its decay is of the same nature (namely $\omega^2$) as
that of other error terms which we will encounter in the parametrix
construction of this section. By doing this, the remaining terms in
the nonlinearity $\calN$ in \eqref{eqlin} decay more rapidly at
infinity. Our main result asserts that
\begin{proposition}
The backward solution $\epsilon$ for \eqref{eqlina} satisfies the
bound
\begin{equation}
\| \epsilon \|_{H^1_N} \lesssim \frac{1}N \|f\|_{L^2_N}
\label{linbd}\end{equation}
for all large enough $N$.
\label{epsmain}\end{proposition}

\begin{proof}
 We work in the coordinates $(R,\tau)$ given by
\[
R= r \lambda(t), \qquad \tau = \int_t^{1} \lambda(s)\, ds =
(\beta+1)^{-1}|\log t|^{\beta+1}
\]
for any $0<t<1$. For future reference, we note that
\[
t\lambda(t) = ((\beta+1)\tau)^{\frac{\beta}{\beta+1}}, \qquad
\lambda(\tau) = ((\beta+1)\tau)^{\frac{\beta}{\beta+1}}
e^{((\beta+1)\tau)^{\frac{1}{\beta+1}}}
\]
We introduce the auxiliary weight function $\omega(\tau)$
\[
\omega(\tau):=\lambda^{-1}\lambda_\tau(\tau)=
\frac{\beta}{\beta+1} \tau^{-1}
+ ((\beta+1)\tau)^{-\frac{\beta}{\beta+1}}
\]
and note that
\begin{equation}
 (t\lambda)^{-1} =
\omega(\tau) - \frac{\beta}{\beta+1}\tau^{-1} \label{eq:tlam}
\end{equation}

Then
\begin{align*}
  \pr_t &= \frac{\pr \tau}{\pr t} (\pr_\tau + R_\tau \pr_R) =
  -\lambda(\tau) (\pr_\tau + \omega R\pr_R) \\
  \pr_t^2 &= \lambda^2(\tau)\big[ (\pr_\tau + \omega
  R\pr_R)^2 + \omega (\pr_\tau + \omega R\pr_R) \big]
\end{align*}
therefore  the equation \eqref{eqlina} becomes
\begin{align*}
\Big[- (\pr_\tau + \omega
  R\pr_R)^2 - \omega (\pr_\tau + \omega R\pr_R) & \
  + \pr_R^2  +\frac{1}R \pr_R + \frac2{R^{2}}(1-3Q(R)^2)\Big]\eps =
\\
&  \lambda^{-2}f -12\omega^2
  \frac{R^2(1-R^2)}{(1+R^2)^3} \eps
\end{align*}
At this point it is convenient to switch to the notations
\begin{equation}
\tileps(\tau,R) =R^{\frac12} \eps(\tau,R), \qquad  \tilde f(\tau,R) :=
  R^{\frac12}\lambda^{-2}f(\tau,R)
\end{equation}
 Since
\[
R^{\frac12} (\pr_\tau + \omega R\pr_R)R^{-\frac12} = \pr_\tau +
\omega R\pr_R -  \omega/2,
\]
one concludes from conjugating the previous PDE by $R^{\frac12}$
that
\begin{align}
\Big[- (\pr_\tau + \omega
  R\pr_R)^2 + \frac12 \dot\omega +
  \frac14 \omega^2 -\calL\Big]\tilde\eps
= \tilde f -12 \omega^2
\frac{R^2(1-R^2)}{(1+R^2)^3}\tilde \eps \label{eq:main1a}
\end{align}
where $\dot\omega:=\partial_\tau\omega$ and
\[\begin{split}
\calL &:= -\pr_R^2  + \frac{15}{4R^2} - \frac{24}{(1+R^2)^2}
\end{split}
\]

Written in terms of $(\tileps,\tilde f)$ the estimate \eqref{linbd}
takes the form\footnote{Here we slightly abuse notations since the
  $N$'s in \eqref{linbd} and \eqref{tlinbd} do not coincide, instead
  they are linearly related.}
\begin{equation}
\| \tileps \|_{\tilde H^1_N} \lesssim \frac{1}{N}\|\tilde f\|_{\tilde L^2_N}
\label{tlinbd}
\end{equation}
where
\[
\| \tileps \|_{\tilde H^1_N} = \sup_{\tau > \tau_0}
\tau^{N-1-\frac{\beta}{\beta+1}} \|\tileps(\tau)\|_{L^2} +
\tau^{N-1}(\|\calL^\frac12 \tileps(\tau)\|_{L^2} + \|(\partial_\tau +
\omega R \partial_R)  \tileps(\tau)\|_{L^2})
\]
respectively
\[
\|\tilde f\|_{\tilde L^2_N} = \sup_{\tau > \tau_0}  \tau^{N} \|\tilde
f(\tau)\|_{L^2}
\]

In order to take advantage of the spectral properties of the operator
$\calL$ we conjugate the equation \eqref{eq:main1a} by the Fourier
transform $\calF$  adapted to $\calL$.  The transference identity is
\[
\calF R\pr_R \calF^{-1} = -2\xi\pr_\xi + \calK
\]
where
\[
\calK = \left[ \begin{matrix} -\frac12 & \calK_{ec} \\ \calK_{ce} &
\calK_{cc} \end{matrix} \right] =-\frac12 \Id + \left[
\begin{matrix} 0 & \calK_{ec} \\ \calK_{ce} &
-(1+\eta\rho'(\eta)/\rho(\eta))\delta(\xi-\eta) + \calK_0
\end{matrix} \right]
\]
and
\begin{align*}
  \calK_{ec}f &= -\int_0^\infty f(\xi) K_e(\xi)\rho(\xi)\,d\xi,\quad
  \calK_{ce} = K_e \\
 K_e(\xi) &= \langle R\phi_0'(R),\phi(R,\xi)\rangle
\end{align*}
We  write
\[
-\calF R\pr_R \calF^{-1} = \frac12\Id + \calK_d + \calK_{nd}
\]
where
\begin{align*}
\calK_d &= \left[ \begin{matrix} 0&0\\0&
2\xi\pr_\xi+(1+\xi\rho'(\xi)/\rho(\xi))
\end{matrix}\right] =  \left[ \begin{matrix} 0&0\\0&
\calD_0
\end{matrix}\right]  \\
\calK_{nd} &= -\left[ \begin{matrix} 0&\calK_{ec}\\ \calK_{ce}&
\calK_0
\end{matrix}\right]
\end{align*}
Then
\[
\calF ( \pr_\tau + \omega
  R\pr_R  )\calF^{-1}= D_\tau+\frac{\omega}2 - \omega \calK_{nd}, \qquad D_\tau =
\pr_\tau -\omega(1 +\calK_d)
\]
therefore
\begin{align*}
&\calF ( \pr_\tau + \omega
  R\pr_R  )^2  \calF^{-1}\! = \! (D_\tau+\frac{\omega}2)^2 \!
 - 2\omega \calK_{nd}D_\tau  \! + \omega^2 (
[\calK_{nd},\calK_d] + \calK_{nd}^2-\calK_{nd} ) \!  - \dot\omega\calK_{nd}
\end{align*}
Next we consider the Fourier transform of the last term in
\eqref{eq:main1a}, which we express in the form
\[
\calF \left(-12\frac{R^2(1-R^2)}{(1+R^2)^3}\tilde \eps\right) = \calJ \calF\tilde \eps,
\qquad
\calJ = \left[ \begin{matrix} \calJ_{ee}&\calJ_{ec}\\ \calF_{ce}&
\calJ_{cc}
\end{matrix}\right]
\]
We note that
\[
\calJ_{ee} = \|\phi_0\|^{-2}_{L^2} \langle \phi_0,
\frac{-12 R^2(1-R^2)}{(1+R^2)^3}\phi_0 \rangle =
\left(\frac16\right)^{-1} \frac{1}{10}=\frac{3}{5}
\]
while
\[
\calJ_{ce} = J_e(\xi) =  \langle \phi(R,\xi),
\frac{-12 R^2(1-R^2)}{(1+R^2)^3}\phi_0 \rangle
\]
and
\[
\calJ_{ec} x = \int_0^\infty \rho(\xi) J_e(\xi) x(\xi) d\xi
\]
We remark that the kernel $J_e$ is bounded and rapidly decreasing at
infinity. Finally,
\[
\calJ_{cc} x(\xi) = \int_0^\infty \rho(\xi) J_{cc}(\xi,\eta) x(\eta) d\eta
\]
with
\[
 J_{cc}(\xi,\eta) = \int_{0}^\infty -12 \frac{R^2(1-R^2)}{(1+R^2)^3}
\rho(\eta) \phi(R,\eta) \phi(R,\xi) dR
\]
This is bounded and has the off-diagonal decay property
\begin{equation}
J_{cc}(\xi,\eta) | \lesssim (1+\xi)^{-\frac12} (1+|\xi^\frac12 -\eta^\frac12|)^{-N}
\label{jcc}\end{equation}

Taking into account all the notations above,
the equation  \eqref{eq:main1a}  becomes
\[
\begin{split}
\Big[- D_\tau^2 - \omega D_\tau -\xi  \Big]\calF\tilde\eps
=&\ \calF \tilde f
 - 2\omega \calK_{nd}D_\tau  \calF \tileps
\\ &\ + \omega^2 (
[\calK_{nd},\calK_d] + \calK_{nd}^2 -\calK_{nd} +  \calJ ) \calF \tileps
 - \dot\omega\calK_{nd}  \calF \tileps
\end{split}
\]

Next,  write $\calF \tilde\eps=\left[\begin{matrix}  x_0 \\
x \end{matrix}\right]$  and  $\calF \tilde f=\left[\begin{matrix}  g_0 \\
g \end{matrix}\right]$, or equivalently,
\begin{align*}
\tilde\eps(\tau,R) &=x_0(\tau)\phi_0(R)\|\phi_0\|_2^{-2}
+\int_0^\infty x(\tau,\xi)\phi(R,\xi)\rho(\xi)\,d\xi =:
 \tilde\eps_0 + \tilde\eps_c \end{align*} where $\tilde\eps_0\perp
\tilde\eps_c$ for all $\tau\ge0$ (recall
$\phi_0(R)=R^{\frac52}(1+R^2)^{-2}$ and $\calL\phi_0=0$).
To write the system for $\left[\begin{matrix}  x_0 \\
x \end{matrix}\right]$ we compute
\begin{align*}
  \calK_{nd}^2 &=  \left[ \begin{matrix} 0&\calK_{ec}\\ \calK_{ce}&
\calK_0
\end{matrix}\right]^2 =
 \left[ \begin{matrix}  \calK_{ec}\calK_{ce} & \calK_{ec}\calK_0\\
\calK_0\calK_{ce} & \calK_{ce}\calK_{ec}+\calK_0^2
\end{matrix}\right] \\
\calK_{nd}\calK_d &=  -\left[ \begin{matrix} 0&\calK_{ec}\\
\calK_{ce}& \calK_0
\end{matrix}\right]\left[ \begin{matrix} 0&0\\0& \calD_0
\end{matrix}\right] = \left[ \begin{matrix}
0 & -\calK_{ec}\calD_0 \\ 0 & - \calK_0 \calD_0\end{matrix}\right]\\
\calK_{d}\calK_{nd} &=  -\left[ \begin{matrix} 0&0\\0& \calD_0
\end{matrix}\right]\left[ \begin{matrix} 0&\calK_{ec}\\
\calK_{ce}& \calK_0
\end{matrix}\right] = \left[ \begin{matrix}
0 & 0 \\  -\calD_0\calK_{ce} & - \calD_0\calK_0 \end{matrix}\right]
\end{align*}
We also note that
\[
-K_{ec} K_{ce} = \int_{0}^\infty \rho(\xi)|K_e(\xi)|^2 d\xi
= \frac{\| R \partial_R \phi_0\|_{L^2}^2 }{\|\phi_0\|_{L^2}^2}
-  \frac{\langle R \partial_R \phi_0,\phi_0\rangle^2 }{\|\phi_0\|_{L^2}^4}
= 6 \frac{17}{120} -\frac14 = \frac35
\]

Then we seek to write the equations for  $x_0$ and $x$  in the form of a
diagonal system with perturbative coupling,
\begin{equation}
P \left[\begin{matrix}  x_0 \\
x \end{matrix}\right] = \left[\begin{matrix}  g_0 \\
g \end{matrix}\right] + Q  \left[\begin{matrix}  x_0 \\
x \end{matrix}\right]
\label{xmain}\end{equation}
where
\[
P = \left[\begin{matrix}  P_e & 0 \\
0 & P_c \end{matrix}\right],
\qquad Q = \left[\begin{matrix}  0 & Q_{ec} \\
Q_{ce} & Q_{cc} \end{matrix}\right]
\]
with the principal part given by
\[
P_e = - \pr_\tau(\pr_\tau -\omega)
\]
respectively
\[
P_c = -D_\tau^2 - \omega D_\tau -\xi
\]
and the coupling terms of the form
\[
 Q_{ec} x = \omega^2 \calR_{ec} x
-2\omega \calK_{ec} D_\tau x
\]
with
\[
\calR_{ec}= (\omega^{-2}\dot\omega -1) \calK_{ec}
  +  \calK_{ec}\calK_0  + \calJ_{ec}
\]
while
\[
Q_{ce} x_0 = \omega^2 \calR_{ce} x_0  -2 \omega \calK_{ce} \partial_\tau
x_0, \qquad Q_{cc} x = \omega^2 \calR_{cc} x  - 2\omega \calK_{0}D_\tau
x
\]
with
\[
 \calR_{cc} = [\calK_0,\calB_0] +  \calK_{0}^2 + \calK_{ce} \calK_{ec}
 + \calJ_{cc}
\]
respectively
\[
\calR_{ce} =  -\calB_0 \calK_{ce}  -\calK_0 \calK_{ce}   + \calJ_{ce}
\]

Our main solvability result in Proposition~\ref{epsmain} for the
equation \eqref{eqlina} is restated in terms of the system
\eqref{xmain} as follows:

\begin{proposition} \label{propxg}
For each
with $(g_0,g)$ which satisfy
\[
|g_0(\tau)|\le \tau^{-N},\quad \|g(\tau,\cdot)\|_{L^2_\rho}\le
\tau^{-N}
\]
there exists an unique solution $(x,x_0)$ for the system
\eqref{xmain}  decaying at infinity. This solution satisfies the
bounds
\begin{equation}
| x_0(\tau)| \lesssim \frac{1}{N} \tau^{-N+\frac{2\beta+1}{\beta+1}},
\qquad
|\partial_\tau x_0|  \lesssim \tau^{-N+\frac{\beta}{\beta+1}}
\end{equation}
respectively
\begin{equation}
 \|x(\tau)\|_{L^2_\rho} \lesssim
\frac{1}{N} \tau^{-N+\frac{2\beta+1}{\beta+1}}, \qquad
\|\xi^\frac12 x(\tau)\|_{L^2_\rho}+ \|D_\tau
x(\tau)\|_{L^2_\rho} \lesssim  \frac{1}{N} \tau^{-N+1}
\end{equation}
\end{proposition}

\begin{proof}
Our strategy is  to solve first  the simpler linear equations
\begin{align}
 \displaystyle
 - \pr_\tau (\pr_\tau-\omega)
x_0 &= g_0 \label{eq:x0} \\
\Big[-D_\tau^2 - \omega D_\tau -\xi\Big] x &= g
\label{eq:x}
\end{align}
Then we will show that the right hand side in the system \eqref{xmain}
is perturbative.  We start with the linear operator governing $x_0$,
and introduce the appropriate function spaces for $x_0$ and $g_0$:
\[
\| g_0\|_{Y_0^N} =  \sup_{\tau \ge\tau_0} \tau^N |g_0(\tau)|
\]
\[
\| x_0\|_{X_0^N} =  \sup_{\tau \ge\tau_0}
 \tau^{N-\frac{2\beta+1}{\beta+1}} |x_0(\tau)|
+  \tau^{N-\frac{\beta}{\beta+1}} |\partial_\tau x_0(\tau)|
\]

\begin{lemma}
  \label{lem:x0_param} The backward solution operator $x_0 = T_e g_0$
 for \eqref{eq:x0} satisfies the estimate
  \begin{equation}
    \label{eq:T0bds}
 \| T_e g_0 \|_{X^N_0} \lesssim  \|g_0(\sigma)\|_{Y^N_0}
  \end{equation}
  for any  $N\ge 2$.
\end{lemma}
\begin{proof}
A fundamental basis of solutions of $ - \pr_\tau(\pr_\tau-\omega)$ is
given by
\[
a_+(\tau) = \lambda(\tau), \qquad a_{-}(\tau) = \lambda(\tau) \int_\tau^\infty \lambda^{-1}(\sigma) \,d\sigma = \omega^{-1}(\tau) (1 + O(\tau^{-\frac{1}{\beta+1}}))
\]
and has Wronskian $W(\tau) = \lambda(\tau)$.
Then the backward fundamental solution is given by
\begin{equation}
U_0(\tau,\sigma) = W^{-1}(\sigma)(a_+(\tau) a_{-}(\sigma) - a_+(\sigma) a_{-}(\tau)) =
\lambda(\tau) \int_\tau^\sigma \lambda^{-1}(s) ds
\label{u0bd}\end{equation}
A direct computation shows that $U_0$ satisfies the bounds
\[
|U_0(\tau,\sigma)| \lesssim  \frac{1}{\omega(\tau)} ,
\qquad | \partial_\tau U_0(\tau,\sigma)| \lesssim  \tau^{-\frac{1}{\beta+1}}+
\frac{\omega(\tau)}{\omega(\sigma)} \frac{\lambda(\tau)}
{\lambda(\sigma)}
\]
The conclusion of the lemma follows.

\end{proof}

Next we bound the solution of the equation~\eqref{eq:x}, which is
hyperbolic. One is tempted to define spaces $X^N$ and $Y^N$ in a
manner which is similar to $X^N_0$ and $Y^N_0$. This would work for
the linear theory for \eqref{eq:x}, but would not be strong enough in
order to treat the right hand side in \eqref{xmain} in a perturbative
manner. Instead we define some stronger spaces using the additional weight
\[
m(\xi) = \xi^{\nu} + \xi^{-\nu}
\]
where $\nu> 0$ is a fixed small parameter.
We define the space $L^\infty_N L^2_\rho$ with norm
\[
\|g\|_{L^\infty_N L^2_\rho} = \sup_{\tau_> \tau_0}
 \tau^{N} \|g_1\|_{L^2_\rho}
\]
and the dyadic $L^2$ space $l^\infty_N L^2_{\rho m}$ with norm
\[
\|g\|_{l^\infty_N L^2_{\rho m}} = \sup_{\tau_> \tau_0}
 \|\sigma^{N-\frac{\beta}{2(\beta+1)}} g(\sigma)\|_{L^2_{\rho m}
   ([\tau,2\tau]\times \R)}
\]
Then  we define the $Y^N$ space as a sum of two spaces,
\[
\|g\|_{Y^N} = \inf_{g=g_1+g_2}
 \|g_1\|_{L^\infty_N L^2_\rho} + \| g_2(\sigma)\|_{l^\infty_{N-\frac{1}{\beta+1}} L^2_{\rho m}}
\]
Similarly we introduce  the $X^N$ space with norm
\[
\|x\|_{X^N} = \| x\|_{L^\infty_{N-1-\frac{\beta}{\beta+1}} L^2_\rho}
+\| (\xi^\frac12 x, D_\tau x)\|_{L^\infty_{N-1}  L^2_\rho \cap
l^\infty_{N-1} L^2_{\rho/ m}}
\]
Then our solvability result for \eqref{eq:x} is as follows:

\begin{lemma}
  \label{lem:xbd} The backward solution operator $x=T_c g$ for
  the equation~\eqref{eq:x} satisfies
  \begin{equation}
\| T_c g\|_{X^N} \lesssim \|g\|_{Y^N}
  \label{eq:xbd} \end{equation}
In addition we have the smallness relation
  \begin{equation}
\| T_c g\|_{X^N} \lesssim \frac1{\sqrt{N}} \|g\|_{L^\infty_N L^2_\rho}
  \label{eq:xbda} \end{equation}
for large $N$.

\end{lemma}
\begin{proof}
The equation~\eqref{eq:x} is equivalent to
\begin{equation}\label{eq:cont_eq}
\Big[-\Big(\pr_\tau-2\omega\xi\pr_\xi \Big)^2
+ \omega\Big(\pr_\tau-2\omega\xi\pr_\xi \Big)
-\xi\Big]\lambda^{-2}\rho^{\frac12}(\xi)\, x =
\lambda^{-2}\rho^{\frac12}(\xi)\, g
\end{equation}
We substitute the functions $(x,g)$ by $(y,h)$ where
$y = \rho^\frac12 x$ and $h= \rho^\frac12 g$. This has the effect
of removing the weight $\rho$ from the estimates.
The functions $(y,h)$
solve
\begin{equation}\label{eq:cont_eq'}
\Big[-\Big(\pr_\tau-2\omega\xi\pr_\xi\Big)^2 +\omega\Big(\pr_\tau-2\omega\xi\pr_\xi \Big)
-\xi\Big]\lambda^{-2}\, y = \lambda^{-2}\, h
\end{equation}
The characteristics of the homogeneous operator on the left are
$(\tau,\lambda^{-2}(\tau)\xi_0)$ which means that
\[
(\pr_\tau-2\omega\,\xi\pr_\xi)f(\tau,\xi)=
\frac{d}{d\tau}f(\tau,\xi(\tau)), \qquad
\xi(\tau)=\lambda^{-2}(\tau)\xi_0
\]
Hence, we are reduced  to solving the ODE
\begin{equation}
\label{eq:inhom} \Big[-\pr_\tau^2 + \omega(\tau) \pr_\tau
-\lambda^{-2}(\tau)\xi_0\Big]\lambda^{-2}\,
y(\tau,\xi(\tau))=\lambda^{-2}\, h (\tau,\xi(\tau))
\end{equation}
with $\xi_0>0$ fixed.  The homogeneous equation has exact solutions
\[
\Big[-\pr_\tau^2 + \omega(\tau)\pr_\tau
-\lambda^{-2}(\tau)\xi_0\Big]  e^{\pm i
\xi_0^{\frac12}\int_\tau^\infty \lambda^{-1}(\sigma)\,d\sigma}=0
\]
This is no surprise since the equation \eqref{eq:x} is equivalent
to the constant coefficient wave equation in the $t,r$ coordinates.

Since the Wronskian
\[
W\Big( e^{i \xi_0^{\frac12}\int_\tau^\infty
\lambda^{-1}(\sigma)\,d\sigma},  e^{- i
\xi_0^{\frac12}\int_\tau^\infty
\lambda^{-1}(\sigma)\,d\sigma}\Big)=2i\xi_0^{\frac12}\lambda^{-1}(\tau),
\]
it follows that  the backward solution to \eqref{eq:cont_eq'} has the form
\[
y(\tau,\xi_0)=\xi_0^{-\frac12}\int_\tau^\infty
\frac{\lambda^2(\tau)}{\lambda(\sigma)}\sin\Big(
\xi_0^{\frac12}\int_\tau^\sigma \lambda^{-1}(u)\,du \Big)\,
h(\sigma,\xi(\sigma))\, d\sigma
\]
Define the forward Green function
\[
 U(\tau,\sigma;\xi):= \xi^{-\frac12}
\frac{\lambda(\tau)}{\lambda(\sigma)}\sin\Big(
\xi^{\frac12}\lambda(\tau)\int_\tau^\sigma \lambda^{-1}(u)\,du \Big)
\]
Since $\xi_0=\lambda^2(\tau)\xi$,
$\xi(\sigma)=\xi\lambda^2(\tau)\lambda^{-2}(\sigma)$, we can write
\[
 y(\tau,\xi) = \int_\tau^\infty
U(\tau,\sigma;\xi)\,  h(\sigma,\xi(\sigma))\, d\sigma
\]
To estimate $D_\tau y$ it is also convenient to evaluate directly
\begin{equation}
D_\tau U(\tau,\sigma;\xi) = \frac{\lambda(\tau)}{\lambda(\sigma)}\cos\Big(
\xi^{\frac12}\lambda(\tau)\int_\tau^\sigma \lambda^{-1}(u)\,du \Big)
\label{dtauy}\end{equation}

To estimate the solution $y$ we either bound $|\sin(v)|\le |v|$ or
$|\sin(v)|\le1$.  Using that
\[
\lambda(\tau)\int_\tau^\infty \lambda^{-1}(u)\,du \les
\omega^{-1}(\tau)
\]
one obtains
\begin{equation}
|U(\tau,\sigma;\xi)|\lesssim \omega^{-1}(\tau)
\frac{\lambda(\tau)}{\lambda(\sigma)}
\label{fUbd}\end{equation}
as well as
\begin{equation}
\xi^\frac12 |U(\tau,\sigma;\xi)| +
|D_\tau U(\tau,\sigma;\xi)| \lesssim
\frac{\lambda(\tau)}{\lambda(\sigma)}
\label{dUbd}\end{equation}

We denote
\begin{equation}
z(\tau,\xi) = \left(\omega(\tau) y(\tau,\xi), \xi^\frac12 y(\tau,\xi),
D_\tau y(\tau,\xi)\right)
\label{zdef}\end{equation}
An immediate consequence of  \eqref{fUbd} and \eqref{dUbd}
is the estimate
\begin{equation}
\lambda^{-1}(\tau)  |z(\tau,\xi(\tau))|
  \lesssim \int_\tau^\infty \lambda^{-1}(\sigma) |h(\sigma,\xi(\sigma))| \,d\sigma
\label{pointest}
\end{equation}
From this we need to conclude that the following four bounds hold:
\begin{equation}
\| z\|_{L^\infty_{N-1} L^2 }
\lesssim \frac{1}{{N}} \|h\|_{L^\infty_{N} L^2}
\label{zh1}\end{equation}
\begin{equation}
\| z\|_{l^\infty_{N-1} L^2_{1/m}}
\lesssim \frac{1}{\sqrt{N}} \|h\|_{L^\infty_{N} L^2}
\label{zh2}\end{equation}
\begin{equation}
\| z\|_{L^\infty_{N-1} L^2 }
\lesssim \|h\|_{l^\infty_{N-\frac{1}{\beta+1}} L^2_m}
\label{zh3}\end{equation}
respectively
\begin{equation}
\| z\|_{l^\infty_{N-1} L^2_{1/m}}
\lesssim \|h\|_{l^\infty_{N-\frac{1}{\beta+1}} L^2_m}
\label{zh4}\end{equation}

Taking $L^2$ norms in $\xi$ on both sides of \eqref{pointest}
we obtain
\begin{equation*}
\|z(\tau)\|_{L^2} \lesssim
\int_\tau^\infty \|h(\sigma)\|_{L^2}\,d\sigma
\end{equation*}
which leads directly to \eqref{zh1}.

Adding flow invariant weights to the above bounds we get
\[
 \| z(\tau)\|_{L^2_{1/m}} \lesssim
\int_\tau^\infty \left\|m^{-1}\left(\frac{\xi
      \lambda(\tau)}{\lambda(\sigma)}\right) h(\sigma)\right\|_{L^2}\,d\sigma
\]
and by Cauchy-Schwarz
\[
 \| z(\tau)\|_{L^2_{1/m}}^2  \lesssim
\frac1N \int_\tau^\infty
\frac{\sigma^N}{\tau^{N-1}} \left\|m^{-1}\left(\frac{\xi
      \lambda(\tau)}{\lambda(\sigma)}\right) h(\sigma)\right\|_{L^2}^2\,d\sigma
\]
Hence
\[
\begin{split}
\int_{\tau_1}^{2\tau_1} &\tau^{2(N-1)-\frac{\beta}{\beta+1}}
\|z(\tau)\|_{L^2_{1/m}}^{2}
d\tau \\ \lesssim &\
\frac1N \int_{\tau_1}^{2\tau_1} \int_{\tau}^\infty
\int_0^\infty \tau^{2(N-1)-\frac{\beta}{\beta+1}}
\frac{\sigma^N}{\tau^{N-1}} m^{-2}\left(\frac{\xi
      \lambda(\tau)}{\lambda(\sigma)}\right) |h(\sigma,\xi)|^2 \,d\xi d\sigma d\tau
 \\ \lesssim &\ \frac1N
\int_{\tau_1}^\infty \sigma^N \|h(\sigma)\|_{L^2}^2 \sup_{\xi > 0}\left(
\int_{\tau_1}^{\min\{\sigma,2\tau_1\}}
 \tau^{N-1-\frac{\beta}{\beta+1}}
 m^{-2}\left(\frac{\xi
      \lambda(\tau)}{\lambda(\sigma)}\right)
 d\tau\right)\, d\sigma
\\ \lesssim &\ \frac{M}{N} \int_{\tau_1}^\infty  \sigma^N
\min\{\sigma,2\tau_1\}^{N-1}
\|h(\sigma)\|_{L^2}^2 \,d\sigma
\\ \lesssim &\ \frac{M}{N}  \| h\|_{L^\infty_N L^2}^2
\end{split}
\]
where
\[
\begin{split}
M =
\sup_{\xi > 0}  \int_{0}^{\infty}
 \tau^{-\frac{\beta}{\beta+1}}
 m^{-2}(\xi \lambda(\tau)) d\tau \approx \sup_{\xi > 0}  \int_{0}^{\infty}
 m^{-2}(\xi \lambda) \frac{d\lambda}{\lambda} \approx 1
\end{split}
\]
This concludes the proof of \eqref{zh2}.

We now turn our attention to \eqref{zh3}, for which we need to take
$h \in l^\infty_{N-\frac{1}{1+\beta}} L^2_{m}$. From
\eqref{pointest} by Cauchy-Schwarz we obtain
\[
\lambda^{-2}(\tau)|z(\tau,\xi(\tau))|^2 \lesssim \!
 \int_\tau^\infty \!\! \lambda^{-2}(\sigma) m(\xi(\sigma))
  \omega^{-1}(\sigma) |h(\sigma,\xi(\sigma))|\, d\sigma \cdot
 \int_\tau^\infty\!\! \! \omega(\sigma) m^{-1}(\xi(\sigma)) \, d\sigma
\]
The second integral has size $O(1)$, therefore
\begin{equation}
\lambda^{-2}(\tau)|z(\tau,\xi(\tau))|^2 \lesssim
 \int_\tau^\infty \lambda^{-2}(\sigma) m(\xi(\sigma))
  \omega^{-1}(\sigma) |h(\sigma,\xi(\sigma))| \, d\sigma
\label{zhuse}\end{equation}
 Hence integrating
with respect to $\xi$ to obtain
\[
\| z(\tau)\|_{L^2}^2 \lesssim \int_{\tau}^\infty
\int_0^\infty m(\xi) \omega^{-1}(\sigma) |h(\sigma,\xi)|^2 \, d\xi
d\sigma
\]
This directly  implies that
\[
\begin{split}
\tau^{2(N-1)} \| z(\tau)\|_{L^2}^2  \lesssim
\| h\|_{l^2_{N-\frac1{\beta+1}}L^2_m}^2
\end{split}
\]
which gives \eqref{zh3}.

Finally, for \eqref{zh4}, from \eqref{zhuse} by using again
Cauchy-Schwarz we obtain
\[
\int_{0}^\infty \omega(\tau) \lambda^{-2}(\tau) |z(\tau,\xi(\tau))|^2 d\tau
\lesssim  \int_0^\infty \lambda^{-2}(\sigma) m(\xi(\sigma))
  \omega^{-1}(\sigma) |h(\sigma,\xi(\sigma))| \, d\sigma
\]
and integrating with respect to $\xi$,
\[
\int_{0}^\infty \omega(\tau) \| z(\tau)\|_{L^2_{1/m}}^2 d\tau
\lesssim
  \int_0^\infty
  \omega^{-1}(\sigma) \|h(\sigma)\|_{L^2_m}^2 \, d\sigma
\]
Since the equation \eqref{eq:inhom} is solved backward
in $\tau$, we can add any nondecreasing weight in the above
estimate. In particular we obtain
\[
\int_{\tau_1}^{2\tau_1}  \tau^{2(N-1)-\frac{\beta}{\beta+1}}
 \| z(\tau)\|_{L^2_{1/m}}^2 d\tau
\lesssim \int_{\tau_1}^\infty \min\{\sigma,2\tau_1\}^{2(N-1)+\frac{\beta}{\beta+1}}\|h(\sigma)\|_{L^2_m}^2 \, d\sigma
\]
Hence \eqref{zh4} follows.
\end{proof}

It remains to show that the right hand side terms in \eqref{xmain} are
perturbative. We solve the equation \eqref{xmain} iteratively
and seek a solution as the sum of the series
\begin{equation}
 \left[\begin{matrix} x_0 \\ x \end{matrix}\right]
= \left(\sum_{k=0}^\infty (TQ)^k \right) T  \left[\begin{matrix} g_0 \\  g \end{matrix}\right]
\label{series}\end{equation}
It remains to establish the convergence of the above series.
 By Lemmas~\ref{lem:x0_param},\ref{lem:xbd} the backward
solution operator $T$ for $P$, given by
\[
T = \left[\begin{matrix}
T_e & 0 \\ 0 & T_c
\end{matrix}\right],
\]
is bounded
\[
T: Y_0^N \times Y^N   \to X_0^N \times X^N
\]
Hence an easy way to establish the convergence of the
series in \eqref{series} would be to show that
\[
\|Q\|_{X_0^N \times X^N \to Y_0^N \times Y^N } \ll 1
\]
We can establish such a bound for certain
components of $Q$, but as a whole  $Q$  is not even bounded
in the above setting. Lacking this, a weaker but still sufficient
alternative would be to prove that
\[
\| T Q\|_{X_0^N \times X^N \to X_0^N \times X^N } < 1
\]
This is still not true, but we will establish a weaker bound,
namely
\begin{equation}
\| T Q\|_{X_0^N \times X^N \to X_0^N \times X^N } \lesssim 1
\label{tqbd}\end{equation}
This ensures that all the terms in the series in \eqref{series}
belong to $X_0^N \times X^N$. In order to ensure convergence
we will split $Q$ into two parts,
\[
Q = Q_{g} + Q_b
\]
The good component $Q_g$ contains most of $Q$ and satisfies a
favorable bound
\begin{equation}
\| T Q_g\|_{X_0^N \times X^N \to X_0^N \times X^N } \lesssim
\frac{1}{N}+ \tau_0^{-\delta}, \qquad \delta > 0
\label{tqgbd}\end{equation}
Here the constant on the right can be made arbitrarily small by
choosing $N$ and $\tau_0$ large enough. For the single
bad component $Q_b$ of $Q$ we cannot establish outright smallness.
However, we will show that for a large enough $n$ we have
\begin{equation}
\| (T Q_b)^n\|_{X_0^N \times X^N \to X_0^N \times X^N } \ll 1
\label{tqbbd}\end{equation}
Combining \eqref{tqgbd} and \eqref{tqbbd} it follows that
for  large enough  $N$ and $\tau_0$ we have
\[
\| (T Q)^n\|_{X_0^N \times X^N \to X_0^N \times X^N } \ll 1
\]
This ensures the convergence of the series in \eqref{series} in $X_0^N
\times X^N$.  Given the bounds in
Lemmas~\ref{lem:x0_param},\ref{lem:xbd}, the proof of
Proposition~\ref{propxg} is concluded. It remains to show that $Q$
admits a decomposition which satisfies \eqref{tqgbd} and
\eqref{tqbbd}.


We begin with the easiest part, namely
\[
Q_{ce} x_0 = \omega^2 \calR_{ce} x_0 -2 \omega \calK_{ce}
\partial_\tau x_0
\]
which will be included in $Q_g$. Since the kernel $K_{ce}$ is bounded
and rapidly decreasing at infinity we obtain
\[
\| 2 \omega \calK_{ce}
\partial_\tau x_0\|_{l^\infty_{N-\frac{1}{2(\beta+1)}} L^2_{\rho m}}
\lesssim \|\partial_\tau x_0\|_{L^\infty_{N- \frac{\beta}{\beta+1}}}
\]
which yields a $\tau^\frac{1}{2(\beta+1)}$ gain,
\begin{equation}
\| \omega \calK_{ce} \partial_\tau x_0\|_{Y^{N+\frac{1}{2(\beta+1)}}}
\lesssim \|x_0\|_{X^N_0}
\label{dx0}\end{equation}

For the second part $\omega^2 \calR_{ce} x_0$ of $Q_{ce} x_0$
 such a simple bound no longer suffices, and we need to use some
cancellations. The final result is somewhat similar to the one above,
in that it gains a power of $\tau$ provided that $\beta > 3/2$.

\begin{lemma}
The following estimate holds:
\begin{equation}
\| T \omega^2 \calR_{ce} x_0\|_{X^{N+\frac{2\beta -3}{2(\beta+1)}}}
\lesssim \|x_0\|_{X^N_0}
\end{equation}
\end{lemma}

\begin{proof}
Suppose that
\[
\|x_0\|_{X_0^N} =1
\]
  We denote $g = \omega^2 \calR_{ce} x_0$ and $x = T_c g$.  As before we
  also introduce the auxiliary variables $y= \rho^\frac12 x$ and $h =
  \rho^\frac12 g$. The kernel $R_{ce}$ of $ \calR_{ce}$ is bounded,
  rapidly decreasing at infinity and has symbol-like behavior at both
  $0$ and infinity.  Then for the function $h$ we directly estimate
\[
| \tau^{N-\frac{1}{\beta+1}} (1+ \xi)
h(\tau)|  + |\tau^{N+\frac{\beta-1}{\beta+1}} (1+ \xi)
(\partial_\tau - 2\omega \xi \partial_\xi)
h(\tau)| \lesssim \| x_0\|_{X_0^N} =1
\]
As in the proof of Lemma~\ref{lem:xbd} we have
\[
y(\tau,\xi(\tau)) = \int_{\tau}^\infty U(\tau,\sigma,\xi(\tau))
h(\sigma,\xi(\sigma)) \, d\sigma,
\]
where
\[
   U(\tau,\sigma,\xi) =\xi(\tau)^{-\frac12}
\frac{\lambda(\tau)}{\lambda(\sigma)}
\sin\left(\xi(\tau)^\frac12 \lambda(\tau)
\int_{\tau}^\sigma \lambda^{-1}(u) du \right)
\]
Hence for $y$ we use \eqref{fUbd} amd \eqref{dUbd} to  obtain the pointwise bound
\[
|y (\tau,\xi(\tau)) | \lesssim  \xi(\tau)^{-\frac12} \min\{
1,\xi(\tau)^{\frac12} \omega(\tau)^{-1}\} \int_{\tau}^\infty\frac{\lambda(\tau)}{\lambda(\sigma)}
\sigma^{-N+\frac{1}{\beta+1}} (1+\xi(\sigma))^{-1}
 \, d\sigma
\]
which we rewrite in the form
\begin{equation}
 \omega(\tau)|y (\tau,\xi) | \lesssim \xi^{-\frac12} \min\{
1,\xi^{\frac12} \omega(\tau)^{-1}\} \int_{\tau}^\infty
\frac{\lambda(\tau)}{\lambda(\sigma)}
\frac{\sigma^{-N-\frac{\beta-1}{\beta+1}}}{(1+\xi \lambda^2(\tau) \lambda^{-2}(\sigma))
} \, d\sigma
\label{cea}\end{equation}

To bound $\xi^\frac12 y$  we integrate by parts,
\[
\begin{split}
y(\tau,\xi(\tau)) = & \
 \xi(\tau)^{-1} \int_\tau^\infty
h(\sigma,\xi(\sigma))
\partial_\sigma
\left(1-\cos\left(\xi(\tau)^\frac12 \lambda(\tau) \int_{\tau}^\sigma
\lambda^{-1}(u) du \right)\right) \, d\sigma
\\
= &\   \xi(\tau)^{-1} \int_\tau^\infty
\left(\cos\left(\xi(\tau)^\frac12 \lambda(\tau) \int_{\tau}^\sigma\!\!
\lambda^{-1}(u) du \right)-1\right) \partial_\sigma
h(\sigma,\xi(\sigma))
\, d\sigma
\end{split}
\]
Estimating either $|1-\cos v| \lesssim 1$ or $|1-\cos v| \lesssim |v|$
this leads to a bound which is weaker than \eqref{cea}, namely
\begin{equation}
 \xi^\frac12 |y (\tau,\xi) | \lesssim  \xi^{-\frac12} \min\{
1,\xi^{\frac12} \omega(\tau)^{-1}\} \int_{\tau}^\infty
\frac{\sigma^{-N-\frac{\beta-1}{\beta+1}}}{(1+\xi \lambda^2(\tau) \lambda^{-2}(\sigma))
}\, d\sigma
\label{ceb}\end{equation}

In a similar manner we evaluate $D_\tau y$,
\[
\begin{split}
D_\tau y(\tau,\xi(\tau)) = & \
  \int_\tau^\infty\frac{\lambda(\tau)}{\lambda(\sigma)}
h(\sigma,\xi(\sigma))
\cos\left(\xi(\tau)^\frac12 \lambda(\tau) \int_{\tau}^\sigma
\lambda^{-1}(u) du \right)\, d\sigma
\\
= &\   \xi(\tau)^{-\frac12} \int_\tau^\infty
\sin\left(\xi(\tau)^\frac12 \lambda(\tau) \int_{\tau}^\sigma\!\!
\lambda^{-1}(u) du \right) \partial_\sigma
h(\sigma,\xi(\sigma))
\, d\sigma
\end{split}
\]
which leads to the same bound as in \eqref{ceb}. Summing up,
for $z$ as in \eqref{zdef} we obtain
\begin{equation}
 |z (\tau,\xi) | \lesssim  \xi^{-\frac12} \min\{
1,\xi^{\frac12} \omega(\tau)^{-1}\} \int_{\tau}^\infty
\frac{\sigma^{-N-\frac{\beta-1}{\beta+1}}}{(1+\xi \lambda^2(\tau) \lambda^{-2}(\sigma))
} \, d\sigma
\label{cec}\end{equation}
It remains to evaluate the integral on the right.  If $\xi < 2$ then
we can neglect the first factor in the denominator of the integrand
and evaluate
\[
\int_{\tau}^\infty
\frac{\sigma^{-N-\frac{\beta-1}{\beta+1}}}{(1+\xi \lambda^2(\tau) \lambda^{-2}(\sigma))}
 \, d\sigma \lesssim \tau^{-N+\frac{2}{\beta+1}}, \qquad
\xi \leq 2
\]
However, if $\xi > 2$ then this factor yields rapid decay when
\[
\xi \lambda^2(\tau) \lambda^{-2}(\sigma) > 1
\]
which corresponds to
\[
\sigma \lesssim \tau + (\ln \xi)^{\beta+1}
\]
Thus we obtain
\[
\int_{\tau}^\infty
\frac{\sigma^{-N-\frac{\beta-1}{\beta+1}}}{(1+\xi \lambda^2(\tau) \lambda^{-2}(\sigma))}\, d\sigma \lesssim
(\tau+ (\ln \xi)^{\beta+1})^{-N+\frac{2}{\beta+1}}, \qquad
\xi \geq 2
\]
Summing up, for $z$ we have obtained the pointwise bound
\[
|z (\tau,\xi) | \lesssim \left\{ \begin{array}{lc}
\tau^{-N+\frac{\beta+2}{\beta+1}} &  \xi < \omega^2(\tau)
\cr
\xi^{-\frac12} \tau^{-N+\frac{2}{\beta+1}} &  \omega^2(\tau) \leq \xi \leq 2
\cr
\xi^{-\frac12} (\tau+ (\ln \xi)^{\beta+1})^{-N+\frac{2}{\beta+1}} & \xi \geq 2
\end{array}
\right.
\]

This allows us to estimate $L^2$ norms, namely
\[
\| z(\tau,\xi)\|_{L^2_{1/m}} \lesssim   \tau^{-N+\frac{2}{\beta+1}}
\]
respectively
\[
\|z(\tau,\xi)\|_{L^2} \lesssim   \tau^{-N+\frac{5}{2(\beta+1)}}
\]
Finally we obtain
\[
\| z\|_{L^\infty_{N-\frac{5}{2(\beta+1)}} L^2
\cap l^\infty_{N-\frac{5}{2(\beta+1)}} L^2_{1/m}}
\lesssim 1
\]
and the conclusion of the lemma follows.

\end{proof}

Next, we turn to the term $\calQ_{ec} x$ given by
\[
\calQ_{ec}  x = \omega^2 \calR_{ec} x
-2\omega \calK_{ec} D_\tau x
\]
We will prove that $\calQ_{ec} x$ can also be included in $Q_g$.  The
kernel $R_{ec}(\xi)$ of $\calR_{ec}$ is bounded and decays rapidly at
infinity. Then the contribution of the first term is easy to estimate
using the $L^\infty_N L^2_\rho$ type bounds on $x$ and $\xi^\frac12 x$,
\begin{equation}
\| \omega^2 R_{ec} x\|_{Y^{N+\frac{\beta-1}{\beta+1}-\delta}_0}
\lesssim \|x\|_{X^N}, \qquad \delta > 0
\end{equation}
with $\delta$ arbitrarily small.
The bound for  the second term in $\calQ_{ec}  x$ is similar:

\begin{lemma}
For $\delta > 0$ we have
\begin{equation}
  \| T_0 \omega \calK_{ec} D_\tau x\|_{X^{N+\frac{\beta-1}{\beta+1} -\delta }_0 }
\lesssim \|x\|_{X^N}
\end{equation}
\end{lemma}
\begin{proof}
  Set $y = \rho^\frac12 x$. The solution $x_0 = T_0 \omega \calK_{ec}
  x$ is represented as
\[
\begin{split}
x_0(\tau) = &\ \int_\tau^\infty\int_0^\infty U_0(\tau,\sigma)  \omega K_{ec}(\xi)
D_\sigma x(\sigma,\xi) \, d\xi d\sigma
\\
= & \ \int_\tau^\infty \int_0^\infty U_0(\tau,\sigma)  \omega K_{ec}(\xi)
\left(\partial_\sigma - 2 \omega (\xi \partial_\xi + 1)\right) y(\sigma,\xi) \, d\xi d\sigma
\end{split}
\]
Integrating by parts we obtain
\[
x_0(\tau) =  \int_\tau^\infty\int_0^\infty  - \partial_\sigma U_0(\tau,\sigma)
\omega(\sigma) K_{ec}(\xi)  y(\sigma,\xi) + 2 U_0(\tau,\sigma)
\omega^2(\sigma)  y(\sigma,\xi) \xi \partial_\xi  K_{ec}(\xi)\, d\xi d\sigma
\]
Hence
\[
|x_0(\tau)| \leq \int_\tau^\infty\int_0^\infty  \omega(\sigma)
(1+\xi)^{-1} | x(\sigma,\xi)|\, d\xi d\sigma
\]
therefore
\[
\| \tau^{N-1-\frac1{\beta+1}-\delta} x_0\|_{L^2} \lesssim \|x\|_{X^N}
\]
A similar computation yields
\[
|\partial_\tau x_0(\tau)| \leq \int_\tau^\infty \int_0^\infty \omega^2(\sigma)
\left( \tau^{-\frac{1}{\beta+1}} + \frac{\lambda(\tau)}{\lambda(\sigma)}\right)
(1+\xi)^{-1} | x(\sigma,\xi)|\, d\xi d\sigma
\]
which leads to
\[
\| \tau^{N-\frac1{\beta+1}-\delta} \partial_\tau x_0\|_{L^2} \lesssim \|x\|_{X^N}
\]
The desired conclusion follows.

\end{proof}

Finally we consider the expression $Q_{cc} x$ which has the form
\begin{equation}
Q_{cc} x = \omega^2 \calR_{cc} x  - 2\omega \calK_{0}D_\tau
\end{equation}
The first term is better behaved and can be included in $Q_g$:

\begin{lemma}
For $\delta > 0$ we have the following bound:
\begin{equation}
\| \omega^2 \calR_{cc} x\|_{Y^{N+\frac{\beta-1}{\beta+1}-\delta}}
\lesssim \|x\|_{X^N}
\end{equation}
\end{lemma}
\begin{proof}
By the definition of the $X^N$ and $Y^N$ spaces, it suffices to show that
\[
\| \calR_{cc} x\|_{L^2_\rho} \lesssim \tau^\delta(\| \xi^\frac12 x\|_{L^2_\rho} +
\tau^{-\frac{\beta}{\beta+1}} \| x\|_{L^2_\rho})
\]
This in turn follows by duality and dyadic summation from the
bound
\begin{equation}
\|\chi_{[0,h)}\calR_{cc}^* f\|_{L^2_\rho} \lesssim \min(h^{\frac12-\delta},1)\|f\|_{L^2_\rho}
\end{equation}
For this we need to prove that the operator $\calR_{cc}^*$ is
quasi-smoothing according to the following definition:

\begin{defi}
  \label{def:smoothing} A bounded operator $T:L^2_\rho(\R^+)\to
  L^2_\rho(\R^+)$ is {\em quasi-smoothing} if for each $\delta > 0$
  there exists $C_\delta > 0 $ so that
  \begin{equation}\label{eq:smooth}
  \begin{aligned}
\|\chi_{[0,h)}\, Tf\|_{L_\rho^2} &\le C_\delta \,\min(h^{\frac12-\delta},1)\|f\|_{L^2_\rho}\\
\end{aligned}
  \end{equation}
  for all $h>0$.
\end{defi}

We remark that the quasi-smoothing operators form an ideal under
composition from the right within the algebra of bounded operators.
Hence, given the expression of $\calR_{cc}$, it suffices to show that
the following operators are quasi-smoothing:
\[
\calK_0, \quad [\xi \partial_\xi, \calK_0], \quad \calK_{ce} \calK_{ec}, \quad J_{cc}
\]

  Recall that
  \[
\calK_0 f(\xi) = \int_0^\infty
\frac{\rho(\xi)F(\xi,\eta)}{\xi-\eta}f(\eta)\,d\eta
  \]
where \[ |F(\xi,\eta)|\les
\min\big[\xi+\eta,(\xi+\eta)^{-\frac52}(1+|\xi^{\frac12}-\eta^{\frac12}|)^{-N}\big]
\] Let $ F_1(\xi,\eta):=\rho(\xi)F(\xi,\eta)$. Then
\[
\calK_0 f(\xi)= \int_0^\infty \frac{
F_1(\xi,\eta)}{\xi-\eta}\chi_{[\frac{\xi}{\eta}\in[\frac12,2]]}\,f(\eta)\,d\eta
+\int_0^\infty \frac{
F_1(\xi,\eta)}{\xi-\eta}\chi_{[\frac{\xi}{\eta}\not\in[\frac12,2]]}\,f(\eta)\,d\eta
\]
For the first operator on the right-hand side one has
\begin{equation}\label{eq:F1_bd}
|F_1(\xi,\eta)| \chi_{[\frac{\xi}{\eta}\in[\frac12,2]]} \le
\min(\xi,1) \end{equation} which implies that the corresponding
operator is quasi-smoothing, see the proof of $L^2_\rho$ boundedness
of $\calK_0$ in the previous section. For the second operator,
observe that
\[
\frac{|
F_1(\xi,\eta)|}{|\xi-\eta|}\chi_{[\frac{\xi}{\eta}\not\in[\frac12,2]]}
\les \min(1,(\xi+\eta)^{-N})
\]
by the rapid off-diagonal decay of $F$. Hence,
\[
\sup_{\xi\ge0}\Big| \int_0^\infty \frac{
F_1(\xi,\eta)}{\xi-\eta}\chi_{[\frac{\xi}{\eta}\not\in[\frac12,2]]}\,f(\eta)\,d\eta
\Big|\les \|f\|_{L^2_\rho}
\]
and therefore
\[
\Big\|\chi_{[0,h)} \int_0^\infty \frac{
F_1(\xi,\eta)}{\xi-\eta}\chi_{[\frac{\xi}{\eta}\not\in[\frac12,2]]}\,f(\eta)\,d\eta
\Big\|_{L^2}\les h^{\frac12}\|f\|_{L^2_\rho}
\]
as desired.

For the commutator $[\xi \partial_\xi, \calK_0]$ we have
\[
[\xi\pr_\xi,\calK_0]f(\xi)=\int_0^\infty\frac{(\xi\pr_\xi+\eta\pr_\eta)
F_1(\xi,\eta)}{\xi-\eta}\,f(\eta)\,d\eta
\]
and one argues as before.

The kernel of operator $\calK_{ce} \calK_{ec}$ is $\rho(\xi) K_e(\xi) K_e(\eta)$,
and is bounded by  $(1+\xi)^{-n}(1+\eta)^{-n}$. The  quasi-smoothing property easily
follows. Finally $\calJ_{cc}$ is  quasi-smoothing due to the kernel bound
\eqref{jcc}.

\end{proof}

It remains to consider the second part of $\calQ_{cc}$
namely the expression $\omega \calK_{0}D_\tau x$.
Since the kernel for $\calK_0$ decays at $0$ and at infinity,
it is easy to establish the bound
\[
\| \calK_0\|_{L^2_{\rho/m} \to L^2_{\rho m}} \lesssim 1
\]
It follows that
\begin{equation}
\| \omega \calK_0 D_\tau x\|_{Y^N} \lesssim \|x\|_{X^N}
\label{jf}\end{equation}
The difficulty is that there is no smallness in the above relation,
and it is not possible to gain any smallness by letting $\tau$ be
large enough. To deal with this we reiterate:

\begin{lemma}
Suppose that $n$ is large enough. Then
\begin{equation}
\| (T \omega \calK_0 D_\tau)^n x\|_{X^N} \ll \|x\|_{X^N}.
\label{jf1}\end{equation}
\end{lemma}

\begin{proof}
By \eqref{jf} and Lemma~\ref{lem:xbd}
it suffices to prove that for large enough $n$,
\begin{equation}
 \|( \omega \calK_0 D_\tau T)^n g\|_{l^\infty_N L^2_{\rho m}} \ll
\|g\|_{l^\infty_N L^2_{\rho m}}
\label{jf2}\end{equation}
We divide the operator $ \calK_0$ in two parts,
\[
\calK_0 = \calK_0^{d} +  \calK_0^{nd}
\]
with kernels
\[
K_0^{d} (\xi,\eta) = K_0(\xi,\eta) \chi_{[|\frac{\xi}{\eta}-1|<\frac{1}{n}]}
,\qquad
K_0^{nd} (\xi,\eta) = K_0(\xi,\eta) \chi_{[|\frac{\xi}{\eta}-1|> \frac{1}{n}]}
\]

The contribution of $\calK_0^{nd}$ is non-resonant,  we
and we expect to gain powers of $\tau$ from oscillations.
Precisely, we will prove that
\begin{equation}
\| ( \omega \calK_0 D_\tau T) ( \omega \calK_0^{nd} D_\tau T)
g\|_{l^\infty_{N+\frac{\beta-2}{\beta+1}} L^2_{\rho m}} \lesssim_n
\|g\|_{l^\infty_N L^2_{\rho m}}
\label{jf3}\end{equation}
Here the implicit constant depends on $n$, but that is not important
since we gain a power of $\tau$.

Assuming \eqref{jf3} holds, in order to prove \eqref{jf2} it
remains to show that for large $n$ we have
\begin{equation}
 \|( \omega \calK_0^{d} D_\tau T)^n g\|_{l^\infty_N L^2_{\rho m}} \ll
\|g\|_{l^\infty_N L^2_{\rho m}}
\label{jf4}\end{equation}

{\bf Proof of \eqref{jf4}:} For another small parameter $\epsilon$ to
be chosen later we further divide $ \calK_0^{d}$ into three parts,
\[
\calK_0^{d} = \calK_{0,1}^{d,\epsilon} +  \calK_{0,2}^{d,\epsilon}
+ \calK_{0,3}^{d,\epsilon}
\]
with kernels
\[
K_{0,1}^{d,\epsilon}(\xi,\eta) = 1_{\xi < \epsilon} K_0^{d}(\xi,\eta),
\qquad
K_{0,3}^{d,\epsilon}(\xi,\eta) = 1_{\xi > \epsilon^{-1}} K_0^{d}(\xi,\eta)
\]
The center part $ \calK_{0,2}^{d,\epsilon}$ enjoys better
localization, while the two tails $ \calK_{0,1}^{d,\epsilon}$ and $\calK_{
0,3}^{d,\epsilon}$ are small. Precisely,
\begin{equation}
\| \omega \calK_{0,1}^{d,\epsilon} D_\tau T g\|_{l^\infty_N L^2_{\rho m}} +
\| \omega \calK_{0,3}^{d,\epsilon} D_\tau T g\|_{l^\infty_N L^2_{\rho m}}\lesssim
\epsilon^\frac14 \|g\|_{l^\infty_N L^2_{\rho m}}
\label{jf5}\end{equation}

It is easy to see that due to the supports of the kernels we have
\[
\calK_{0,1}^{d,\epsilon} D_\tau T  \calK_{0,2}^{d,\epsilon(1+\frac1n)}
= 0
\]
and
\[
 \calK_{0,2}^{d,\epsilon(1+\frac1n)}  D_\tau T  \calK_{0,3}^{d,\epsilon(1+\frac1n)}
= 0
\]
Hence we obtain the decomposition
\[
\begin{split}
& ( \omega \calK_0^{d} D_\tau T)^n =
 ( \omega \calK_{0,2}^{d,\epsilon} D_\tau T)^n
\\
& + \sum_{k=1}^n  ( \omega \calK_{0,2}^{d,\epsilon} D_\tau T)^{k-1}
 ( \omega \calK_{0,1}^{d,\epsilon} D_\tau T)
( \omega \calK_{0,1}^{d,2\epsilon} D_\tau T)^{n-k}
\\
& + \sum_{j=1}^n  ( \omega \calK_{0,3}^{d,2\epsilon} D_\tau T)^{j-1}
 ( \omega \calK_{0,3}^{d,\epsilon} D_\tau T)
 ( \omega \calK_{0,2}^{d,\epsilon} D_\tau T)^{n-j}
\\
& + \!\!\!\! \sum_{1 \leq j < k \leq n} \!\!\!\!\!\!
  ( \omega \calK_{0,3}^{d,2\epsilon} D_\tau T)^{j-1} \!
 ( \omega \calK_{0,3}^{d,\epsilon} D_\tau T)
 ( \omega \calK_{0,2}^{d,\epsilon} D_\tau T)^{k-j-1} \!
 ( \omega \calK_{0,1}^{d,\epsilon} D_\tau T)
( \omega \calK_{0,1}^{d,2\epsilon} D_\tau T)^{n-k}
\end{split}
\]
For the middle part we will prove the bound
\begin{equation}
 \|( \omega \calK_{0,2}^{d,\epsilon} D_\tau T)^k g\|_{l^\infty_N L^2_{\rho m}} \leq
\frac{(C|\log \epsilon|)^k}{(k-1)!}
\|g\|_{l^\infty_N L^2_{\rho m}}
\label{jf6}\end{equation}
Combining this with \eqref{jf5} we obtain
\[
\begin{split}
\|( \omega \calK_0^{d} D_\tau T)^n g\|_{l^\infty_N L^2_{\rho m}} \leq &\
\left(\sum_{k=0}^n \frac{(C|\log \epsilon|)^k}{(k-1)!}
\epsilon^{\frac{n-k}4} \right)
\|g\|_{l^\infty_N L^2_{\rho m}}
\end{split}
\]
Choosing $\epsilon = n^{-4}$ this gives
\[
\|( \omega \calK_0^{d} D_\tau T)^n g\|_{l^\infty_N L^2_{\rho m}} \leq \frac{(C \log n)^n}{n^{n-2}}
\|g\|_{l^\infty_N L^2_{\rho m}}
\]
for a new constant $C$.  Thus \eqref{jf4} is established for $n$
sufficiently large.

We return to prove \eqref{jf6}. Since
\[
D_\tau U(\tau,\sigma,\xi) := V(\tau,\sigma,\xi) =
\frac{\lambda(\tau)}{\lambda(\sigma)}\cos\left(\xi^\frac12 \lambda(\tau) \int_{\tau}^\sigma
  \lambda^{-1}(\tau) \right)
\]
we can write the function
\[
y(\tau,\xi) =  \rho(\xi)^\frac12
 ( \omega \calK_{0,2}^{d,\epsilon} D_\tau T)^n g
\]
in the form
\[
\begin{split}
y(\tau,\xi) =&\ \int_{\tau}^\infty\int_{\R^+} d \sigma_1 d \eta_0
\omega \calK_{0,2}^{d,\epsilon}(\xi,\eta_0) V(\tau,\sigma_1,\eta_0)
\\ &\
\int_{\sigma_{1}}^\infty \int_{\R^+} d \sigma_2 d \eta_1
\omega \calK_{0,2}^{d,\epsilon}
\left(\eta_{0}
\frac{\lambda^2(\sigma_{0})}{\lambda^2(\sigma_{1})},\eta_{1}\right)
V(\sigma_{1},\sigma_2,\eta_{1}) \cdots
 \\ &\
 \int_{\sigma_{n-1}}^\infty \int_{\R^+} d \sigma_n d \eta_n
 \omega \calK_{0,2}^{d,\epsilon}
 \left(\eta_{n-2}
 \frac{\lambda^2(\sigma_{n-2})}{\lambda^2(\sigma_{n-1})},\eta_{n-1}\right)
 V(\sigma_{n-1},\sigma_n,\eta_{n-1})
 \\
 &\ \int_{\R^+} d\eta_n \omega \calK_{0,2}^{d,\epsilon}
 \left(\eta_{n-1}
 \frac{\lambda^2(\sigma_{n-1})}{\lambda^2(\sigma_{n})},\eta_{n}\right)
   h(\sigma_n, \eta_{n}) \frac{\lambda(\tau)}{\lambda(\sigma_m)}
\end{split}
\]
In order for the above integrand to be nonzero we must have
\[
\left|\frac{ \eta_k \lambda^2(\sigma_k)}{ \eta_{k+1}
    \lambda^2(\sigma_{k+1})}-1\right| \leq \frac1n, \qquad
\left|\frac{\xi}{\eta_0}-1\right| \leq \frac1n.
\]
This implies that
\[
\eta_n \lambda^2(\sigma_n) \leq 3 \xi \lambda^2(\tau)
\]
Since $\epsilon < \sigma_n,\xi \leq \epsilon^{-1}$ it follows that
\[
 \lambda^2(\sigma_n) \leq 3 \epsilon^2 \lambda^2(\tau)
\]
If $\tau$ is sufficiently large this implies that
\[
\sigma_n \leq \sigma(\tau) =
\tau + C \tau^\frac{\beta}{\beta+1} |\log \epsilon|^{{\beta+1}}
\]
Using the $L^2$ boundedness of $\calK_{0,2}^{d,\epsilon}$
and of the transport along the flow (as $|V| \leq 1$) it follows that
\[
\| y(\tau)\|_{L^2_{\rho m}} \leq m^2(\epsilon)(C\omega(\tau))^{n+1} \int_{\tau}^{\sigma(\tau)} d \sigma_1
\int_{ \sigma_1}^{\sigma(\tau)} d \sigma_2 \cdots
\int_{\sigma_{n-1}}^{\sigma(\tau)}    \| h(\sigma_n)\|_{L^2_{\rho/m}} \, d\sigma_n
\]
Changing the order of integration this yields
\[
\| y(\tau)\|_{L^2_{\rho m}} \leq m^2(\epsilon)
\frac{(C\omega(\tau))^n}{(n-1)!} \int_{\tau}^{\sigma(\tau)}
 (\tau-\sigma_n)^{n-1} \| h(\sigma_n)\|_{L^2_{\rho/m}}\, d\sigma_n
\]
Since
\[
\int_{\tau}^{\sigma(\tau)}  (\tau-\sigma_n)^{n-1} \, d\sigma^n \approx
\frac{1}n \omega(\tau)^{-n}
\]
we finally obtain
\[
\| y\|_{l^\infty_{N+\frac{\beta}{\beta+1}} L^2_{\rho m}} \leq m^2(\epsilon)
\frac{(C|\log \epsilon|^{\beta+1})^n}{n!}
  \| h\|_{l^\infty_N L^2_{\rho/m} }
\]
Thus \eqref{jf6} is proved.

{\bf Proof of \eqref{jf3}:} Denoting $x = T ( \omega \calK_0^{nd}
D_\tau T) g$, $y = \rho^\frac12 x$ and $h = \rho^\frac12 g$ we
need to prove that
\begin{equation}
\| D_\tau y\|_{l^\infty_{N-\frac{2}{\beta+1}} L^2_{1/m}} \lesssim
\| h\|_{l^\infty_{N} L^2_m}
\end{equation}
Due to the formula \eqref{dtauy} we
have the integral representation

\[
\begin{split}
  D_\tau y(\tau,\xi) = &\ \int_\tau^\infty \omega(s)
  \frac{\lambda(\tau)}{\lambda(s)} \cos\left(\xi^\frac12 \lambda(\tau)
    \int_{\tau}^s \lambda^{-1}(\theta ) d\theta \right) \int_0^\infty
  K_0^{nd} (\xi(s),\eta(s)) \frac{\lambda^2(\sigma)}{\lambda^2(s)} \\
  &\int_s^\infty \frac{\lambda(s)}{\lambda(\sigma)}
  \cos\left(\eta^\frac12 \lambda(\sigma) \int_{s}^\sigma
    \lambda^{-1}(\theta) d\theta \right) y(\sigma,\eta) d \sigma d
  \eta ds
\end{split}
\]
where $\xi(s) = \xi \frac{\lambda^2(\tau)}{\lambda^2(s)}$ and
$\eta(s)= \eta \frac{\lambda^2(\sigma)}{\lambda^2(s)}$.  In the
support of the kernel $K_0^{nd}$ we have $ \left|
  \frac{\xi(s)}{\eta(s)} - 1\right| > \frac1n $ therefore $ \left|
  \frac{\xi^\frac12 \lambda(\tau)}{\eta^\frac12 \lambda(\sigma)} -
  1\right| \gtrsim \frac1n $.  Thus the two oscillatory factors have
different frequencies, and we can gain if we integrate by parts with
respect to $s$. Denoting
\[
u(s) =\xi^\frac12 \lambda(\tau) \int_{\tau}^s \lambda^{-1}(\theta ) d\theta,
\qquad
v(s) = \eta^\frac12 \lambda(\sigma) \int_{s}^\sigma \lambda^{-1}(\theta) d\theta
\]
 we write
\[
2 \cos u(s) \cos v(s) = \cos(u(s)+v(s))+ \cos(u(s)-v(s))
\]

We change the order of integration in the above expression for $y$
and integrate by parts with respect to $s$.
Since
\[
\frac{d}{ds}(u\pm v) = \lambda^{-1}(s)(\xi^\frac12 \lambda(\tau) \mp
\eta^\frac12 \lambda(\sigma)) = \xi(s)^\frac12 \mp \eta(s)^\frac12
\]
we integrate the cosine and differentiate the rest to obtain
\[
\begin{split}
D_\tau y(\tau,\xi) = &  \sum_{\pm} \!
\int_\tau^\infty  \!\!\! \int_0^\infty \!\! \int_{\tau}^\sigma \!
 \frac{\lambda(\tau)}{\lambda(\sigma)}  \frac1{\xi(s)^\frac12  \mp
   \eta(s)^\frac12} \frac{1}{\lambda(s)}
\frac{d}{ds} \left( \omega(s) K_0^{nd} (\xi(s),\eta(s))
\frac{\lambda^2(\sigma)}{\lambda(s)}\right) \\ & \  \sin(u(s)+v(s))
h(\sigma,\eta) \, ds d\eta d\sigma
\\
& \!\! \pm \! \int_\tau^\infty \!\!  \int_0^\infty \!\! \omega(\tau) \frac{\lambda(\tau)}{\lambda(\sigma)} \frac1{\xi^\frac12 \mp \eta(\tau)^\frac12}
 K_0^{nd} (\xi,\eta(\tau))
\frac{\lambda^2(\sigma)}{\lambda^2(\tau)}   \sin (v(\tau))
h(\sigma,\eta)  \,  d\eta d\sigma
\\
&\!\!  - \int_\tau^\infty\!\! \int_0^\infty \omega(\sigma) \frac{\lambda(\tau)}{\lambda(\sigma)} \frac1{\xi(\sigma)^\frac12  \mp \eta^\frac12 }
 K_0^{nd} (\xi(\sigma),\eta)  \sin (u(\sigma))
h(\sigma,\eta) \,  d\eta d\sigma
\end{split}
\]
We have
\[
\frac{d}{ds} \left( K_0^{nd} (\xi(s),\eta(s)) \lambda^{-1}(s)\right)
= \omega(s) ( \xi \partial_\xi + \eta
  \partial_\eta -1) K_0^{nd} (\xi(s),\eta(s)) \lambda^{-1}(s)
\]
Due to Theorem~\ref{tp} the kernel $K^{nd}_0$
is bounded and decays rapidly at infinity therefore we can bound it by
\[
|K^{nd}_0(\xi,\eta)| \lesssim_n \frac1{(1+\xi)(1+\eta)}
\]
We also have
\[
| ( \xi \partial_\xi + \eta
  \partial_\eta -1) K^{nd}_0(\xi,\eta)| \lesssim_n \frac1{(1+\xi)(1+\eta)}
\]
Hence the following rough bounds are valid:
\[
\left| \frac{K^{nd}_0(\xi,\eta)}{\xi^\frac12 \pm \eta^\frac12}\right|
+ \left| \frac{( \xi \partial_\xi + \eta
  \partial_\eta -1) K^{nd}_0(\xi,\eta)}{\xi^\frac12 \pm \eta^\frac12}\right|
\lesssim_n \frac{1}{\xi^\frac12 \eta^\frac12}
\]
Inserting this in the bounds for $D_\tau y$ we obtain
\[
\begin{split}
  |D_\tau y(\tau,\xi)| \lesssim_n &\ \int_\tau^\infty \int_0^\infty
  \int_{\tau}^\sigma \omega^2(s) \frac{\lambda(\tau)}{\lambda(\sigma)}
  \frac{1}{\xi(s)^\frac12 \eta(s)^\frac12}
  \frac{\lambda^2(\sigma)}{\lambda^2(s)} |h(\sigma,\eta)|\, ds d\eta
  d\sigma
  \\
  &\ + \int_\tau^\infty \int_0^\infty \omega(\tau)
  \frac{\lambda(\tau)}{\lambda(\sigma)} \frac{1}{\xi^\frac12
    \eta(\tau)^\frac12} \frac{\lambda^2(\sigma)}{\lambda^2(\tau)}
  |h(\sigma,\eta)| \, d\eta d\sigma
  \\
  &\ + \int_\tau^\infty \int_0^\infty \omega(\sigma)
  \frac{\lambda(\tau)}{\lambda(\sigma)}
  \frac{1}{\xi(\sigma)^\frac12 \eta^\frac12} |h(\sigma,\eta)|\, d\eta d\sigma
\end{split}
\]
This can be rewritten in the form
\[
\begin{split}
\xi^\frac12  |D_\tau y(\tau,\xi)| \lesssim_n &\
 \int_\tau^\infty \int_0^\infty
  \int_{\tau}^\sigma \omega^2(s) \frac{|h(\sigma,\eta)|}{\eta^\frac12} \, ds d\eta
  d\sigma
  + \int_\tau^\infty \int_0^\infty \omega(\tau)
  \frac{|h(\sigma,\eta)|}{\eta^\frac12}\,  d\eta d\sigma
\\  \lesssim_n &\
 \omega^2(\tau) \int_\tau^\infty \int_0^\infty
  \sigma  \frac{|h(\sigma,\eta)|}{\eta^\frac12}  \, d\eta
  d\sigma
\end{split}
\]
Taking weighted $L^2$ norms we obtain
\[
\|D_\tau y(\tau,\xi)\|_{L^2_{1/m}}   \lesssim_n
 \omega^2(\tau) \int_\tau^\infty
  \sigma  \|h(\sigma)\|_{L^2_m}
  \, d\sigma
\]
Then by Cauchy-Schwarz
\[
\|D_\tau y(\tau)\|_{L^2_{1/m}}   \lesssim_n
 \tau^{-N+ \frac{3}{2(\beta+1)} }  \|h\|_{l^\infty_{N} L^2_m}
\]
and further
\[
\|D_\tau y\|_{l^\infty_{N-\frac{2}{\beta+1}}   L^2_{1/m}} \lesssim_n
 \|h\|_{l^\infty_{N} L^2_m}
\]
Thus \eqref{jf5} is proved, and the proof of the lemma is concluded.
\end{proof}

Proposition~\ref{propxg} follows.
\end{proof}

The proof of Proposition~\ref{epsmain} is also concluded.
\end{proof}

We now turn to the proof of Theorem~\ref{tlin} in
Section~\ref{sec:mainproof} which estimates forward solutions $\eps$
of the equation~\eqref{eqlin}, which we rewrite as
\[
P_0  \epsilon = f,
\qquad
P_0 = -\partial_{t}^{2}+\partial_{r}^{2}+\frac{1}{r}\partial_{r}+\frac{2}{r^{2}}
    (1-3Q(\lambda(t)r)^2 -6 Q(\lambda(t)r) v_{10})
\]
under the action of the invariant vector field $S=t\pr_t+r\pr_r$.
So far we have proved the bound \eqref{kbd0} for $\epsilon$.
In order to prove \eqref{kbd1} we write an equation for $S \epsilon$,
namely
\[
P_0 S \epsilon = S f + [P_0,S] \epsilon
\]
A direct computation yields
\[
[P_0,S] = 2P_0 - V, \qquad V = \frac{1}{r^2} S( 3Q(\lambda(t)r)^2 -6 Q(\lambda(t)r) v_{10})
\]
Hence
\[
P_0 S \epsilon = (S+2) f + V \epsilon
\]
A direct computation shows that
\[
|V| \lesssim \frac{1}{r^2} \frac{R^2}{(1+R^2)^2} \lesssim \lambda(t)^2
\]
Hence applying \eqref{kbd0} we obtain
\[
\| S \epsilon\|_{H^1_{N_1}} \lesssim \frac{1}{N_1} ( \| S
f\|_{L^2_{N_1}} + \| f\|_{L^2_{N_1}} + \| \lambda^2 \epsilon\|_{L^2_{N_1}})
\]
Then \eqref{kbd1} follows since
\[
\| \lambda^2 \epsilon\|_{L^2_{N_1}} \lesssim \|
\epsilon\|_{H^1_{N_1+2\beta+1}}
\]
We remark that this requires
\[
N_0 \geq N_1 + 2 \beta + 1
\]

The proof of \eqref{kbd2} is similar.

\end{document}